%% file: deep_rom.tex
\documentclass[final,10pt]{elsarticle}
\usepackage{graphicx}
\usepackage{rotating}
\usepackage{epstopdf, epsfig}
\usepackage[utf8]{inputenc} 
\usepackage[T1]{fontenc}    
\usepackage{hyperref}       
\usepackage{url}            
\usepackage{amsfonts}       
\usepackage{amsthm}       
\usepackage{mathtools}       
\usepackage{nicefrac}       
\usepackage{microtype}      
\usepackage{graphics}
\usepackage{pgfplots}
\usepackage{pgfplotstable}
\usepgfplotslibrary{fillbetween}
\usepackage{tikz}
\usepackage{bm}
\usepackage{xfrac}
\usepackage{float}
\usepackage{color}
\usepackage{subcaption}
\usepackage{algorithm}
\usepackage[noend]{algpseudocode}
\usepackage{mathtools}
\usepackage{amssymb}
\usepackage{amsmath}
\usepackage{standalone}
\usepackage{fullpage}
\usepackage[keeplastbox]{flushend}

\definecolor{green}{RGB}{0,102,0}

\usepackage{bm}

\newcommand{\RO}[1]{{\textcolor{black}{#1}}} 
\newcommand{\RT}[1]{{\textcolor{black}{#1}}} 
\newcommand{\OWN}[1]{{\textcolor{black}{#1}}} 

\renewcommand{\algorithmicrequire}{\textbf{Input: }}
\renewcommand{\algorithmicensure}{\textbf{Output: }}

\newcommand*{\range}{\text{Ran}}
\newcommand*{\pseudoinvSymb}{+}
\newcommand*{\pseudoinv}{^\pseudoinvSymb}
\newcommand*{\tran}{^{{T}}}

\newcommand*{\expect}{\mathbb E}
\newcommand*{\identity}{\bm{I}}			
\newcommand*{\zero}{\bm{0}}			
\newcommand*{\ones}{\bm{1}}			
\newcommand*{\GalerkinName}{manifold Galerkin}			
\newcommand*{\GalerkinNameCap}{Manifold Galerkin}			
\newcommand*{\GalerkinNameDeep}{Deep Galerkin}			
\newcommand*{\GalerkinNameDeepCap}{Deep Galerkin}			
\newcommand*{\LSPGName}{manifold LSPG}			
\newcommand*{\LSPGNameLong}{manifold least-squares Petrov--Galerkin (LSPG)}			
\newcommand*{\LSPGNameCap}{Manifold LSPG}			
\newcommand*{\LSPGNameDeep}{Deep LSPG}			
\newcommand*{\LSPGNameDeepCap}{Deep LSPG}			

\newcommand{\RR}[1]{\mathbb{R}^{#1}} 					
\newcommand{\RRstar}[1]{\mathbb{R}_\star^{#1}} 					
\newcommand{\RRplus}{\mathbb{R}_{+}} 					
\newcommand{\restrictionOp}{\bm{R}} 					
\newcommand{\prolongationOp}{\bm{P}} 					
\newcommand{\scalingOp}{\bm{S}} 					
\newcommand{\randperm}{\bm{P}} 					
\newcommand{\randpermTrain}{\bm{P}_\text{train}} 					
\newcommand{\scalingOpEl}[3]{s_{#1\cdots #2#3}} 					
\newcommand{\nspaceDim}[1]{n_{#1}} 					
\newcommand{\nDim}{d} 					
\newcommand{\nchannel}{n_\text{chan}} 					
\newcommand{\nsnap}{n_\text{snap}} 					
\newcommand{\nsubsnap}{n_t} 					

\newcommand{\dofhf}{N} 					
\newcommand{\nstate}{\dofhf} 					
\newcommand{\ndof}{\dofhf} 					
\newcommand{\dofrom}{p}					
\newcommand{\nstatered}{p}					
\newcommand{\intrinsicDim}{p^\star}					
\newcommand{\nseq}{{N_t}}					
\newcommand{\ntime}{{N_t}}					
\newcommand{\ntrain}{{n_{\text{train}}}}
\newcommand{\ntrainNN}{m}
\newcommand{\nvalidationNNfrac}{\omega}
\newcommand{\nvalNN}{\bar m}
\newcommand{\ninput}{n_{x}}
\newcommand{\nlatent}{n_{\hat{x}}}
\newcommand*{\nconv}{n_{\text{conv}}}
\newcommand*{\nfull}{n_{\text{full}}}
\newcommand*{\nlayer}{n_{\text{L}}}
\newcommand*{\nenclayer}{n_{\text{L}}}
\newcommand*{\ndeclayer}{\bar{n}_{\text{L}}}
\newcommand*{\dofenclayer}[1]{p_{#1}} 			
\newcommand*{\dofdeclayer}[1]{\bar{p}_{#1}} 			
\newcommand{\nepoch}{n_{\text{epoch}}}
\newcommand{\nbatch}{n_{\text{batch}}}

\newcommand*{\states}{\bm{x}}				
\newcommand*{\stateElementNo}{{x}}				
\newcommand*{\state}{\states}				
\newcommand*{\initstate}{\states^{0}}	
\newcommand*{\initrdstate}{\rdstate^{0}}	
\newcommand*{\stateRef}{\state_\text{ref}}	
\newcommand*{\xoned}{x}
\newcommand*{\xtwod}{\vec{\xoned}}
\newcommand*{\stateFOM}{\bm{x}}				
\newcommand*{\stateFOMArg}[1]{\stateFOM^{#1}}				
\newcommand*{\stateROM}{{\tilde{\bm{x}}}}				
\newcommand*{\stateROMArg}[1]{\stateROM^{#1}}				
\newcommand*{\stateROMGalArg}[1]{\stateRef + \rdbasisnl(\rdstateGal^{#1})}				
\newcommand*{\negstateROMGalArg}[1]{- \stateRef - \rdbasisnl(\rdstateGal^{#1})}				
\newcommand*{\stateROMLSPGArg}[1]{\stateRef + \rdbasisnl(\rdstateLSPG^{#1})}				
\newcommand*{\negstateROMLSPGArg}[1]{- \stateRef - \rdbasisnl(\rdstateLSPG^{#1})}				

\newcommand*{\gstate}{\bm{w}} 			

\newcommand*{\aprxstate}{\tilde{\states}}

\newcommand*{\eqvar}{w}
\newcommand*{\eqvarvec}{\bm{\eqvar}}
\newcommand*{\bigO}{O}

\newcommand*{\velo}{\bm{f}}

\newcommand*{\gvelo}{\hat{\bm{v}}}

\newcommand*{\unitVecRedArg}[1]{\hat{\bm{e}}_{#1}}

\newcommand*{\paramSymb}{{\mu}}			
\newcommand*{\param}{\bm{\paramSymb}}			
\newcommand*{\gparam}{\bm{\nu}}			
\newcommand*{\paramelem}[1]{\paramSymb_{#1}}
\newcommand*{\paramParProof}{}
\newcommand*{\gparamParProof}{}
\newcommand*{\paramSemiProof}{}
\newcommand*{\timesparamspaceProof}{}
\newcommand*{\gparamSemiProof}{}

\newcommand*{\paramtrain}{\param_{\text{train}}}
\newcommand*{\paramtest}{\param_{\text{test}}}
\newcommand*{\difftime}[1]{\dot{#1}}

\newcommand*{\paramspace}{\mathcal D}	
\newcommand*{\nparam}{{n_\paramSymb}}	
\newcommand*{\paramspacetrain}{\paramspace_{\text{train}}}

\newcommand*{\res}{\bm{r}}				
\newcommand*{\timestep}{\Delta t}				
\newcommand*{\resnFOM}{\bm{r}_\star^n}				
\newcommand*{\resnROM}{\tilde{\bm{r}}^n}				
\newcommand*{\resnFOMArg}[1]{\bm{r}_\star^n\left(#1\right)}				
\newcommand*{\resnROMArg}[1]{\tilde{\bm{r}}^n\left(#1\right)}				

\newcommand*{\rdstate}{\hat{\states}}		
\newcommand*{\rdstateGal}{\rdstate_G}		
\newcommand*{\rdstateLSPG}{\rdstate_P}		
\newcommand*{\grdstate}{\hat{\bm{\xi}}}
\newcommand*{\krdstate}{\hat{\bm{w}}}
\newcommand*{\manifold}{\mathcal S}
\newcommand*{\tangentSpace}[1]{T_{#1}\manifold }
\newcommand*{\defeq}{\vcentcolon =}
\newcommand*{\natno}[1]{\RO{\{1,\ldots,#1\}}}
\newcommand*{\innatno}[1]{\in\natno{#1}}
\newcommand*{\snapshotMat}{\bm{W}}
\newcommand*{\stateSnapshot}[1]{\bm{w}^{#1}}				
\newcommand*{\snapshotMatArg}[1]{\snapshotMat(#1)}
\newcommand*{\snapshotTrain}{\snapshotMat_{\text{train}}}
\newcommand*{\minibatchArg}[1]{\snapshotMat_{\text{mini},#1}}
\newcommand*{\minibatchVecArgs}[2]{\state_{\text{mini},#1}^{#2}}
\newcommand*{\iterToBatch}[1]{\mathcal I(#1)}
\newcommand*{\iterToBatchNo}{\mathcal I}
\newcommand*{\nminibatchArg}[1]{{\ntrainNN}^{#1}}
\newcommand*{\snapshotTrainArg}[1]{\state^{#1}_{\text{train}}}
\newcommand*{\snapshotVal}{\snapshotMat_{\text{val}}}
\newcommand*{\snapshotValArg}[1]{\state^{#1}_{\text{val}}}
\newcommand*{\gradientApproxArg}[1]{\widetilde {\nabla\objfunc}^{(#1)}}
\newcommand*{\gradient}{\nabla\objfunc}

\newcommand*{\rdbasisnl}{\bm{g}}			
\newcommand*{\rdbasisnlEnc}{\bar{\bm{g}}}			
\newcommand*{\jacrdbasisnl}{\bm{J}}				
\newcommand*{\jacrdresApprox}{\tilde{\bm{J}}^n_{\hat r_G}}				
\newcommand*{\jacrdresLSPGApprox}{\tilde{\bm{J}}^n_{\hat r_L}}				
\newcommand*{\searchDirNK}{{\bm{p}}^{n(k)}}				
\newcommand*{\rdstateNK}{{\rdstate}^{n(k)}}				
\newcommand*{\rdstateNKp}{{\rdstate}^{n(k+1)}}				
\newcommand*{\rdstateNZero}{{\rdstate}^{n(0)}}				
\newcommand*{\linesearchNK}{{\alpha}^{n(k)}}				
\newcommand*{\compactStiefel}[2]{ V_k(\RR{#2})}				
\newcommand*{\stateSnapshotRestrict}[1]{{\mathcal W}^{#1}}				
\newcommand*{\stateSnapshotRestrictEl}[4]{\mathcal W^{#1}_{#2\cdots#3#4}}				

\newcommand*{\rdbasis}{\bm{\Phi}}

\newcommand*{\rdbasisi}{\bm{\phi}}
\newcommand*{\rdtestbasis}{\bm{\Psi}}
\newcommand*{\testBasisGal}{\bm{\Psi}_G}
\newcommand*{\scalarConstant}{\gamma}

\newcommand*{\constantAnalysis}{\bm{A}^n(\grdstate\gparamSemiProof)}
\newcommand*{\constantAnalysisGal}{\bm{A}_G^n(\grdstate\gparamSemiProof)}
\newcommand*{\constantSymb}{\bm{A}^n}
\newcommand*{\constantSymbGal}{\bm{A}^n_G}

\newcommand*{\constantDefAnalysis}{[\testBasisTransposeAnalysis\rdbasis]^{-1}}
\newcommand*{\constantDefAnalysisGal}{[\testBasisTransposeAnalysisGal\rdbasis]^{-1}}
\newcommand*{\testBasis}{\rdtestbasis^{n}(\grdstate;\gparam)}
\newcommand*{\testBasisAnalysis}{\rdtestbasis^{n}(\grdstate\gparamSemiProof)}
\newcommand*{\testBasisAnalysisGalNo}{\rdtestbasis_G}
\newcommand*{\testBasisAnalysisGal}{\testBasisAnalysisGalNo(\grdstate\gparamSemiProof)}

\newcommand*{\testBasisTransposeAnalysis}{\testBasisAnalysis\tran}
\newcommand*{\testBasisTransposeAnalysisGal}{\testBasisGal(\grdstate)\tran}

\newcommand*{\testBasisAnalysisLongLinear}{\left( \alpha_0
\identity - \Delta t \beta_0 \frac{\partial \velo}{\partial \dstate}
\left(\stateRef\gparamParProof  + \rdbasis \grdstate, t^n\gparamSemiProof \right)
\right)\rdbasis}

\newcommand*{\rdres}{\hat{\bm{r}}_G}
\newcommand*{\rdresDisc}{\hat{\bm{r}}_{G,\text{disc}}}
\newcommand*{\rdresn}{\rdres^n}
\newcommand*{\rdresLSPG}{\hat{\bm{r}}_L}
\newcommand*{\rdresLSPGn}{\rdresLSPG^n}

\newcommand*{\dstate}{\bm{\xi}}
\newcommand*{\drdstate}{\hat{\bm{\xi}}}
\newcommand*{\drdstateArg}[1]{\hat{{\xi}}_{#1}}

\newcommand*{\minstateArg}[1]{\mathcal W_{#1}^{\min}}
\newcommand*{\maxstateArg}[1]{\mathcal W_{#1}^{\max}}

\newcommand{\norm}[1]{\|#1\|_2}
\newcommand{\abs}[1]{\left|#1\right|}
\newcommand{\lipschitz}{\kappa}
\newcommand{\hconstant}{h}
\newcommand{\gammaconstantj}{\gamma_j}

\newcommand*{\autoenc}{\bm{h}}			
\newcommand*{\encoder}{\autoenc_{\text{enc}}}
\newcommand*{\decoder}{\autoenc_{\text{dec}}}
\newcommand*{\enclayer}{\autoenc}			
\newcommand*{\declayer}{\bar{\autoenc}}			

\newcommand*{\gvec}{\bm{x}}				
\newcommand*{\gaprxvec}{\tilde{\gvec}}
\newcommand*{\latcode}{\hat{\gvec}}

\newcommand*{\lossfunc}{\mathcal{L}}

\newcommand*{\activation}{\bm{\phi}}		
\newcommand*{\encactivation}{\activation}		
\newcommand*{\decactivation}{\bar{\activation}}		
\newcommand*{\affinetrans}{h}
\newcommand*{\encaffinetrans}{\affinetrans}
\newcommand*{\decaffinetrans}{\bar{\affinetrans}}
\newcommand*{\nnparam}{\bm{\theta}}
\newcommand*{\nnparamOpt}{\nnparam^\star}
\newcommand*{\encparam}{\nnparam_{\text{enc}}}
\newcommand*{\decparam}{\nnparam_{\text{dec}}}
\newcommand*{\encparaml}[1]{\bm{\Theta}_{#1}}
\newcommand*{\decparaml}[1]{\bar{\bm{\Theta}}_{#1}}
\newcommand*{\metric}{{\bm\Theta}}
\newcommand*{\metricHalf}{\bm\Theta^{1/2}}
\newcommand*{\metricHalfT}{\bm\Theta^{T/2}}

\newcommand*{\objfunc}{\bm{J}}
\newcommand*{\empdist}{\hat p_{\text{data}}}

\newcommand{\learningrate}{\eta}
\newcommand{\learningrateVec}{\bm{\learningrate}}
\newcommand{\learningrateVecArg}[1]{\learningrateVec^{(#1)}}

\newcommand{\diffusivity}{\kappa}
\newcommand{\velocity}{\bm{v}}
\newcommand{\reaction}{q}
\newcommand{\reactionvec}{\bm{\reaction}}

\newcommand*{\spatialspace}{\Omega}
\newcommand{\fuel}{\text{H}_2}
\newcommand{\oxi}{\text{O}_2}
\newcommand{\product}{\text{H}_2\text{O}}

\newcommand*{\activationcost}{c_{\activation}}

\theoremstyle{definition}

\newtheorem{proposition}{Proposition}[section]
\newtheorem{theorem}{Theorem}[section]

\newtheorem{remark}[theorem]{Remark}

\begin{document}
\numberwithin{equation}{section}
\begin{frontmatter}
	\title{Model reduction of dynamical systems on nonlinear manifolds\\ using
	deep convolutional autoencoders}

\author[sandia]{Kookjin Lee\corref{sandiacor}}
\ead{koolee@sandia.gov}
\author[sandia]{Kevin T.\ Carlberg}
\ead{ktcarlb@sandia.gov}
\ead[url]{sandia.gov/~ktcarlb}

\address[sandia]{Sandia National Laboratories}
\cortext[sandiacor]{7011 East Ave, MS 9159, Livermore, CA 94550.}

\begin{abstract}
Nearly all model-reduction techniques project the governing equations onto a
	linear subspace of the original state space. Such subspaces are typically
	computed using methods such as balanced truncation, rational interpolation,
	the reduced-basis method, and (balanced) proper orthogonal decomposition
	(POD). Unfortunately, restricting the
	state to evolve in a linear subspace imposes a fundamental limitation
	to the accuracy of the resulting reduced-order model (ROM). In particular,
	linear-subspace ROMs can be expected to produce low-dimensional models with
	high accuracy only if the problem admits a fast decaying Kolmogorov
	$n$-width (e.g., diffusion-dominated problems). Unfortunately, many problems
	of interest exhibit a slowly decaying Kolmogorov $n$-width (e.g.,
	advection-dominated problems). To address this, we propose a novel framework
	for projecting dynamical systems onto nonlinear manifolds using
	minimum-residual formulations at the time-continuous and time-discrete
	levels; the former leads to \textit{\GalerkinName}\ projection, while the
	latter leads to \textit{\LSPGNameLong}\ projection.  We perform analyses that
	provide insight into the relationship between these proposed approaches and
	classical linear-subspace reduced-order models\OWN{; we also derive
	\textit{a posteriori} discrete-time error bounds for the proposed
	approaches}. In addition, we propose a
	computationally practical approach for computing the nonlinear 
	manifold, which is based on convolutional autoencoders from deep learning.
	Finally, we demonstrate the ability of the method to significantly
	outperform even the optimal linear-subspace ROM on benchmark
	advection-dominated problems, thereby demonstrating the method's ability to
	overcome the intrinsic $n$-width limitations of linear subspaces.
\end{abstract}

\end{frontmatter}

\section{Introduction}\label{sec:intro}

Physics-based modeling and simulation has become indispensable across many
applications in engineering and science, ranging from aircraft design to
monitoring national critical infrastructure. However, as simulation is playing
an increasingly important role in scientific discovery, decision making, and
design, greater demands are being placed on model fidelity.  Achieving high
fidelity often necessitates including fine spatiotemporal resolution in
computational models of the system of interest; this can lead to very
large-scale models whose simulations consume months on thousands of computing
cores. This computational burden precludes the integration of such
high-fidelity models in important scenarios that are \textit{real time} or
\textit{many query} in nature, as these scenarios require the (parameterized)
computational model to be simulated very rapidly (e.g., model predictive control) or
thousands of times (e.g., uncertainty propagation).


Projection-based reduced-order models (ROMs) provide one approach for
overcoming this computational burden.  These techniques comprise two stages:
an \textit{offline} stage and an \textit{online} stage. During the offline
stage, these methods perform computationally expensive training tasks (e.g., simulating
the high-fidelity model at several points in the parameter space) to compute a
representative low-dimensional `trial' subspace for the system state. Next,
during the inexpensive online stage, these methods rapidly compute approximate
solutions for different points in the parameter space via projection: they
compute solutions in the low-dimensional trial subspace by enforcing the
high-fidelity-model residual to be orthogonal to a low-dimensional test
subspace of the same dimension.

As suggested above, nearly all projection-based model-reduction approaches
employ \textit{linear trial subspaces}.  This includes the reduced-basis
technique \cite{prud2002reliable,rozza2007reduced} and proper orthogonal
decomposition (POD) \cite{holmes2012turbulence,carlberg2011low} for
parameterized stationary problems; balanced truncation
\cite{moore1981principal}, rational interpolation
\cite{baur2011interpolatory,GugercinIRKA}, and Craig--Bampton model reduction
\cite{craig1968coupling} for linear time invariant (LTI) systems; and Galerkin
projection \cite{holmes2012turbulence}, least-squares Petrov--Galerkin
projection \cite{carlberg2011efficient}, and other
Petrov--Galerkin projections \cite{willcox2002bmr} with (balanced) POD
\cite{holmes2012turbulence,lall2002sab,willcox2002bmr,rowley2005mrf} for
nonlinear dynamical systems.

The Kolmogorov $n$-width \cite{pinkus2012n} provides one way to quantify the
optimal linear trial subspace; it is defined as
$$
d_n(\mathcal M) \defeq \inf_{\mathcal S_n}\sup_{f\in\mathcal M}\inf_{g\in \mathcal
S_n}\|f-g\|,
$$
where the first infimum is taken over all $n$-dimensional subspaces of the
state space, and $\mathcal M$
denotes the manifold of solutions over all time and parameters. Assuming the
dynamical system has a unique trajectory for each parameter instance, 
the \textit{intrinsic solution-manifold dimensionality} is (at
most) equal to the
number of parameters plus one (time).
For problems that exhibit a fast decaying Kolmogorov $n$-width (e.g.,
diffusion-dominated problems), employing a linear trial subspace is
theoretically justifiable \cite{bachmayr2017kolmogorov,binev2011convergence}
and has enjoyed many successful demonstrations. Unfortunately, many
computational problems exhibit slowly decaying Kolmogorov $n$-width (e.g.,
advection-dominated problems). In such cases, the use of low-dimensional
linear trial subspaces often produces inaccurate results; the ROM
dimensionality must be significantly increased to yield acceptable accuracy
\cite{ohlberger2016reduced}. Indeed, the Kolmogorov $n$-width with $n$ equal
to the intrinsic solution-manifold dimensionality is often quite large for
such problems.


Several approaches have been pursued to address this $n$-width limitation of
linear trial subspaces.  One set of approaches transforms the trial basis to improve its approximation properties for
advection-dominated problems.  Such methods include separating transport
dynamics via `freezing' \cite{ohlberger2013nonlinear}, applying
a coordinate transformation to the trial basis
\OWN{\cite{iollo2014advection,nair2017transported,cagniart2019model}}, shifting the POD basis 
\cite{reiss2018shifted}, transforming the physical domain of the
snapshots \cite{welper2017interpolation,welper2017h}, constructing the trial
basis on a Lagrangian formulation of the governing equations
\cite{mojgani2017lagrangian}, and using Lax pairs of the Schr\"{o}dinger
operator to construct a time-evolving trial basis
\cite{gerbeau2014approximated}. Other approaches pursue the use of multiple
linear subspaces instead of employing a single global linear trial subspace;
these local subspaces can be tailored to different regions of the time domain
\cite{drohmann2011adaptive,dihlmann2011model}, physical domain
\cite{taddei2015reduced}, or state space
\cite{amsallem2012nonlinear,peherstorfer2015online}.
Ref.~\cite{carlberg2015adaptive} aims to overcome the limitations of using a
linear trial subspace by providing online-adaptive $h$-refinement mechanism that
constructs a hierarchical sequence of linear subspaces that converges to
the original state space. However, all of these methods attempt to construct,
manually transform, or refine a linear basis to be locally accurate; they do not
consider nonlinear trial manifolds of more general structure. Further, many of
these approaches rely on substantial additional knowledge about the problem,
such as the particular advection phenomena governing basis shifting.

This work aims to address the fundamental $n$-width deficiency of linear trial
subspaces. However, in contrast to the above methods, we
pursue an approach that is both \textit{more general} (i.e., it should not be limited to
piecewise linear manifolds or mode transformations) and only requires the
\textit{same
snapshot data} as typical POD-based approaches (e.g., it should require no
special knowledge about any particular advection phenomena).  To accomplish
this objective, we propose an approach that (1) performs optimal projection of
dynamical systems onto arbitrary nonlinear trial manifolds (during the online
stage), and (2) computes this nonlinear trial manifold from snapshot data
alone (during the offline stage).  

For the first component, we perform optimal projection onto arbitrary (continuously
differentiable) nonlinear trial manifolds by applying minimum-residual
formulations at the time-continuous (ODE) and time-discrete (O$\Delta$E)
levels. The time-continuous formulation leads to \textit{\GalerkinName}\
projection, which can be interpreted as performing orthogonal projection of
the velocity onto the tangent space of the trial manifold. The time-discrete
formulation leads to \textit{\LSPGNameLong}\ projection, which can also be
straightforwardly extended to stationary (i.e., steady-state) problems. We also perform analyses
that illustrate the relationship between these manifold ROMs and classical
linear-subspace ROMs. \GalerkinNameCap\ and \LSPGName\ projection require the
trial manifold \RO{to} be characterized as a (generally nonlinear) mapping from the low-dimensional reduced state to
the high-dimensional state; the mapping from the high-dimensional state to the
low-dimensional state is not required.

The second component aims to compute a nonlinear trial manifold from snapshot
data alone. Many machine-learning methods exist to perform nonlinear
dimensionality reduction. However, many of these methods do not provide the
required mapping from the low-dimensional embedding to the high-dimensional input; examples include
Isomap \cite{tenenbaum2000global}, locally linear embedding (LLE)
\cite{roweis2000nonlinear}, Hessian eigenmaps \cite{donoho2003hessian},
spectral embedding \cite{belkin2003laplacian}, and t-SNE
\cite{maaten2008visualizing}. Methods that \textit{do} provide this required mapping
include self-organizing maps \cite{kohonen1982self}, generative topographic
mapping \cite{bishop1997gtm}, kernel principal component analysis (PCA)
\cite{scholkopf1998nonlinear}, Gaussian process latent variable model
\cite{lawrence2004gaussian}, diffeomorphic dimensionality reduction
\cite{walder2009diffeomorphic}, and autoencoders \cite{hinton2006reducing}.
In principle, \GalerkinName\ and \LSPGName\ projection could be applied with
manifolds constructed by any of the methods in the latter category.  However,
this study restricts focus to autoencoders---more specifically deep
convolutional autoencoders---due to their expressiveness and scalability, as well as the availability of
high-performance software tools for their construction.

Autoencoders (also known as auto-associators \cite{demers1993non}) comprise a
specific type of feedforward neural network that aim to learn the identity
mapping: they attempt to copy the input to an accurate approximation of
itself.  Learning the identity mapping is not a particularly useful task unless,
however, it associates with a dimensionality-reduction procedure comprising
data compression and subsequent recovery. This is precisely what autoencoders
accomplish by employing a neural-network architecture consisting of two parts:
an \textit{encoder} that provides a nonlinear mapping from the
high-dimensional input to a low-dimensional embedding, and a \textit{decoder}
that provides a nonlinear mapping from the low-dimensional embedding to an
approximation of the high-dimensional input. Convolutional autoencoders are a
specific type of autoencoder that employ convolutional layers, which have been
shown to be effective for extracting representative features in images
\cite{masci2011stacked}.
Inspired by the analogy \OWN{between} images and spatially
distributed dynamical-system states (e.g., when the dynamical
system corresponds to the spatial discretization of a
partial-differential-equations model), we propose a specific deep
convolutional autoencoder architecture tailored to dynamical systems with
states that are spatially distributed.  Critically, training this autoencoder
requires only the same snapshot data as POD; no additional problem-specific
information is needed.

In summary, new contributions of this work include:
\begin{enumerate} 
	\item \GalerkinNameCap\ (Section \ref{sec:Galerkin}) and \LSPGName\
		(Section \ref{sec:LSPG}) projection techniques, which project
	the dynamical-system model onto arbitrary continuously-differentiable manifolds.
		We equip these methods with 
	 \begin{enumerate} 
		 \item the ability to exactly satisfy the initial condition (Remark
			 \ref{rem:IC}), and
			  \item quasi-Newton solvers (Section
		\ref{sec:quasi}) to solve the system of algebraic equations
			 arising from implicit time integration.
					 \end{enumerate}
	\item Analysis (Section \ref{sec:analysis}), which includes
\begin{enumerate} 
\item demonstrating that employing an affine trial manifold recovers classical
	linear-subspace Galerkin and LSPG projection  (Proposition
		\ref{thm:affine}),
\item sufficient conditions for commutativity of time discretization and
	\GalerkinName\ projection (Theorem \ref{prop:comm})\OWN{,}
\item conditions under which \GalerkinName\ and \LSPGName\ projection are
	equivalent (Theorem \ref{thm:equivalence})\OWN{, and}
\item \OWN{\textit{a posteriori} discrete-time error bounds for both
the manifold Galerkin and manifold LSPG projection methods
		(Theorem \ref{thm:error_bound}).}
\end{enumerate}
\item A novel convolutional autoencoder architecture tailored to spatially
	distributed dynamical-system states (Section \ref{sec:autoencoder}) with
		accompanying offline training algorithm that requires only the same
		snapshot data as POD (Section \ref{sec:offline}).
	\item  Numerical experiments on advection-dominated benchmark problems
		(Section \ref{sec:numexp}). These experiments illustrate the ability of
		the method to outperform even the projection of the solution onto the
		optimal linear subspace; further, the proposed method is close to achieving the
		optimal performance of any nonlinear-manifold method. This demonstrates
		the method's ability to overcome the intrinsic $n$-width limitations of
		linear trial subspaces.  
\end{enumerate}
We note that the methodology is applicable to both linear and nonlinear
dynamical systems.

To the best of our knowledge, Refs.\
\cite{hartman2017deep,kashima2016nonlinear} comprise the only attempts to
incorporate an autoencoder within a projection-based ROM.  These methods seek
solutions in the nonlinear trial manifold provided by an autoencoder; however,
these methods reduce the number of equations by applying the encoder to the
velocity.  Unfortunately, as we discuss in Remark \ref{rem:badGal}, this
approach is \textit{kinematically inconsistent}, as the velocity resides in
the tangent space to the manifold, not the manifold itself. Thus, encoding the
velocity can produce significant approximation errors. Instead, the proposed
\GalerkinName\ and \LSPGName\ projection methods produce approximations that
associate with minimum-residual formulations and adhere to the kinematics
imposed by the trial manifold.

Relatedly, Ref.~\cite{gu2011model} proposes a general framework for projection
of dynamical systems onto nonlinear manifolds. However, the proposed method
constructs a piecewise linear trial manifold by generating local linear
subspaces and concatenating those subspaces. Then, the method projects the
residual of the governing equations onto a nonlinear test manifold that is
also piecewise linear; this is referred to as a `piecewise linear projection
function'. Thus, the approach is limited to piecewise-linear manifolds, and 
the resulting approximation does not associate with any optimality property.

We also note that autoencoders have been applied to various non-intrusive
model-reduction methods that are purely data driven in nature and are not
based on a projection process. Examples include \OWN{Ref.~\cite{milano2002neural}, 
which constructs a nonlinear model of wall turbulence using an autoencoder;}
Ref.~\cite{gonzalez18}, which
applies an autoencoder to compress the state, followed by a recurrent neural
network (RNN) \cite{rumelhart1986learning} to learn the dynamics;
Refs.~\cite{takeishi2017learning,otto2017linearly,lusch2017deep,morton2018deep},
which apply autoencoders to learn approximate invariant subspaces of the
Koopman operator; and Ref.~\cite{carlberg2018recovering}, which applies
hierarchical dimensionality reduction comprising autoencoders and PCA followed
by dynamics learning to recover missing CFD data.

The paper is organized as follows. Section \ref{sec:fom} describes the
full-order model, which corresponds to a parameterized system of (linear or
nonlinear) ordinary
differential equations. Section \ref{sec:rom} describes model reduction on
nonlinear manifolds, including the mathematical characterization of the
nonlinear trial manifold (Section \ref{sec:nonlinearTrialManifold}),
\GalerkinName\ projection (Section \ref{sec:Galerkin}), \LSPGName\ projection
(Section \ref{sec:LSPG}), and associated quasi-Newton methods to solve the
system of algebraic equations arising at each time instance in the case of
implicit time integration (Section \ref{sec:quasi}).  
\OWN{Section \ref{sec:analysis} provides the aforementioned analysis results.}
Section
\ref{sec:manifold} describes a practical approach for constructing the
nonlinear trial manifold using deep convolutional autoencoders, including a
brief description of autoencoders (Section \ref{sec:autoenc}), the proposed
autoencoder architecture applicable to spatially distributed states (Section
\ref{sec:autoencoder}), and the way in which the proposed autoencoder can be
used to satisfy the initial condition (Section \ref{sec:initialCondition}).
When the \GalerkinName\ and \LSPGName\ ROMs employ this choice of
decoder, we refer to them as \textit{\GalerkinNameDeepCap}\ and
\textit{\LSPGNameDeepCap}\ ROMs, respectively.  Section \ref{sec:offline}
describes offline training, which entails snapshot-based data collection
(Section \ref{sec:snapshots}),  data standardization (Section
\ref{sec:standardization}), and autoencoder training (Section
\ref{sec:grad_learning}); Algorithm \ref{alg:ae_train} summarizes the offline
training stage.  Section \ref{sec:numexp} assesses the
performance of the proposed \GalerkinNameDeep\ and \LSPGNameDeep\ ROMs compared
to (linear-subspace) POD--Galerkin and POD--LSPG ROMs on two advection-dominated benchmark
problems. Finally, Section \ref{sec:conc} concludes the paper.

\section{Full-order model}\label{sec:fom}
This work considers the full-order model (FOM) to correspond to a
dynamical system expressed as a parameterized
system of ordinary differential equations
(ODEs)
\begin{equation}\label{eq:govern_eq}
	\difftime{\states} = \velo(\states,t;\param), \qquad \states(0;\param) = \initstate(\param),
\end{equation}
where $t \in [0, T]$ denotes time with final time $T \in \mathbb{R}_{+}$, and
$\states: [0, T] \times \paramspace  \rightarrow \mathbb{R}^{\dofhf}$  denotes
the time-dependent, parameterized state implicitly defined as the solution to
problem \eqref{eq:govern_eq} given the parameters $\param \in \paramspace$.
Here, $\paramspace\subseteq\RR{\nparam}$ denotes the parameter space,
$\initstate : \paramspace \rightarrow \mathbb{R}^{\dofhf}$ denotes the
parameterized initial condition, and $\velo: \mathbb{R}^{\dofhf} \times [0,
T] \times \paramspace \rightarrow \mathbb{R}^{\dofhf}$ denotes the velocity,
which may be linear or nonlinear in its first argument. We denote differentiation of a
variable $x$ with
respect to time by $\dot x$. Such dynamical systems
may arise from the semidiscretization of a partial-differential-equations (PDE)
model, for example. We refer to Eq.~\eqref{eq:govern_eq} as the FOM ODE.

Numerically solving the FOM ODE~\eqref{eq:govern_eq} requires application of a
time-discretization method. For simplicity, this work restricts attention to
linear multistep methods; see Ref.~\cite{carlbergGalDiscOpt} for analysis of
both (linear-subspace) Galerkin and LSPG reduced-order models with 
Runge--Kutta schemes. A linear $k$-step method applied to numerically solve
the FOM ODE \eqref{eq:govern_eq} leads to solving the system of 
algebraic equations
\begin{align}\label{eq:fomODeltaE}
\res^{n}(\state^{n};\param) = \zero,\quad n=1,\ldots,\nseq,
\end{align}
where the time-discrete residual $\res^{n}:
\RR{\nstate}\times\paramspace\rightarrow\RR{\nstate}$ is defined as
\begin{align}\label{eq:disc_res}
\res^{n}:(\dstate;\gparam) \mapsto \alpha_0 \dstate - \Delta t \beta_0
	\velo(\dstate,t^n;\gparam) + \sum_{j=1}^{k} \alpha_j \states^{n-j}
	- \Delta t \sum_{j=1}^{k} \beta_j \velo (\states^{n-j}, t^{n-j}; \gparam).
\end{align}
Here, 
$\timestep \in \mathbb{R}_{+}$ denotes the time step, $\states^{k}$
denotes the numerical approximation to $\states(k\timestep; \param)$, and the
coefficients $\alpha_j$ and $\beta_j$, $j=0,\ldots, k$ with
$\sum_{j=0}^k\alpha_j=0$ define a particular linear multistep
scheme. These methods are implicit if $\beta_0\neq 0$. For notational simplicity, we have assumed a uniform time
step $\timestep$ and a fixed number of steps $k$ for each time instance;
the \RT{extensions} to nonuniform grids and a non-constant value of $k$ are
straightforward. 
We refer to Eq.~\eqref{eq:fomODeltaE} as the FOM O$\Delta$E.
\begin{remark}[Intrinsic solution-manifold
	dimensionality]\label{rem:intrinsic}
	Assuming the initial value problem \eqref{eq:fomODeltaE} has a unique
	solution for each parameter instance $\param\in\paramspace$, the
	intrinsic dimensionality of the solution manifold
	$\{\state(t;\param)\;|\;t\in[0,T],\ \param\in\paramspace\}$ is (at most)
	$\intrinsicDim=\nparam+1$, as the mapping $(t;\param)\mapsto \state$
	is unique in this case. This provides a practical lower bound on the
	dimension of a nonlinear trial manifold for exactly representing the
	dynamical-system state.
\end{remark}

\section{Model reduction on nonlinear manifolds}\label{sec:rom}

This section proposes two classes of residual-minimizing ROMs on nonlinear
manifolds. The first minimizes the (time-continuous) FOM ODE residual and is
analogous to classical Galerkin projection, while the second minimizes the
(time-discrete) FOM O$\Delta$E residual and is analogous to least-squares
Petrov--Galerkin (LSPG) projection
\cite{carlberg2011efficient,carlberg2013gnat,carlbergGalDiscOpt}.  Section
\ref{sec:nonlinearTrialManifold} introduces the notion of a nonlinear trial
manifold, Section \ref{sec:Galerkin} describes the \GalerkinName\ ROM resulting from
time-continuous residual minimization, and Section \ref{sec:LSPG} describes
the \LSPGName\ ROM resulting from time-discrete residual minimization. 

\subsection{Nonlinear trial manifold}\label{sec:nonlinearTrialManifold}
We begin by prescribing the subset of the original state space $\RR{\ndof}$ on
which the ROM techniques will seek approximate solutions. Rather than introducing a
classical affine
\textit{trial subspace} for this purpose, we instead introduce a continuously
differentiable nonlinear \textit{trial
manifold}. In particular, we seek approximate solutions $\aprxstate
\approx \states$ of the form
\begin{equation}\label{eq:state_approx}
	\aprxstate(t; \param) = \stateRef(\param)+ \rdbasisnl(\rdstate(t;\param)),
\end{equation}
where
$\aprxstate:\RRplus\times\paramspace\rightarrow \RO{\stateRef(\param)+} \manifold$ and
$\manifold \defeq\{
\rdbasisnl(\grdstate)\,|\,
\grdstate\in\RR{\dofrom}\}$
denotes the nonlinear trial manifold \OWN{from the extrinsic view}.
Here, 
$\stateRef:\paramspace\rightarrow\RR{\ndof}$ denotes a parameterized  reference
state and
$\rdbasisnl:\drdstate\mapsto\rdbasisnl(\drdstate)$
with
$\rdbasisnl: \mathbb{R}^{\dofrom} \rightarrow \mathbb{R}^{\dofhf}$ and
$\dofrom \leq \dofhf$ denotes the \OWN{parameterization function, which we
refer to in this work as the} \textit{decoder}, which comprises a
nonlinear mapping from low-dimensional
generalized coordinates $\rdstate:\RRplus\times\paramspace\rightarrow\RR{\dofrom}$ to the high-dimensional state approximation.
From application of the chain rule, the approximated velocity is then
\begin{equation}\label{eq:vel_approx}
	\difftime{\aprxstate}(t; \param) = \jacrdbasisnl(\rdstate(t;\param))
	\difftime{\rdstate}(t;\param),
\end{equation}
where $\jacrdbasisnl:\grdstate\mapsto \frac{d\rdbasisnl}{d
\drdstate}(\grdstate)$ with
$\jacrdbasisnl:\RR{\nstatered}\rightarrow\mathbb{R}^{\dofhf \times \dofrom}$
denotes the Jacobian of the decoder. This Jacobian defines the tangent space
to the manifold 
\begin{equation}\label{eq:tan_space}
	\tangentSpace{\rdstate} \equiv \range(\jacrdbasisnl(\rdstate)),
\end{equation}
where
$\range(\bm{A})$ denotes the range of matrix $\bm{A}$. From Remark
\ref{rem:intrinsic}, we observe that the nonlinear trial manifold dimension
must be greater than or equal to the intrinsic solution-manifold dimension,
i.e., $\dofrom\geq\intrinsicDim$, to be able to exactly \RO{represent} the
state.

\begin{remark}[Initial-condition satisfaction]\label{rem:IC}
Satisfaction of the initial conditions requires
	the initial generalized coordinates
	$\rdstate(0;\param)=\initrdstate(\param)$	to satisfy
	$\rdbasisnl(\initrdstate(\param))= \initstate(\param)$. This can be achieved
	for any choice of $\initrdstate(\param)$ provided the reference state is set
	to 
	\begin{equation}\label{eq:refStateIC}
	\stateRef(\param)= \initstate(\param)-\rdbasisnl(\initrdstate(\param)).
	\end{equation}
	However, doing so must ensure that the decoder can accurately
	represent deviations from this reference state \eqref{eq:refStateIC}. In the
	context of the proposed autoencoder-based trial manifold, Section
	\ref{sec:initialCondition} describes a strategy for computing the initial
	generalized coordinates $\initrdstate(\param)$ and defining training data
	for manifold construction such that the decoder accomplishes this.
\end{remark}

\begin{remark}[Linear trial subspace]
	The classical linear (or more precisely \textit{affine}) trial subspace corresponds to a specific case of
a nonlinear trial manifold; this occurs when the decoder is linear. In this
case, the decoder can be expressed as
	$\rdbasisnl:\drdstate\mapsto \rdbasis\drdstate$, where $\rdbasis \in
\RRstar{\dofhf \times \dofrom}$ denotes trial-basis matrix and
$\RRstar{n\times m}$ denotes the set of full-column-rank $n\times m$ matrices
(the non-compact Stiefel manifold). In this case, the approximated state and
velocity can be
expressed as
\begin{equation}
\aprxstate(t; \param) = \initstate(\param) + \rdbasis
	\rdstate(t;\param)\quad\text{and}\quad
	\difftime{\aprxstate}(t; \param) = \rdbasis
	\difftime{\rdstate}(t;\param),
\end{equation}
respectively,
such that the trial manifold is affine, i.e.,
$\manifold = \RO{\range(\rdbasis)}$, and the decoder
Jacobian is the constant matrix
$\jacrdbasisnl(\rdstate(t;\param)) = \jacrdbasisnl= \rdbasis$.
Note that this approach can enforce the initial condition by setting the initial
	generalized coordinates 
	$\initrdstate(\param)$ to zero and subsequently setting
	$\stateRef(\param)=\initstate(\param)$.
Common choices for computing the reduced-basis matrix $\rdbasis$ when the velocity is
linear in its first argument include balanced truncation \cite{moore1981principal},
rational interpolation \cite{baur2011interpolatory,GugercinIRKA}, the reduced-basis method
\cite{prud2002reliable,rozza2007reduced}, Rayleigh--Ritz eigenvectors
\cite{craig1968coupling}; common choices when the velocity is nonlinear
include POD \cite{POD} and balanced POD \cite{lall2002sab,willcox2002bmr,rowley2005mrf}.
\end{remark}

\subsection{\GalerkinNameCap\ ROM: time-continuous residual minimization}\label{sec:Galerkin}

We now derive the ROM corresponding to time-continuous residual minimization.
To do so, define the FOM ODE residual as
\begin{equation*}
\res:(\bm{v}, \dstate, \tau; \gparam) \mapsto \bm{v} - \velo(\dstate,\tau;\gparam)
\end{equation*}
with
$\res:\RR{\ndof}\times\RR{\ndof}\times\RRplus\times\paramspace\rightarrow\RR{\ndof}$.
It can be easily verified that the FOM ODE~\eqref{eq:govern_eq} is equivalent to 
\begin{equation}\label{eq:govern_eqRes}
	\res\left(\difftime{\states}, \states,t;\param\right)=\zero, \qquad \states(0;\param) = \initstate(\param).
\end{equation}

To derive the time-continuous residual-minimizing ROM, we substitute the
approximated state $\states \leftarrow \aprxstate$ defined in
Eq.~\eqref{eq:state_approx} and velocity $\difftime{\states} \leftarrow
\difftime{\aprxstate}$ defined in Eq.~\eqref{eq:vel_approx} into the FOM ODE
\eqref{eq:govern_eqRes} and minimize the $\ell^2$-norm of the residual,
i.e.,
\begin{equation}\label{eq:cont_opt_problem}
	\difftime{\rdstate}(t;\param)  =
	\underset{\gvelo\in\RR{\nstatered}}{\arg\min} \left
	\| \res\left(\jacrdbasisnl(\rdstate(t;\param)) \gvelo,
	\stateRef(\param) + \rdbasisnl(\rdstate(t;\param)), t;\param \right)
	\right\|_2^2
\end{equation}
with initial condition $\rdstate(0;\param)=\initrdstate(\param)$.
If the
Jacobian $\jacrdbasisnl(\rdstate(t;\param))$ has full column rank, then
Problem \eqref{eq:cont_opt_problem} is convex and has a unique solution that 
leads to the \textit{\GalerkinName} ODE
\begin{align} \label{eq:cont_rom}
	\difftime{\rdstate}(t;\param) =
	\jacrdbasisnl(\rdstate(t;\param))\pseudoinv
	\velo\left(\stateRef(\param)+\rdbasisnl(\rdstate(t;\param)), t; \param
	\right),\quad\rdstate(0;\param)=\initrdstate(\param),
\end{align}
where the superscript $\pseudoinvSymb$ denotes the Moore--Penrose
pseudoinverse.  

We note that although the state evolves on the nonlinear
manifold $\manifold$, the generalized coordinates evolve in the Euclidean
space $\RR{\nstatered}$, which facilitates time integration of the
\GalerkinName\ ODE \eqref{eq:cont_rom}.
Indeed, this can be accomplished using any
time integrator. If a linear multistep scheme is employed, then the
\GalerkinName\ O$\Delta$E corresponds to 
\begin{equation}\label{eq:GalerkinODeltaE}
\rdresn(\rdstate^n; \param) = \zero,
\end{equation}
where the \GalerkinName\ O$\Delta$E residual is
\begin{align}\label{eq:time_cont_res}
	\begin{split}
		\rdres^n:(\grdstate; \gparam) \mapsto &\alpha_0 \grdstate - \Delta t
		\beta_0 \jacrdbasisnl(\grdstate)\pseudoinv \velo(
		\stateRef(\gparam)+\rdbasisnl(\grdstate),t^n;\gparam) +\\
		&\sum_{j=1}^{k} \alpha_j \rdstate^{n-j} - \Delta t \sum_{j=1}^{k} \beta_j
		\jacrdbasisnl(\rdstate^{n-j})\pseudoinv \velo(\stateRef(\gparam)+\rdbasisnl(\rdstate^{n-j}),t^{n-j};\gparam).
	\end{split}
\end{align}
We note that the \GalerkinName\ O$\Delta$E is nonlinear if either the velocity
is nonlinear in its first argument or if the trial manifold is nonlinear.
\begin{remark}[\GalerkinNameCap\ projection as orthogonal projection onto the
	tangent space]\label{rem:tanSpace}
	Eq.~\eqref{eq:cont_opt_problem} can be written equivalently as 
\begin{equation}\label{eq:cont_opt_problem_vel}
	\difftime{\aprxstate}(t; \param)  =
	\underset{\bm{v}\in\tangentSpace{\rdstate(t;\param)}}{\arg\min} \left
	\| \bm{v} - \velo
	\left(
	\stateRef(\param) + \rdbasisnl(\rdstate(t;\param)), t;\param \right)
	\right\|_2^2,
\end{equation}
	where the tangent space is defined in Eq.~\eqref{eq:tan_space}.
	The solution to this minimization problem is $\difftime{\aprxstate}(t; \param)=
	\jacrdbasisnl(\rdstate(t;\param))\jacrdbasisnl(\rdstate(t;\param))\pseudoinv
	\velo \left( \stateRef(\param) + \rdbasisnl(\rdstate(t;\param)), t;\param
	\right) $. Thus, \GalerkinNameCap\ projection can be interpreted as
	computing the orthogonal (i.e.,
	$\ell^2$-optimal) projection  of the velocity onto
	the tangent space 
	$\tangentSpace{\rdstate(t;\param)}$
	of the trial manifold.
\end{remark}
\begin{remark}[Weighted $\ell^2$-norms for \GalerkinName\
	projection]\label{rem:normGal}
	We note that the minimization problem \eqref{eq:cont_opt_problem} could be
	posed in other norms. For example, if we consider residual minimization in
	the weighted $\ell^2$-norm
	$\|\cdot\|_\metric$ satisfying
	$\|\bm{x}\|_\metric\equiv(\bm{x},\bm{x})_\metric$ with
	$(\bm{x},\bm{y})_\metric\equiv \sqrt{\bm{x}^T\metric\bm{y}}$ 
	and $\metric\equiv\metricHalfT\metricHalf$, then Problem
	\eqref{eq:cont_opt_problem} becomes
\begin{equation}\label{eq:cont_opt_problem_metric}
	\difftime{\rdstate}(t;\param)  =
	\underset{\gvelo\in\RR{\nstatered}}{\arg\min} \left
	\| \res\left(\jacrdbasisnl(\rdstate(t;\param)) \gvelo,
	\stateRef(\param) + \rdbasisnl(\rdstate(t;\param)), t;\param \right)
	\right\|_\metric^2,
\end{equation}
leading to the \GalerkinName\ ODE
\begin{align} \label{eq:cont_rom_metric}
	\difftime{\rdstate}(t;\param) =
	[\metricHalf\jacrdbasisnl(\rdstate(t;\param))]\pseudoinv
	\metricHalf\velo\left(\stateRef(\param)+\rdbasisnl(\rdstate(t;\param)), t; \param
	\right),\quad\rdstate(0;\param)=\initrdstate(\param),
\end{align}
and \GalerkinName\ O$\Delta$E residual
\begin{align}\label{eq:GalerkinODeltaE_metric}
\rdres^n:(\grdstate; \gparam) \mapsto &\alpha_0 \grdstate - \Delta t
		\beta_0 [\metricHalf\jacrdbasisnl(\grdstate)]\pseudoinv \metricHalf\velo(
		\stateRef(\gparam)+\rdbasisnl(\grdstate),t^n;\gparam) +\\
		&\sum_{j=1}^{k} \alpha_j \rdstate^{n-j} - \Delta t \sum_{j=1}^{k} \beta_j
		[\metricHalf\jacrdbasisnl(\rdstate^{n-j})]\pseudoinv \metricHalf\velo(\stateRef(\gparam)+\rdbasisnl(\rdstate^{n-j}),t^{n-j};\gparam).
\end{align}
For notational simplicity, this work restricts attention to the $\ell^2$-norm;
future work will consider such weighted norms, e.g., to enable
	hyper-reduction.
\end{remark}

\begin{remark}[Alternative Galerkin projection]\label{rem:badGal}
	Refs.~\cite{hartman2017deep,kashima2016nonlinear} provide an alternative way
	to perform Galerkin projection on nonlinear manifolds constructed from
	autoencoders. In contrast to the proposed \GalerkinName\ ROM (and
	\LSPGName\ ROM introduced in Section \ref{sec:LSPG} below), these methods additionally
	require the existence of an \textit{encoder}
	$\rdbasisnlEnc: \mathbb{R}^{\dofhf}\rightarrow\mathbb{R}^{\dofrom}$ that
	satisfies $\rdbasisnl\circ \rdbasisnlEnc:\state\mapsto\tilde\state$ with
	$\tilde\state\approx\state$, which is trained on state-snapshot data. These
	methods 
	then form the low-dimensional system of ODEs
\begin{align} \label{eq:cont_rom_dumb}
	\difftime{\rdstate}(t;\param) =
	\rdbasisnlEnc\left(
	\velo\left(\stateRef(\param)+\rdbasisnl(\rdstate(t;\param)), t; \param
	\right)\right),\quad\rdstate(0;\param)=\rdbasisnlEnc(\initstate(\param)).
\end{align}
One problem with this approach is that it implicitly assumes that the velocity
$\velo$ can be represented on the manifold used to represent the state.
	This is \textit{kinematically inconsistent}, as the velocity resides in the
	tangent space to this manifold as derived in
	Eq.~\eqref{eq:vel_approx}, i.e.,
	$\difftime{\aprxstate}(t; \param)	\in
	\tangentSpace{\rdstate(t;\param)}$; the tangent space and the
	manifold coincide if and only if the trial manifold associates with a linear
	trial subspace.
	Thus, encoding the velocity is likely to produce
	a poor approximation of the reduced-state velocity 
	$\difftime{\rdstate}(t;\param)$.
	Instead, the proposed
	\GalerkinName\ ROM performs orthogonal projection of the velocity onto the
	tangent space (see Remark \ref{rem:tanSpace}).  We also note the approaches
	proposed in Refs.~\cite{hartman2017deep,kashima2016nonlinear} compute the
	initial state by encoding the initial condition, which is not guaranteed to
	satisfy the initial conditions as discussed in Remark \ref{rem:IC}.
\end{remark}

\subsection{\LSPGNameCap\ ROM: time-discrete residual minimization}\label{sec:LSPG}
Analogously to the previous section, we now derive the ROM corresponding to
time-discrete residual minimization. To do so, we simply substitute the
approximated state $\states \leftarrow \aprxstate$ defined in
Eq.~\eqref{eq:state_approx} into the FOM
O$\Delta$E \eqref{eq:fomODeltaE}
and minimize the $\ell^2$-norm of the residual, i.e.,
\begin{equation}\label{eq:disc_opt_problem}
	\rdstate^n(\param)  =
	\underset{\grdstate\in\RR{\nstatered}}{\arg\min} \left\|
	\res^{n}\left(
	\stateRef(\param)+\rdbasisnl(\grdstate);\param)\right)\right\|_2^2,
\end{equation}
which is solved sequentially for $n=1,\ldots,\nseq$ with initial condition
$\rdstate(0;\param)=\initrdstate(\param)$.
Necessary optimality conditions for Problem~\eqref{eq:disc_opt_problem}
correspond to stationarity of the objective function, i.e., the solution
$\rdstate^n$ satisfies the \LSPGName\ O$\Delta$E
\begin{align}\label{eq:lspg_proj}
	\rdtestbasis^{n}(\rdstate^{n};\param)\tran
	\res^{n}\left(
	\stateRef(\param)+\rdbasisnl(\rdstate^n);\param\right) = \zero,
\end{align}
where the test basis matrix
$\rdtestbasis^{n}:\RR{\nstatered}\times\paramspace\rightarrow\RR{\nstate\times\nstatered}$
is defined as
\begin{align}\label{eq:testBasisMatDef}
\rdtestbasis^{n}:(\grdstate;\gparam) &\mapsto \frac{\partial \res^n}{\partial
	\dstate}(
	\stateRef(\gparam)+\rdbasisnl(\grdstate);\gparam)\jacrdbasisnl(\grdstate)= \left( \alpha_0
	\identity - \Delta t \beta_0 \frac{\partial \velo}{\partial \dstate}
	\left( \stateRef(\gparam)+\rdbasisnl(\grdstate), t^n;\gparam \right) \right) \jacrdbasisnl (\grdstate).
\end{align}
As with the \GalerkinName\ O$\Delta$E, the \LSPGName\ O$\Delta$E is nonlinear
if either the velocity is nonlinear in its first argument or if the trial
manifold is nonlinear.
\begin{remark}[Weighted $\ell^2$-norms for \LSPGName\
	projection]\label{rem:normLSPG}
	As discussed in Remark \ref{rem:normGal}, the minimization problem
	in \eqref{eq:disc_opt_problem} could be posed in other norms. If the
	weighted $\ell^2$-norm $\|\cdot\|_\metric$ is employed, then Problem 
	\eqref{eq:disc_opt_problem} becomes
\begin{equation}\label{eq:disc_opt_problem_norm}
	\rdstate^n(\param)  =
	\underset{\grdstate\in\RR{\nstatered}}{\arg\min} \left\|
	\res^{n}\left(
	\stateRef(\param)+\rdbasisnl(\grdstate);\param)\right)\right\|_\metric^2,
\end{equation}
	and the test basis matrix in 
	the \LSPGName\ O$\Delta$E\eqref{eq:lspg_proj} becomes
\begin{align*}
\rdtestbasis^{n}:(\grdstate;\gparam) &\mapsto \metric\frac{\partial \res^n}{\partial
	\dstate}(
	\stateRef(\gparam)+\rdbasisnl(\grdstate);\gparam)\jacrdbasisnl(\grdstate).
\end{align*}
As with \GalerkinName\ projection, this work restricts attention to the
	$\ell^2$-norm for notational simplicity.
\end{remark}

\begin{remark}[Application to stationary problems]
	\LSPGNameCap\ projection can also be applied to stationary (i.e.,
	steady-state) problems.
	Stationary problems are characterized by computing $\states(\param)$ as the
	implicit solution to Eq.~\eqref{eq:govern_eq} with
	$\difftime{\states}=\zero$, i.e.,
\begin{equation}\label{eq:govern_eq_steady}
	\velo(\states;\param)=\zero,
\end{equation}
	where the (time-independent) velocity is equivalent to the stationary-problem
	residual. Substituting the approximated state
	$\state(\param)\leftarrow\aprxstate(\param) = \stateRef(\param)+
	\rdbasisnl(\rdstate(\param)) $ 
	into Eq.~\eqref{eq:govern_eq_steady}	
	and subsequently minimizing the $\ell^2$-norm
	of the residual yields
\begin{equation}\label{eq:disc_opt_problem_steady}
	\rdstate(\param)  =
	\underset{\grdstate\in\RR{\nstatered}}{\arg\min} \left\|
	\velo\left(
	\stateRef(\param)+\rdbasisnl(\grdstate);\param)\right)\right\|_2^2,
\end{equation}
which has the same nonlinear-least-squares form as \LSPGName\ projection applied to the dynamical
	system model \eqref{eq:disc_opt_problem}. In this case,  necessary
	optimality conditions for Problem \eqref{eq:disc_opt_problem_steady}
	correspond to stationarity of the objective function, i.e., the solution
	$\rdstate(\param)$ satisfies the system of algebraic
	equations 
\begin{align}\label{eq:lspg_proj_steady}
	\rdtestbasis(\rdstate(\param);\param)\tran
	\velo\left(
	\stateRef(\param)+\rdbasisnl(\rdstate(\param));\param\right) = \zero,
\end{align}
where the test basis matrix
$\rdtestbasis:\RR{\nstatered}\times\paramspace\rightarrow\RR{\nstate\times\nstatered}$
is defined as
\begin{align*}
\rdtestbasis:(\grdstate;\gparam) &\mapsto
\frac{\partial \velo}{\partial \dstate}
	\left( \stateRef(\gparam)+\rdbasisnl(\grdstate);\gparam \right)	
\jacrdbasisnl (\grdstate).
\end{align*}
Note that we do not consider a \GalerkinName\ projection applied to stationary problems,
	as the \GalerkinName\ ROM was derived by minimizing the time-continuous
	residual, and a Galerkin-like projection  of the form $ \jacrdbasisnl
	(\rdstate(\param))\tran \velo\left(
	\stateRef(\param)+\rdbasisnl(\rdstate(\param));\param\right) = \zero $ does
	not generally associate with any optimality property.
\end{remark}

\subsection{Quasi-Newton methods for implicit integrators}\label{sec:quasi}
When an implicit time integrator is employed such that $\beta_0\neq 0$ and the
trial manifold is nonlinear, then solving the \GalerkinName\ O$\Delta$E
\eqref{eq:time_cont_res} and \LSPGName\ O$\Delta$E \eqref{eq:lspg_proj} using
Newton's method is challenging, as the residual Jacobians involve high-order
derivatives of the decoder. For this purpose, we propose quasi-Newton methods that
approximate these Jacobians while retaining convergence of the nonlinear
solver to the solution of the O$\Delta$Es under certain conditions.

\textbf{\GalerkinNameCap}. We first consider the \GalerkinName\ case. 
If an implicit integrator is employed (i.e., $\beta_0\neq 0$), then the \GalerkinName\
	O$\Delta$E \eqref{eq:GalerkinODeltaE} corresponds to a system of algebraic
	equations. Applying Newton's method to solve this
	system can be challenging for nonlinear decoders, as the Jacobian of the
	residual requires computing a third-order tensor of second derivatives in
	this case. Indeed, the $i$th column of the Jacobian is
\begin{align*}
	\frac{\partial \rdresn}{\partial \drdstateArg{i}}: (\krdstate;\gparam)\mapsto \alpha_0
	\unitVecRedArg{i} - \Delta t \beta_0 \left( \frac{\partial \jacrdbasisnl
	\pseudoinv}{\partial \drdstateArg{i}}\left(\krdstate\right)\velo(
	\stateRef(\gparam)+\rdbasisnl(\krdstate), t^n;\gparam)  + \jacrdbasisnl(\krdstate)\pseudoinv
	\frac{\partial \velo }{\partial \dstate}(
	\stateRef(\gparam)+\rdbasisnl(\krdstate), t^n;\gparam) \jacrdbasisnl(\krdstate)\unitVecRedArg{i}\right)
\end{align*}
	for $i=1,\ldots,\nstatered$, where
	$\unitVecRedArg{i}\in\{0,1\}^{\nstatered}$ denotes the $i$th
	canonical unit vector. The gradient of the
	pseudo-inverse of the decoder Jacobian, which appears in the second term of
	the right-hand side, can be computed from the gradients of this
	Jacobian as
 \begin{align*} 
 \begin{split} 
\frac{\partial\jacrdbasisnl\pseudoinv}{\partial
	 \drdstateArg{i}}\left(\krdstate\right)
	 =&-\jacrdbasisnl(\krdstate)\pseudoinv
\frac{\partial\jacrdbasisnl}{\partial \drdstateArg{i}}\left(\krdstate\right)
\jacrdbasisnl(\krdstate)\pseudoinv
	 +\jacrdbasisnl(\krdstate)\pseudoinv\jacrdbasisnl(\krdstate)^{\pseudoinvSymb\,T}
\frac{\partial\jacrdbasisnl^T}{\partial \drdstateArg{i}}\left(\krdstate\right)
	 (\identity - \jacrdbasisnl(\krdstate)\jacrdbasisnl(\krdstate)\pseudoinv)\\
	 &+ 
	 (\identity- \jacrdbasisnl(\krdstate)\pseudoinv \jacrdbasisnl(\krdstate))
\frac{\partial\jacrdbasisnl^T}{\partial \drdstateArg{i}}\left(\krdstate\right)
	 \jacrdbasisnl(\krdstate)^{\pseudoinvSymb\,T}\jacrdbasisnl(\krdstate)\pseudoinv
	  \end{split} 
	  \end{align*} 
		The primary difficulty of computing this term is the requirement of
		computing second derivatives of the form
		$
		{\partial
	\jacrdbasisnl}/{\partial
	\drdstateArg{i}}=
		{\partial^2\rdbasisnl}/{\partial \drdstate\partial
		\drdstateArg{i}}$. For this reason, we propose to
		approximate the residual Jacobian as $\jacrdresApprox\approx{\partial
		\rdres}/{\partial \drdstate}
	$ by neglecting this term such that
\begin{align*}
	\jacrdresApprox: (\krdstate;\gparam)\mapsto \alpha_0
	\identity - \Delta t \beta_0  \jacrdbasisnl(\krdstate)\pseudoinv
	\frac{\partial \velo }{\partial \dstate}(
	\stateRef(\gparam)+\rdbasisnl(\krdstate), t^n;\gparam) \jacrdbasisnl(\krdstate)
\end{align*}
	The resulting quasi-Newton method to solve the \GalerkinName\ O$\Delta$E
	\eqref{eq:GalerkinODeltaE} takes the form
\begin{align} \label{eq:quasiNewtonGal}
\begin{split} 
	&\jacrdresApprox(\rdstateNK;\param)\searchDirNK =
	-\rdres(\rdstateNK;\param)\\
	&\rdstateNKp = \rdstateNK + \linesearchNK\searchDirNK,
\end{split} 
\end{align} 
	for $k=0,\ldots,K$. Here, $\rdstateNZero$ is the initial guess (often taken
	to be $\rdstate^{n-1}$ or a data-driven extrapolation
	\cite{carlberg2015decreasing}), and $\linesearchNK\in\RR{}$ is the step
	length.
	Convergence of the quasi-Newton iterations 
	\eqref{eq:quasiNewtonGal} to a root of  $\rdres(\cdot;\param)$
	is ensured if	$$\rdres(\rdstateNK;\param)^T
	\jacrdresApprox(\rdstateNK;\param)^T \frac{\partial \rdresn}{\partial
	\drdstate} (\rdstateNK;\param) \rdres(\rdstateNK;\param) > 0,\quad
	k=0,\ldots,K $$ and
	$\linesearchNK\in\RR{}$, $k=0,\ldots,K$ is computed to satisfy the strong
	Wolfe conditions, as this ensures that the search direction $\searchDirNK$
	is a descent direction of the function $\| \rdres(\grdstate; \gparam)
	\|_2^2$ evaluated at $(\grdstate; \gparam)=(\rdstateNK;\param)$ for
	$k=0,\ldots,K$
	\cite{martinez2000practical}.
	
\textbf{\LSPGNameCap}. We now consider the \LSPGName\ case. As in previous
works that considered applying LSPG on affine trial subspaces
\cite{carlberg2011efficient,carlberg2013gnat,carlbergGalDiscOpt}, we propose
to solve the nonlinear least-squares problem \eqref{eq:disc_opt_problem} using
the Gauss--Newton method, which leads to the iterations
\begin{align} \label{eq:GaussNewton}
\begin{split} 
	& \rdtestbasis^{n}(\rdstateNK;\param)\tran
	\rdtestbasis^{n}(\rdstateNK;\param) \searchDirNK =
	-
	\rdtestbasis^{n}(\rdstateNK;\param)\tran
\res^{n}\left(
	\stateRef(\param)+\rdbasisnl(\rdstateNK);\param\right)\\
	&\rdstateNKp = \rdstateNK + \linesearchNK\searchDirNK,
\end{split} 
\end{align} 
	for $k=0,\ldots,K$ with $\rdstateNZero$ a provided initial guess, and
	$\linesearchNK\in\RR{}$ is a step length chosen to satisfy the strong Wolfe
	conditions, which ensure global convergence to a local minimum of the
	\LSPGName\ objective function in \eqref{eq:disc_opt_problem} that satisfies
	the stationarity conditions associated with the LSPG O$\Delta$E
	\eqref{eq:lspg_proj}.
	
	To show the connection between the Gauss--Newton method and the quasi-Newton
	approach proposed for the \GalerkinName\ method, we define the discrete residual
	associated with \LSPGName\ projection as
\begin{equation} 
	\rdresLSPGn:(\grdstate;\gparam)\mapsto\rdtestbasis^{n}(\grdstate;\gparam)\tran
	\res^{n}\left(
	\stateRef(\gparam)+\rdbasisnl(\grdstate);\gparam\right)
\end{equation} 
such that the \LSPGName\ O$\Delta$E \eqref{eq:lspg_proj} can be expressed
simply as the system of algebraic equations
\begin{align}\label{eq:lspg_projsimple}
	\rdresLSPGn(\rdstate^n;\param)=\zero.
\end{align}
Solving Eq.~\eqref{eq:lspg_projsimple} using Newton's method requires
computing the residual Jacobian, whose $i$th column is
\begin{align*}
	\frac{\partial \rdresLSPGn}{\partial \drdstateArg{i}}: (\krdstate;\gparam)&\mapsto 
	\frac{\partial\rdtestbasis^{n}}{\partial\drdstateArg{i}}(\krdstate;\gparam)\tran
	\res^{n}\left(
	\stateRef(\gparam)+\rdbasisnl(\krdstate);\gparam\right)
	+
\rdtestbasis^{n}(\krdstate;\gparam)\tran
	\rdtestbasis^{n}(\krdstate;\gparam)\unitVecRedArg{i}\\
	&=
\frac{\partial \jacrdbasisnl
	}{\partial \drdstateArg{i}}\left(\krdstate\right)\tran
	\frac{\partial \res^{n}}{\partial \dstate}\left(
	\stateRef(\gparam)+\rdbasisnl(\krdstate);\gparam\right)\tran
\res^{n}\left(
	\stateRef(\gparam)+\rdbasisnl(\krdstate);\gparam\right)
+
\rdtestbasis^{n}(\krdstate;\gparam)\tran
	\rdtestbasis^{n}(\krdstate;\gparam)\unitVecRedArg{i}.
\end{align*}
Thus, it is clear that the Gauss--Newton iterations \eqref{eq:GaussNewton} are
equivalent to employing a quasi-Newton method to solve the \LSPGName\
O$\Delta$E \eqref{eq:lspg_projsimple} with an approximated 
residual Jacobian 
$\jacrdresLSPGApprox\approx{\partial
		\rdresLSPGn}/{\partial \drdstate}$ defined as
$\jacrdresLSPGApprox:(\krdstate;\gparam) \mapsto \rdtestbasis^{n}(\krdstate;\gparam)\tran
	\rdtestbasis^{n}(\krdstate;\gparam)$, which is obtained from neglecting the
	term containing second derivatives of the decoder, namely ${\partial
	\jacrdbasisnl}/{\partial
	\drdstateArg{i}}={\partial^2\rdbasisnl}/{\partial \drdstate\partial
		\drdstateArg{i}}$.

\input{analysis}
\OWN{\input{error_analysis}}
\section{Nonlinear trial manifold based on deep convolutional autoencoders}\label{sec:manifold}
This section describes the approach we propose for constructing the decoder
$\rdbasisnl:\RR{\nstatered}\rightarrow\RR{\nstate}$ that defines the nonlinear
trial manifold. As described in the introduction, any nonlinear-manifold
learning method equipped with a continuously differentiable mapping from the
generalized coordinates (i.e., the latent state) to an approximation of the
state (i.e., the data) is compatible with the \GalerkinName\ and \LSPGName\
methods proposed in Section \ref{sec:rom}. Here, we
pursue deep convolutional autoencoders for this purpose, as they are very
expressive and scalable; there is also high-performance software available
for their construction.
When the nonlinear trial manifold is constructed using the proposed
deep convolutional autoencoder,
we refer to the \GalerkinNameCap\ ROM as
the \textit{\GalerkinNameDeepCap}\ ROM and the \LSPGNameCap\ ROM as the
\textit{\LSPGNameDeepCap} ROM.\footnote{This is analogous to the naming
convention employed for linear-subspace ROMs, as linear-subspace Galerkin
projection with POD is referred to as a
POD--Galerkin ROM.}

Section \ref{sec:autoenc} describes the mathematical structure of
autoencoders, which provides a neural-network model for the decoder
$\rdbasisnl(\cdot; \nnparam):\mathbb{R}^{\dofrom} \rightarrow
\mathbb{R}^{\dofhf}$, where $\nnparam$ denotes the neural-network parameters
to be computed during training; Section \ref{sec:autoencoder} describes the
proposed autoencoder architecture, which is tailored to spatially distributed
dynamical-system states; Section \ref{sec:initialCondition} describes how we
propose to satisfy the initial condition using the proposed autoencoder in
line with Remark \ref{rem:IC}.

\subsection{Autoencoder}\label{sec:autoenc}

An autoencoder is a type of feedforward neural network that aims to learn the
identity mapping, i.e., $\autoenc:\gvec\mapsto\gaprxvec$ with
$\gaprxvec\approx \gvec$ and $\autoenc:\RR{\ninput}\rightarrow\RR{\ninput}$.
Autoencoders achieve this using an architecture composed of two parts: the
encoder $\encoder:\gvec\mapsto\latcode$ with
$\encoder:\RR{\ninput}\rightarrow\RR{\nlatent}$ and $\nlatent\ll\ninput$, which 
maps a high-dimensional vector $\gvec$ to a low-dimensional 
\textit{code} $\latcode$;
and the decoder
$\decoder:\latcode\mapsto\gaprxvec$ with
$\decoder:\RR{\nlatent}\rightarrow\RR{\ninput}$, which maps the code
$\latcode$ to an
approximation of the original high-dimensional vector $\gaprxvec$. Thus, the resulting
autoencoder takes the form
\begin{equation*}
\autoenc:\gvec \mapsto \decoder \circ \encoder(\gvec).
\end{equation*}
If $\autoenc(\gvec)\approx\gvec$ over a data set
$\gvec\in\{\gvec^{(i)}\}_{i}$, then the low-dimensional codes $\encoder(\gvec^{(i)})$
contain sufficient information to recover accurate approximations
 of the data $\gaprxvec^{(i)}\approx\gvec^{(i)}$ via the application of the
 decoder $\decoder$. In this work, we propose to employ the autoencoder
 decoder $\decoder$ for the decoder $\rdbasisnl$ used to define the
 nonlinear trial manifold introduced in Section
 \ref{sec:nonlinearTrialManifold}.

 In feedforward networks, each network layer typically corresponds to a vector or tensor, whose values are computed by applying an affine transformation to
 the previous layer followed by a nonlinear activation function. An encoder
 with $\nenclayer$ layers takes the form
\begin{equation*}
\encoder:(\gvec; \encparam) \mapsto \enclayer_{\nenclayer} (\cdot; \encparaml{\nenclayer}) \circ \enclayer_{\nenclayer-1}(\cdot;\encparaml{{\nenclayer-1}}) \circ \cdots \circ \enclayer_1 (\gvec; \encparaml{1}),
\end{equation*}
where $\enclayer_i(\cdot;\encparaml{i}): \mathbb{R}^{\dofenclayer{i-1}}
\rightarrow \mathbb{R}^{\dofenclayer{i}}$, $i=1,\ldots,\nenclayer$ denotes the
function applied at layer $i$ of the neural network; $\encparaml{i}$,
$i=1,\ldots,\nenclayer$ denotes the weights and the biases employed at layer
$i$ with $\encparam \equiv (\encparaml{1},\ldots,\encparaml{\nenclayer})$; and
$\dofenclayer{i}$, $i=1,\ldots,\nenclayer$ denotes the dimensionality of the
output at layer $i$. The input has dimension
$\dofenclayer{0} = \ninput$ and the final layer produces a code with dimension
$\dofenclayer{\nenclayer} = \nlatent$. The nonlinear activation function is
applied in layers 1 to $\nenclayer$ to a function of the weights, biases, and the
outputs from the previous layer 
such that
\begin{equation*}
\enclayer_{i}:(\gvec; \encparaml{i}) \mapsto \encactivation_i ( \encaffinetrans_i (\encparaml{i}, \gvec)),
\end{equation*}
where $\encactivation_i$ is an element-wise nonlinear activation function. For fully-connected layers as in the traditional multilayer
perceptron (MLP), $\encaffinetrans_i (\encparaml{i}, \gvec) =
\encparaml{i}[1, \gvec\tran]\tran$ with $\encparaml{i} \in
\mathbb{R}^{\dofenclayer{i} \times (\dofenclayer{i-1}+1)}$ a real-valued
matrix. For convolutional layers, $\encaffinetrans_i$ corresponds to a convolution
operator with $\encparaml{i}$ providing the
convolutional-filter weights.

A decoder with $\ndeclayer$ layers also corresponds to a feedforward network;
it takes the form
\begin{equation*}
\decoder:(\latcode; \decparam) \mapsto \declayer_{\ndeclayer} (\cdot;
	\decparaml{\ndeclayer}) \circ
	\declayer_{\ndeclayer-1}(\cdot;\decparaml{{\ndeclayer-1}}) \circ \cdots
	\circ \declayer_1 (\latcode; \decparaml{1})
\end{equation*}
with $\declayer_{i} (\cdot; \decparaml{i}) : \mathbb{R}^{\dofdeclayer{i-1}}
\rightarrow \mathbb{R}^{\dofdeclayer{i}}$, $i=1,\ldots,\ndeclayer$, and
$\decparam \equiv (\decparaml{1},\ldots,\decparaml{\ndeclayer})$. The
dimension of the input is $\dofdeclayer{0} = \nlatent$ and the dimension of
the final layer (i.e., the output layer) is $\dofdeclayer{\ndeclayer} =
\ninput$. Again, the nonlinear activation is applied to the output of the
previous layer such that
\begin{equation*}
\declayer_{i}:(\gvec; \decparaml{i}) \mapsto \decactivation_i(
	\decaffinetrans_i (\decparaml{i}, \gvec)).
\end{equation*}
As in the encoder, 
$\decaffinetrans_i (\decparaml{i}, \gvec) =
\decparaml{i}[1, \gvec\tran]\tran$ with $\decparaml{i} \in
\mathbb{R}^{\dofdeclayer{i} \times (\dofdeclayer{i-1}+1)}$ a real-valued
matrix
for fully-connected layers, while $\decaffinetrans_i$ corresponds to a
\RO{convolution} operator with $\decaffinetrans_i$ providing transposed
convolutional-filter weights for convolutional layers. 

Because MLP autoencoders are fully connected, the number of parameters
(corresponding to the edge weights and biases) \RO{can} be extremely large when the
number of inputs $\ninput$ is large; in such scenarios, these models typically
require a large amount of training data. As this work aims to enable
model reduction for large-scale dynamical systems, directly applying an
MLP autoencoder to the state such that $\ninput=\nstate$ is not 
practical in many scenarios. To address this, alternative neural-network architectures
have
been devised that make use of \textit{parameter sharing} to reduce the total
number of parameters in the model, and thus the amount of data needed for
training. In the context of dynamical systems characterized by a state that
can be represented as spatially distributed data, we propose to apply
\textit{convolutional autoencoders}. Such methods  are
applicable to multi-channel spatially distributed input data and
employ parameter sharing such that they
can be trained with less data.
Further, such
models tend to generalize well to unseen test data
\cite{lecun2015deep,lecun1998gradient} because they exploit three key
properties of natural signals: local connectivity, parameter sharing, and
equivariance to translation \cite{goodfellow2016deep,lecun2015deep}.

\subsection{Deep convolutional autoencoders for spatially distributed
states}\label{sec:autoencoder}

Many dynamical systems are characterized by a state that can be represented as
spatially distributed data, e.g., spatially discretized
partial-differential-equations
models. In such cases, there is a \textit{restriction operator}
$\restrictionOp$ that maps the
state to a tensor representing spatially distributed data, i.e., 
\begin{equation}\label{eq:restrictionOP}
\restrictionOp:\RR{\nstate}\rightarrow\RR{\nspaceDim{1}\times\cdots\times\nspaceDim{\nDim}\times\nchannel},
\end{equation}
where $\nspaceDim{i}$ denotes the number of discrete points in spatial
dimension $i$; $\nDim\in\{1,2,3\}$ denotes the spatial dimension,
and $\nchannel$ denotes the number of \textit{channels}. For images,
typically $\nchannel=3$ (i.e., red, green\RO{, and} blue). For dynamical system models
the number of channels $\nchannel$ is equal to the number of state variables
defined at a given spatial location; for example, $\nchannel$ is equal to the
number of conserved variables when the dynamical system arises from the
spatial discretization of a conservation law. We also write the associated
\textit{prolongation operator}, which aims to provide the inverse mapping such that
\begin{equation}\label{eq:prolongationOp}
\prolongationOp:\RR{\nspaceDim{1}\times\cdots\times\nspaceDim{\nDim}\times\nchannel}\rightarrow\RR{\nstate}.
\end{equation}
For coarse discretizations of the spatial domain, it is possible to ensure
$\restrictionOp\circ\prolongationOp(\cdot)$ corresponds to the
identity map; for fine discretizations, it is possible to ensure
$\prolongationOp\circ\restrictionOp(\cdot)$ corresponds to the identity
map; for cases where underlying grid provides an isomorphic representation of
the state, it is possible to achieve both \cite{carlberg2018recovering}.

After reformatting the state from a vector to a tensor by applying the
restriction operator $\restrictionOp$, we apply an invertible affine
scaling operator 
$\scalingOp:\RR{\nspaceDim{1}\times\cdots\times\nspaceDim{\nDim}\times\nchannel}
\rightarrow\RR{\nspaceDim{1}\times\cdots\times\nspaceDim{\nDim}\times\nchannel}$
for data-standardization purposes.
Section \ref{sec:standardization} defines the specific
elements of this operator, which are computed from
the training data.


\begin{sidewaysfigure}[p]
\centering
    \includegraphics[width=\linewidth]{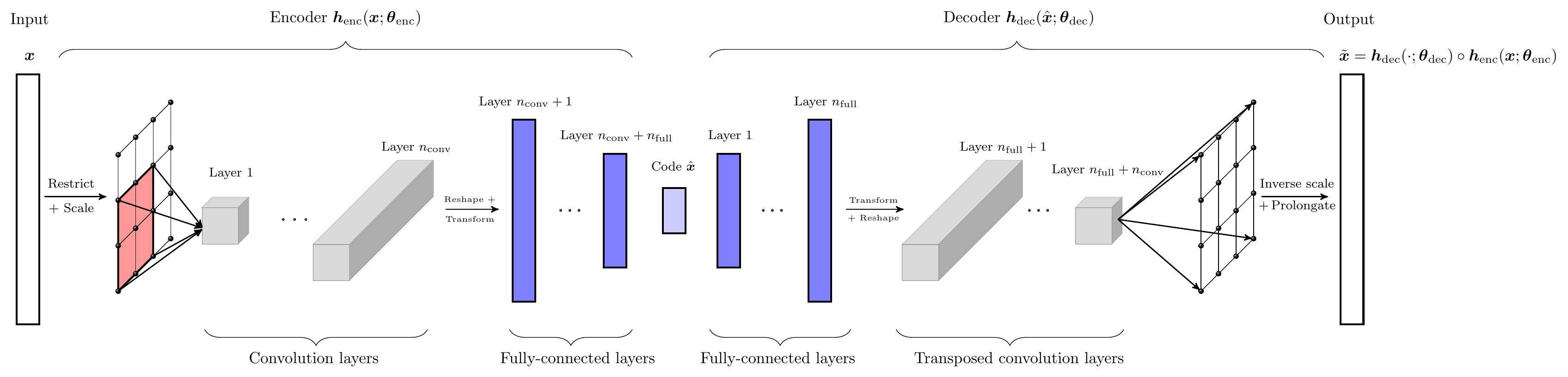}
\caption{Network architecture of the deep convolutional autoencoder, which takes a state of dynamical systems as an input and produces an approximated state as an output. The encoder extracts low-dimensional code $\latcode$ from the discrete representation of the state $\gvec$ by applying the restriction operator (Eq. \eqref{eq:restrictionOP}) and a set of convolutional layers (gray boxes), followed by a set of fully-connected layers (blue rectangles). The decoder approximately reconstructs the high-dimensional vector $\autoenc(\gvec;\nnparam)$ by performing the inverse operations of the encoder, applying fully-connected layers (blue rectangles), followed by the transposed-convolutional layers (gray boxes), then applying the prolongation operator (Eq. \eqref{eq:prolongationOp}) to the resulting quantity.}
\label{fig:network}
\end{sidewaysfigure}

Figure \ref{fig:network} depicts the architecture of the proposed deep
convolutional autoencoder. Before the first convolutional layer, the autoencoder applies the
restriction operator $\restrictionOp$ and scaling operator $\scalingOp$, while
after the last convolutional layer, the autoencoder applies the inverse scaling
operator $\scalingOp^{-1}$ and prolongation operator $\prolongationOp$.
\RT{We note that the restriction and prolongation operators
are (deterministic) operations that simply reshape the state into an appropriate
format for the autoencoder, while the scaling operator is defined explicitly
from the range of the training data; thus, these quantities are not subject to training.}
We consider the combination of convolutional
layers (gray boxes) and fully-connected layers (blue rectangles).
\ref{app:conv} provides a description of convolutional layers, including
hyperparameters and parameters that are subject to optimization.  The encoder
network is composed \RO{of} restriction and scaling, followed \RO{by} a sequence of
$\nconv$ convolutional layers and
$\nfull$ fully-connected layers.  The decoder network is
composed of \RO{a} sequence of $\nfull$ fully-connected layers and $\nconv$ transposed-convolutional layers with no nonlinear
activation in the last layer, followed by inverse scaling and prolongation. As a result, the dimension of the input to
the first convolutional layer is $\dofenclayer{0}=
\nchannel\prod_{i=1}^d\nspaceDim{i}$. The dimension $\dofenclayer{\nconv}$ of
the output of encoder layer $\nconv$ (i.e., the final convolutional layer) is
determined by the hyperparameters defining the kernels used in the
convolutional layers (i.e., depth, stride, zero-padding).  Similarly, the
dimension $\dofdeclayer{\nfull}$ of the output of decoder layer $\nfull$
(i.e., the final fully connected layer) is determined by the hyperparameters
defining the kernels in the subsequent convolutional layers.  The dimension of
the output of the final decoder layer is $\dofdeclayer{\nlayer} =
\dofenclayer{0}= \nchannel\prod_{i=1}^d\nspaceDim{i}$.

A particular instance of the network architecture can be defined by specifying
the number of convolutional layers $\nconv$, the number of fully-connected
layers $\nfull$, the number of units in each layer, and the types of nonlinear
activations. Such architecture-related parameters are typically considered
hyperparameters, as they are not subject to optimization during training.

We propose to set the decoder for the proposed manifold ROMs to the decoder
arising from the deep convolutional autoencoder architecture described in
Figure \ref{fig:network}, i.e., $\rdbasisnl=\decoder$. When the 
\GalerkinNameCap\ and 
\LSPGNameCap\ ROMs employ this choice of decoder, we 
refer to them as \textit{\GalerkinNameDeepCap}\ and \textit{\LSPGNameDeepCap}\
ROMs, respectively.

\subsection{Initial condition satisfaction}\label{sec:initialCondition}
As discussed in Remark \ref{rem:IC}, the initial generalized coordinates
$\initrdstate(\param)$ can be ensured to satisfy the initial conditions by
employing a reference state defined by Eq.~\eqref{eq:refStateIC}; however,
this implies that the decoder must be able to accurately represent deviations
from this reference state.  To accomplish this for the proposed autoencoder,
we propose (1) to train the autoencoder with snapshot data centered on the
initial condition along with the zero vector (as described in Section
\ref{sec:snapshots}), and (2) to set the initial generalized coordinates to an
encoding of zero, i.e., $\initrdstate(\param) = \encoder(\zero)$ for all
$\param\in\paramspace$.  Then, setting the reference state according to
Eq.~\eqref{eq:refStateIC} leads to a reference state of $\stateRef(\param)=
\initstate(\param)-\rdbasisnl(\encoder(\zero))$. Further, if
$\decoder(\encoder(\zero))\approx \zero$, as is encouraged by including the
zero vector in training, then the reference state comprises a perturbation of
the initial condition, and the decoder $\decoder$ must only represent
deviations from this perturbation. This is consistent with the training of the
autoencoder, as the snapshots have been
centered on the initial condition.

\section{Offline training}\label{sec:offline}
This section describes the offline training process used to train the deep
convolutional autoencoder proposed in Section \ref{sec:autoencoder}.  The
approach employs precisely the same snapshot data used by POD. Section
\ref{sec:snapshots} describes the (snapshot-based) data collection procedure,
which is identical to that employed by 
POD. Section \ref{sec:standardization} describes how the data
are scaled to improve numerical stability of autoencoder training.  Section
\ref{sec:grad_learning} summarizes the gradient-based-optimization approach
employed for training the autoencoder (i.e., computing the optimal parameters
$\nnparamOpt$) given the scaled training data.
Algorithm \ref{alg:ae_train} provides a summary of offline training.

\begin{algorithm}[H]
    \caption{Offline training}
    \label{alg:ae_train}
    \algorithmicrequire{Training-parameter instances $\paramspacetrain$;
	restriction operator $\restrictionOp$; prolongation operator $\prolongationOp$;
	autoencoder architecture; SGD hyperparameters (adaptive learning-rate strategy; initial parameters $\nnparam^{(0)}$; 
		 number of minibatches $\nbatch$; maximum number of epochs $\nepoch$;
		 early-stopping criterion)}\\
    \algorithmicensure{Encoder
		$\encoder$; decoder $\rdbasisnl$}
    \begin{algorithmic}[1] 
				\State Solve the FOM O$\Delta$E \eqref{eq:fomODeltaE} for
		$\param\in\paramspacetrain$ and form the
			snapshot matrix $\snapshotMat$ (Eq.~\eqref{eq:snapshotMat}).
		\State Compute the scaling operator $\scalingOp$
		(Eq.~\eqref{eq:scalingDef}), which completes the definition of the
		autoencoder $\autoenc(\gvec; \nnparam) = \decoder(\cdot;\decparam)\circ
		\encoder(\gvec; \encparam)$.
			\State Train the autoencoder by executing Algorithm \ref{alg:ae_sgd}
			in \ref{app:SGD} with inputs corresponding to the snapshot matrix
			$\snapshotMat$; the autoencoder $\autoenc(\gvec; \nnparam)$; and SGD hyperparameters.
			This returns the encoder $\encoder$ and decoder $\rdbasisnl$.
\end{algorithmic}
\end{algorithm}

\subsection{Snapshot-based data collection}\label{sec:snapshots}
The first step of offline training is snapshot-based data collection. This
requires solving the FOM O$\Delta$E \eqref{eq:fomODeltaE} for
training-parameter instances
$\param\in\paramspacetrain \equiv\{\paramtrain^{i}
\}_{i=1}^{\ntrain} \subset \paramspace$ and assembling the snapshot matrix
\begin{equation} \label{eq:snapshotMat}
	\snapshotMat \defeq \left[\snapshotMatArg{\paramtrain^{1}}\ \cdots\
	\snapshotMatArg{\paramtrain^{\ntrain}}
	\right] \in \mathbb{R}^{\dofhf \times \nsnap}
\end{equation} 
with $\nsnap\defeq\nseq\ntrain$ and
$
	\snapshotMatArg{\param} \defeq [
	\state^{1}(\param) - \state^{0}(\param)\ 
	\cdots\ 
	\state^{\nseq}(\param) - \state^{0}(\param)].
$
		 
\begin{remark}[Proper orthogonal decomposition]\label{rem:pod}
POD employs the snapshot matrix
	$\snapshotMat$ to compute a trial basis matrix
	$\rdbasis$ used to define an affine trial subspace $\stateRef(\param) +
	\range(\rdbasis)$. To do so, POD computes the
	singular value decomposition (SVD) and sets the trial basis matrix to be equal to the first $\nstatered$
	left singular \OWN{vectors}, i.e.,
	\begin{align}
	\snapshotMat &= \bm{U} \bm{\Sigma} \bm{V}\tran,\qquad
		\rdbasisi_i = \bm{u}_i,\ \RO{i =1,\ldots,\dofrom}.
\end{align}
The resulting POD trial basis matrix $\rdbasis$ satisfies the minimization
	problem
	\begin{equation}\label{eq:PODopt} 
	\underset{\rdbasis\in\compactStiefel{\nstatered}{\ndof}}{\text{minimize}}
	\sum_{i=1}^{\nsnap}\|\stateSnapshot{i} -
	\rdbasis\rdbasis^T\stateSnapshot{i}\|_2^2,
\end{equation} 
	where $\stateSnapshot{i}$ denotes the $i$th column of
	$\snapshotMat$
	and
	$\compactStiefel{k}{n}$ and denotes the set of orthogonal $k\times n$
	matrices (the compact Stiefel manifold).  The solution is unique up to
	a rotation.
	This technique is equivalent (up to the data-centering process) to principal
	component analysis \cite{hotelling1933analysis}, an unsupervised
	machine-learning method for linear dimensionality reduction.
\end{remark}

\subsection{Data standardization}\label{sec:standardization}

As described in Section \ref{sec:autoencoder}, the first layer of the proposed
autoencoder applies a restriction operator $\restrictionOp$, which
reformats the input vector into a tensor compatible with convolutional
layers, followed by an affine scaling operator
$\scalingOp$; the last layer applies the inverse of this scaling operator
$\scalingOp^{-1}$ and subsequently applies the prolongation operator
$\prolongationOp$ to reformat the data into a vector. We now define
the scaling operator from the training data to ensure that all elements of the
training data lie between zero and one. This scaling improves numerical
stability of the gradient-based optimization for training
\cite{hsu2003practical,sarle1997neural}, which will be described in Section
\ref{sec:grad_learning}. We adopt a standard scaling procedure also
followed, e.g., by Ref.~\cite{gonzalez18}. Namely, 
defining the restriction of the $i$th snapshot as
$\stateSnapshotRestrict{i}\defeq\restrictionOp(\stateSnapshot{i})\in\RR{\nspaceDim{1}\times\cdots\times\nspaceDim{\nDim}\times\nchannel}$,
we set the elements of the scaling operator $\scalingOp$ to
\RO{
\begin{equation} \label{eq:scalingDef}
	\scalingOpEl{i_1}{i_\nDim}{j}:\stateElementNo\mapsto \frac{\stateElementNo -
	\minstateArg{j}}{\maxstateArg{j} -
	\minstateArg{j}}
\end{equation} 
}
with
\begin{align*} 
	\maxstateArg{j}&\defeq\max_{k\innatno{\nsnap},\;i_1\innatno{\nspaceDim{1}},\ldots,\;i_\nDim\innatno{\nspaceDim{\nDim}}}\stateSnapshotRestrictEl{k}{i_1}{i_\nDim}{j}, \quad \text{and} \\
	\minstateArg{j}&\defeq\min_{k\innatno{\nsnap},\;i_1\innatno{\nspaceDim{1}},\ldots,\;i_\nDim\innatno{\nspaceDim{\nDim}}}\stateSnapshotRestrictEl{k}{i_1}{i_\nDim}{j}.
	\end{align*} 

\subsection{Autoencoder training}\label{sec:grad_learning}

Once the autoencoder architecture has been defined (including the restriction,
prolongation and scaling operators), it takes the form $\autoenc(\gvec;
\nnparam) = \decoder(; \decparam) \circ \encoder(\gvec; \encparam)$, where the
undefined parameters $\nnparam$ correspond to convolutional-filter weights and
weights and biases for the fully connected layers.  We compute optimal values
of these parameters $\nnparamOpt$ using a standard approach from deep
learning: stochastic gradient descent (SGD) with minibatching and early
stopping \cite{bottou2018optimization}. \ref{app:SGD} provides
additional details, where Algorithm \ref{alg:ae_sgd} provides the training
algorithm.

\section{Numerical experiments}\label{sec:numexp}

This section assesses the performance of the proposed \GalerkinNameDeep\ and
\LSPGNameDeep\ ROMs, which employ nonlinear trial manifolds, compared to
POD--Galerkin and POD--LSPG ROMs, which employ affine trial subspaces.  We
consider two advection-dominated benchmark problems: 1D Burgers' equation and
a chemically reacting flow.  We employ the numerical PDE tools and ROM
functionality provided by \texttt{pyMORTestbed}
\cite{zahr2015progressive}, and we construct the autoencoder using
\texttt{TensorFlow} \cite{tensorflow2015-whitepaper}.

For both benchmark problems, the 
\GalerkinNameDeep\ and
\LSPGNameDeep\ ROMs
employ a \OWN{10}-layer convolutional autoencoder corresponding to the architecture
depicted in Figure \ref{fig:network}.
The encoder  $\encoder$ consists of $\nenclayer=\OWN{5}$ layers  with $\nconv=4$ convolutional
layers, followed by $\nfull=\OWN{1}$ fully-connected layer. 
The decoder $\decoder$ consists of 
$\nfull=\OWN{1}$ fully-connected layer, followed by \RO{$\nconv=4$} transposed-convolution
layers.  The latent code of the autoencoder is of dimension $\nstatered$
(i.e., $\dofenclayer{\nenclayer} = \dofdeclayer{0} = \nstatered$), which will
vary during the experiments to define different reduced-state dimensions.
Table \ref{tab:autoenc_arch} specifies attributes of the kernels used in
convolutional and transposed-convolutional layers.

For the nonlinear activation functions $\encactivation_i$, $i=1,\ldots,\nenclayer$ and $\decactivation_i$, $i=1,\ldots,\ndeclayer-1$, we use exponential linear units (ELU) \cite{clevert2015fast}, which is defined as 
\begin{align*}
\activation(\gvec) = \left \{  \begin{array}{cl} \gvec & \text{if } \gvec \geq 0 \\ \exp(\gvec)-1 & \text{otherwise} \end{array} \right.,
\end{align*}
and an identity activation function in the output at layer $\ndeclayer$ in
decoder (as is common practice). \RT{This choice of activation ensures that
the resulting decoder $\rdbasisnl$ is continuously differentiable everywhere, and is twice
continuously differentiable almost everywhere.} We employ the
$\ell^2$-loss function defined in Eq.~\eqref{eq:mse} in the minimization
problem \eqref{eq:cost_func}, which is equivalent to
the loss function minimized by POD (see Remark \ref{rem:pod}). We apply the
Adam optimizer \cite{kingma2014adam}, which is compatible with the stochastic
gradient descent (SGD) algorithm reported in Algorithm \ref{alg:ae_sgd}; here,
the adaptive learning rate strategy computes rates for different parameters
using estimates of first and second moments of the gradients. 

\begin{table}[!h]
	\caption{Parameters of the autoencoder architecture described in Figure
	\ref{fig:network} applied to both benchmark problems. The encoder 
	consists of $\nenclayer=\OWN{5}$ layers  with $\nconv=4$ convolutional layers,
	followed by $\nfull=\OWN{1}$ fully-connected layers.  The decoder 
	consists of $\nfull=\OWN{1}$ fully-connected layers, followed by \RO{$\nconv=4$}
	transposed-convolution layers. For the parameterized 1D Burgers' equation,
	\RO{we employ 1D convolutional kernel filters with the kernel length 25 at every convolutional layer and transposed-	convolutional layer,} and apply length-\OWN{2} stride ($s=\OWN{2}$) encoder layer 1 and the decoder layer \OWN{5},
	and length-\OWN{4} stride ($s=\OWN{4}$) for other convolutional layers and transposed-convolutional layers. For the chemically reacting flow, \OWN{we employ
	$5\times 5$ convolutional kernel filters at every convolutional layer and
	transposed-convolutional layer,} and apply length-2
	stride ($s=2$). \RO{For the definition of stride, we refer to \ref{app:conv}.} For zero-padding, we use half padding \cite{dumoulin2016guide,goodfellow2016deep}. \RT{With these settings, $\dofenclayer{\nconv} = \dofdeclayer{\nfull} = 128$ for the 1D Burgers' equation and $\dofenclayer{\nconv} = \dofdeclayer{\nfull} = 512$ for the chemically reacting flow.}}
\label{tab:autoenc_arch}
\begin{center}
\begin{tabular}{c c}
\begin{tabular}{|c c c |}
\multicolumn{3}{c}{Encoder network}\vspace{1mm}\\
\hline
\multicolumn{3}{|c|}{Convolution layers}\\
\hline
Layer & \multicolumn{2}{c|}{Number of filters}\\
\hline
1 & \multicolumn{2}{c|}{8}\\
2 & \multicolumn{2}{c|}{16}\\
3 & \multicolumn{2}{c|}{32}\\
4 & \multicolumn{2}{c|}{64}\\
\hline
\hline
\multicolumn{3}{|c|}{Fully-connected layers}\\
\hline
Layer & Input dimension & Output dimension  \\
\hline
\OWN{5} & \OWN{$\dofenclayer{\nconv}$} & \OWN{$\dofrom$}  \\
\hline
\end{tabular} & 
\begin{tabular}{|c c c |}
\multicolumn{3}{c}{Decoder network}\vspace{1mm}\\
\hline
\multicolumn{3}{|c|}{Fully-connected layers}\\
\hline
Layer & Input dimension & Output dimension  \\
\hline
\OWN{1} & \OWN{$\dofrom$} & \OWN{$\dofdeclayer{\nfull}$}\\
\hline
\hline
\multicolumn{3}{|c|}{Transposed-convolutional layers}\\
\hline
Layer & \multicolumn{2}{c|}{Number of filters} \\
\hline
2 & \multicolumn{2}{c|}{64}\\
3 & \multicolumn{2}{c|}{32}\\
4 & \multicolumn{2}{c|}{16}\\
5 & \multicolumn{2}{c|}{1}\\
\hline
\end{tabular}
\end{tabular}
\end{center}
\end{table}

Using the same snapshots as that to train the autoencoder, we also compute a
POD basis $\rdbasis$ following the steps discussed in Remark \ref{rem:pod}.

We compare the performance of four ROMs: 1) POD--Galerkin: linear-subspace
Galerkin projection (Eq.\ \eqref{eq:cont_romGal}) with the POD basis defining
$\rdbasis$, 2) POD-LSPG: linear-subspace LSPG projection (Eq.\
\eqref{eq:lspg_projLinear}) with the POD basis defining $\rdbasis$, 3)
\GalerkinNameDeepCap{}: \GalerkinName\ projection (Eq.\
\eqref{eq:cont_rom}) with the deep convolutional decoder defining
$\rdbasisnl$, and 4) \LSPGNameDeepCap{}: \LSPGName\ projection (Eq.\
\eqref{eq:lspg_proj}) with the deep convolutional decoder defining
$\rdbasisnl$. All ROMs enforce the initial condition by setting the reference
state according to Remark \ref{rem:IC}; this implies $\stateRef(\param) =
\initstate(\param)$ for the linear-subspace ROMs. 

To solve the O$\Delta$Es arising at each time instance, we 
apply Newton's method for POD--Galerkin, the Gauss--Newton method for
POD--LSPG, and the quasi-Newton methods proposed in Section \ref{sec:quasi}
for the \GalerkinNameDeep\ and \LSPGNameDeep.  We terminate the (quasi)-Newton
iterations when the residual norm drops below $10^{-6}$ of its initial guess
at that time instance; the initial guess corresponds to the solution at the
previous time instance.

To assess the ROM accuracy, we compute the relative $\ell^2$-norm
of the state error
\begin{equation}\label{eq:rel_err}
\text{relative error} = \left. \sqrt{\sum_{n=1}^{\nseq}  \|\states^n(\param)-\aprxstate^n(
	\param) \|_2^2} \middle/  \sqrt{\sum_{n=1}^{\nseq}
	\|\states^n(\param)\|_2^2} \right..
\end{equation}
We also include the projection error of the solution 
\begin{equation}\label{eq:proj_err}
\text{projection error} =\left. \sqrt{\sum_{n=1}^{\nseq}
\|(\identity - \rdbasis\rdbasis^T)(\states^n(\param)-\states^0(\param))]\|_2^2} \middle/  \sqrt{\sum_{n=1}^{\nseq}
	\|\states^n(\param)\|_2^2} \right.,
\end{equation}
onto both (1) the POD basis, in which case $\rdbasis$ is the POD basis
employed by POD--Galerkin and POD--LSPG, and (2) the optimal basis, in which
case $\rdbasis=\rdbasis^\star$, which consists of the first $\nstatered$ left
singular vectors of the snapshot matrix collected at the online-parameter
instance
$\snapshotMatArg{\param}$. 
\OWN{We refer to the} former metric \OWN{as the \textit{POD projection error}, as it} provides a lower bound for the POD--Galerkin and POD--LSPG
relative errors; \OWN{we refer to} the latter metric \OWN{as the \textit{optimal
projection error}, as it} provides an $\ell^2$-norm counterpart to
the Kolmogorov $\nstatered$-width. \OWN{Finally, we compute the projection of
the FOM solution on the trial manifold
\begin{equation}\label{eq:disc_opt_problem_optimal}
	\aprxstate_\star^n(
	\param) =
	\underset{\gstate\in
	\stateRef(\param)+ \manifold
	}{\arg\min} \left\|
\state^{n}-
\gstate
 \right\|_2
\end{equation}
and compute its relative error by employing
$\aprxstate^n\leftarrow\aprxstate_\star^n$
in Eq.~\eqref{eq:rel_err}; we refer to this as
the \textit{manifold projection error}.
}

\subsection{1D Burgers' equation}\label{sec:burg}
We first consider a parameterized inviscid Burgers'
equation~\cite{rewienski2003trajectory}, as it comprises a very simple
benchmark problem for which linear subspaces are ill suited due to its slowly
decaying Kolmogorov $n$-width. The governing system of partial differential
equations (with initial and boundary conditions) is
\begin{equation}\label{eq:burg_1d_cont}
\begin{split}
\frac{\partial \eqvar(\xoned,t;\param)}{\partial t} + \frac{\partial f (\eqvar(\xoned,t;\param))}{\partial \xoned} &= 0.02 e^{\paramelem{2} \xoned},  \quad \forall  x \in [0,100], \: \forall t \in [0, T] \vspace{2mm}\\
\eqvar(0,t;\param) &= \paramelem{1}, \quad  \forall t \in [0,T] \vspace{2mm}\\
\eqvar(\xoned,0) &= 1, \quad \forall \xoned \in [0,100],
\end{split}
\end{equation}
where the flux is $f(\eqvar) = 0.5 \eqvar^2$ and there are $\nparam=2$
parameters; thus, the intrinsic solution-manifold dimension is
$\intrinsicDim=3$ (see Remark \ref{rem:intrinsic}). We set the parameter domain to $\paramspace =[4.25, 5.5] \times
[0.015,0.03]$ and the final time to $T=35$.  We apply Godunov's scheme with
256 control volumes to spatially discretize Eq.\ \eqref{eq:burg_1d_cont},
which results in a system of parameterized ODEs of the form
\eqref{eq:govern_eq} with $\dofhf = 256$ spatial degrees of freedom and
initial condition $\initstate(\param)=\initstate=\ones$. For time
discretization, we use the backward-Euler scheme, which corresponds to a
linear multistep scheme with $k=1$, $\alpha_0=\beta_0=1$, $\alpha_1=-1$, and
$\beta_1=0$ in Eq.\ \eqref{eq:disc_res}. We consider a uniform time step
$\timestep = 0.07$, resulting in $\nseq = 500$ time instances.

\begin{figure}[!t]
\centering
    \begin{subfigure}[b]{0.45 \linewidth}
  \centering
    \includegraphics[width=\linewidth]{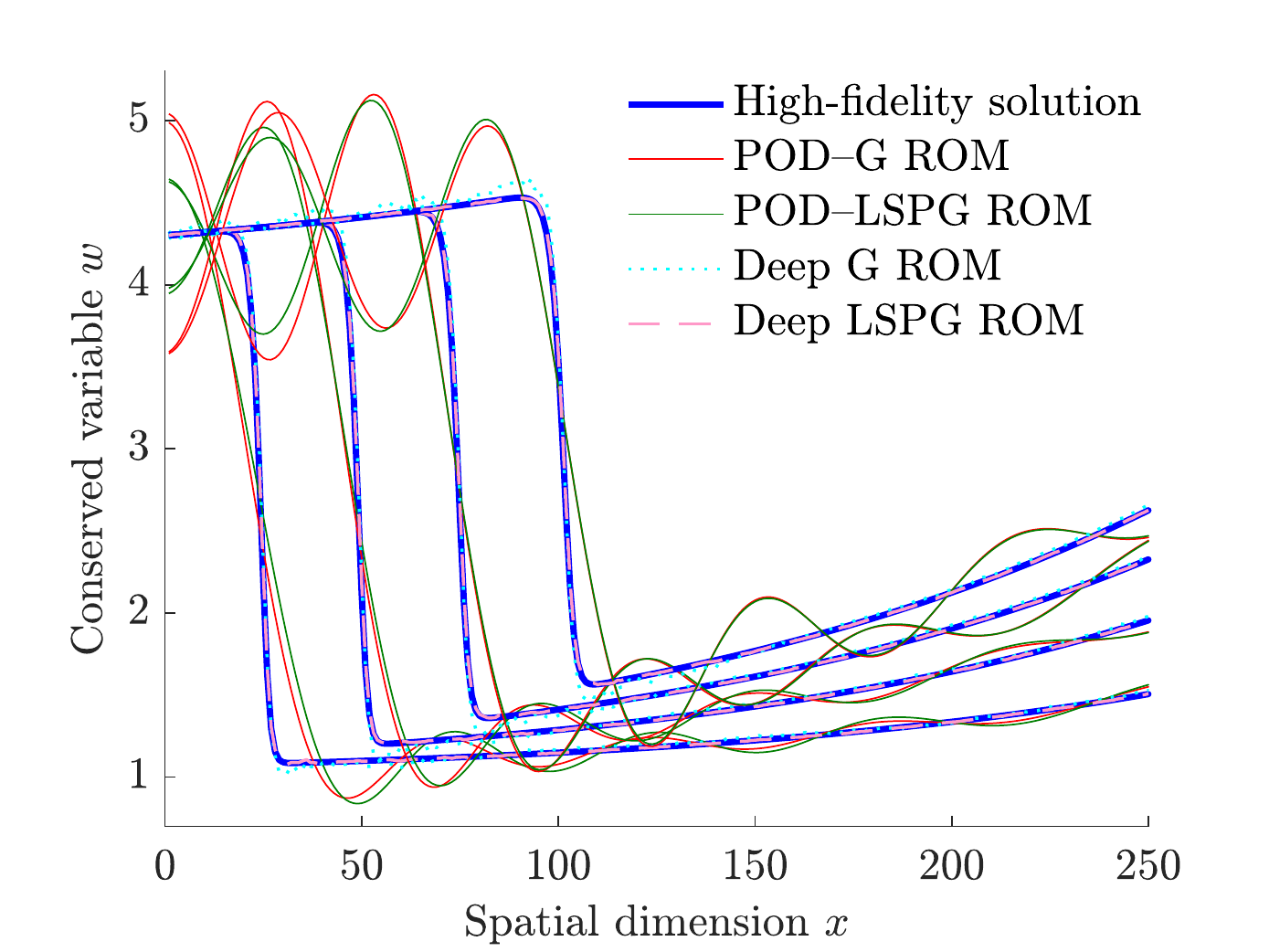}
			\caption{Online-parameter instance $\paramtest^{1} = \left(4.3, 0.021\right)$ with $\dofrom=10$}
  \end{subfigure}\hspace{5mm}
      \begin{subfigure}[b]{0.45 \linewidth}
  \centering
    \includegraphics[width=\linewidth]{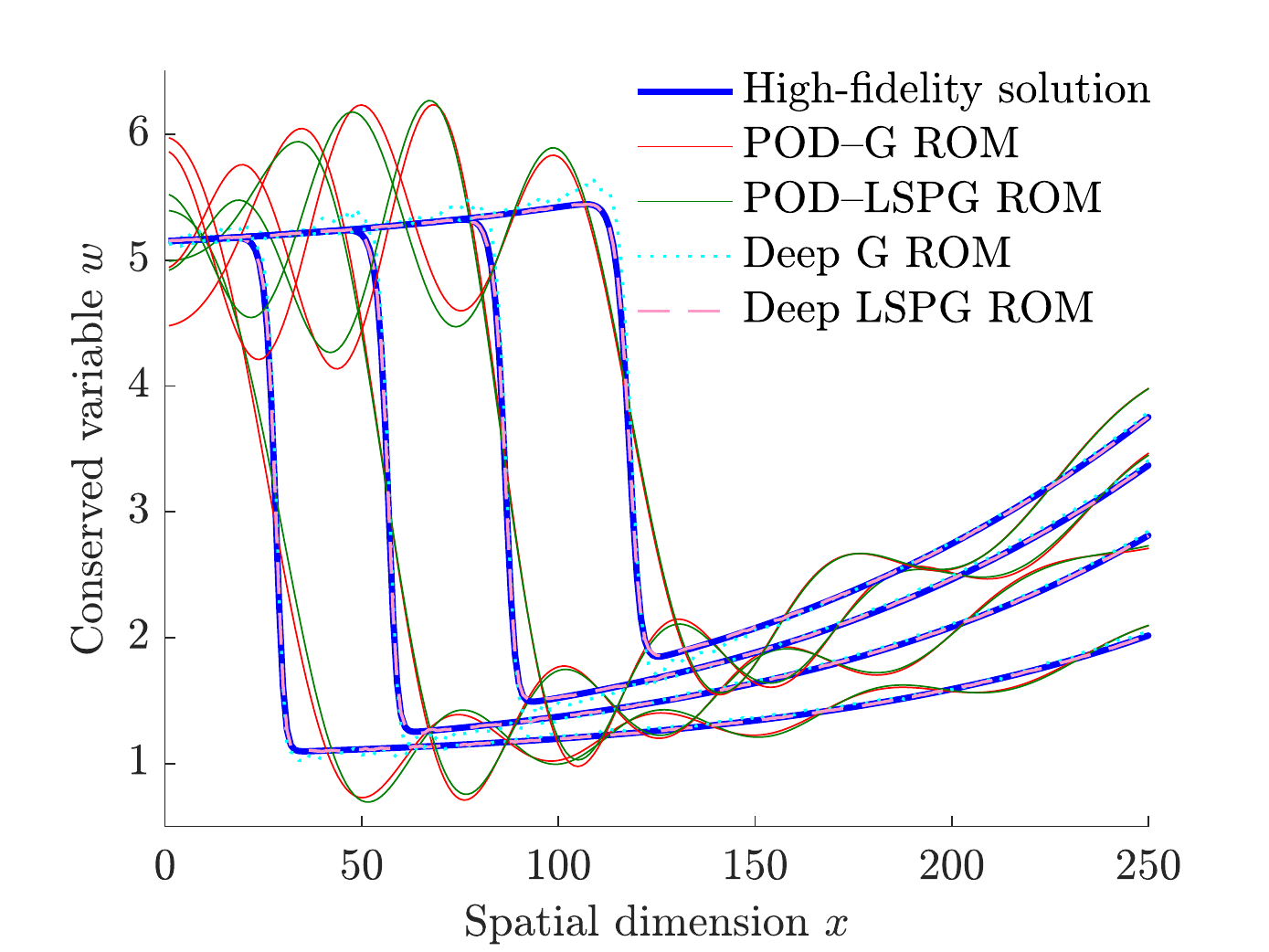}
			\caption{Online-parameter instance $\paramtest^{2} = \left(5.15, 0.0285\right)$ with $\dofrom=10$}
  \end{subfigure}\\
    \begin{subfigure}[b]{0.45 \linewidth}
  \centering
    \includegraphics[width=\linewidth]{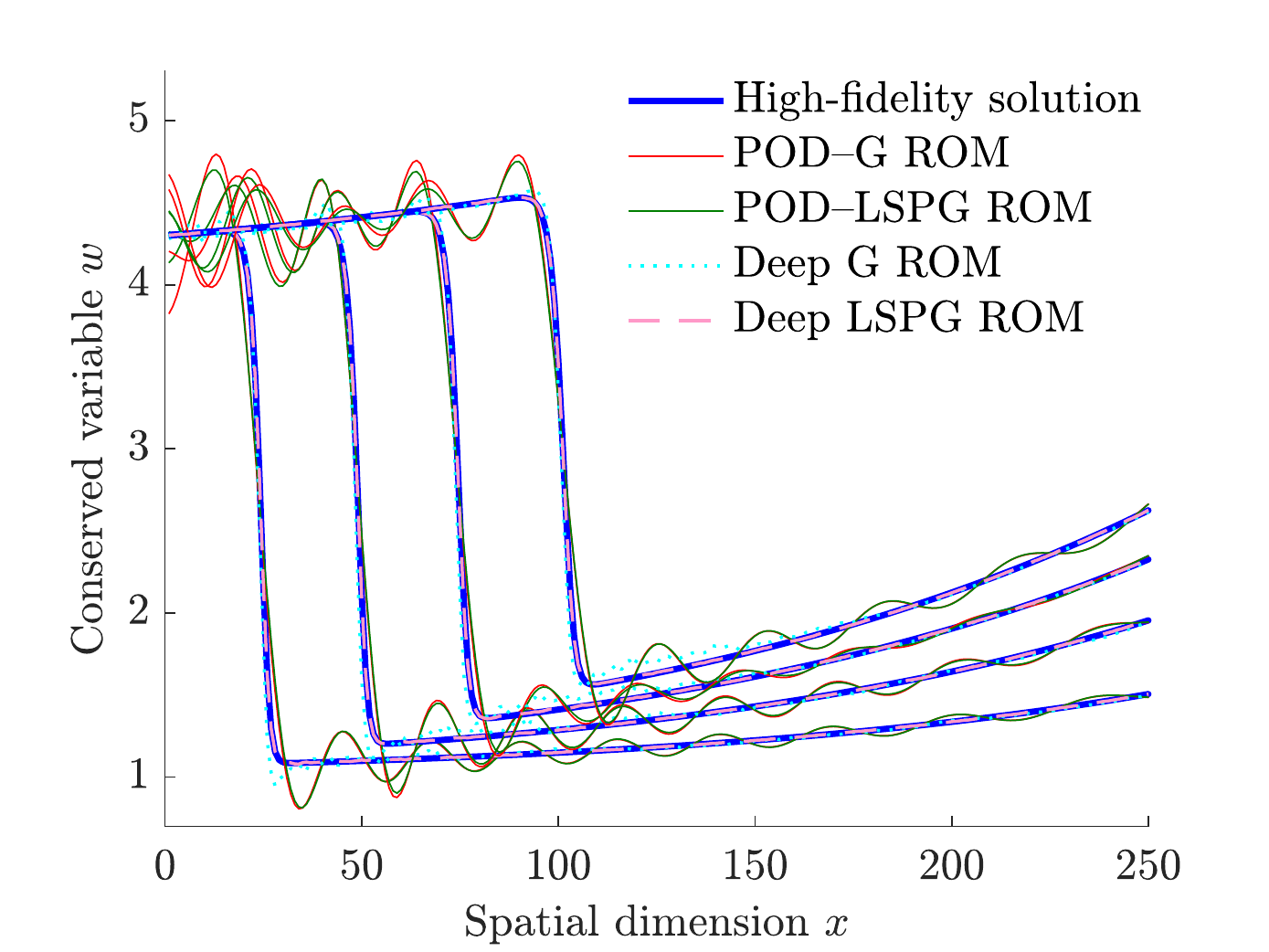}
			\caption{Online-parameter instance $\paramtest^{1} = \left(4.3, 0.021\right)$ with $\dofrom=20$}
  \end{subfigure}\hspace{5mm}
    \begin{subfigure}[b]{0.45 \linewidth}
  \centering
    \includegraphics[width=\linewidth]{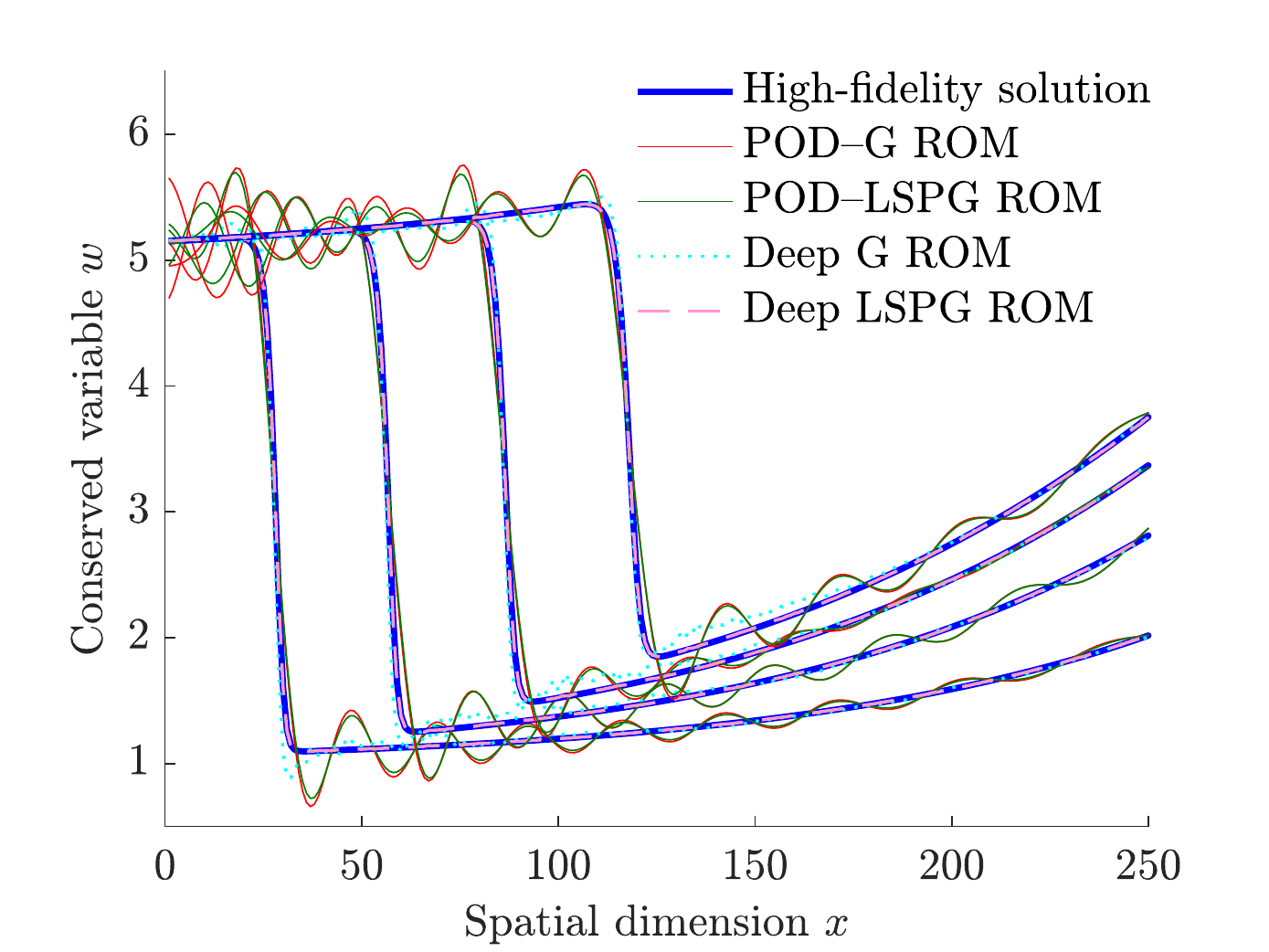}
			\caption{Online-parameter instance $\paramtest^{2} = \left(5.15,
			0.0285\right)$ with $\dofrom=20$}
  \end{subfigure}
\caption{\textit{1D Burgers' equation.} Online solutions at four different time
	instances $t=\{3.5, 7.0, 10.5, 14\}$ computed by solving the FOM
	(High-fidelity solution, solid blue line), POD--ROMs with
	Galerkin projection (POD--G ROM, solid red line) and LSPG projection
	(POD--LSPG ROM, solid green line), and the nonlinear-manifold ROMs equipped
	with the deep convolutional decoder associated with time-continuous residual
	minimization (Deep G ROM, dotted cyan
	line) and with time-discrete residual minimization (Deep LSPG ROM, dashed
	magenta line). All ROMs employ a reduced dimension of  $\dofrom=10$ (left) and $\dofrom = 20$ (right).}
\label{fig:burger_snapshot_fig}
\end{figure}

\begin{figure}[!t]
\centering
	\begin{subfigure}[b]{0.46 \linewidth}
	\centering
     	\includegraphics[width=\linewidth]{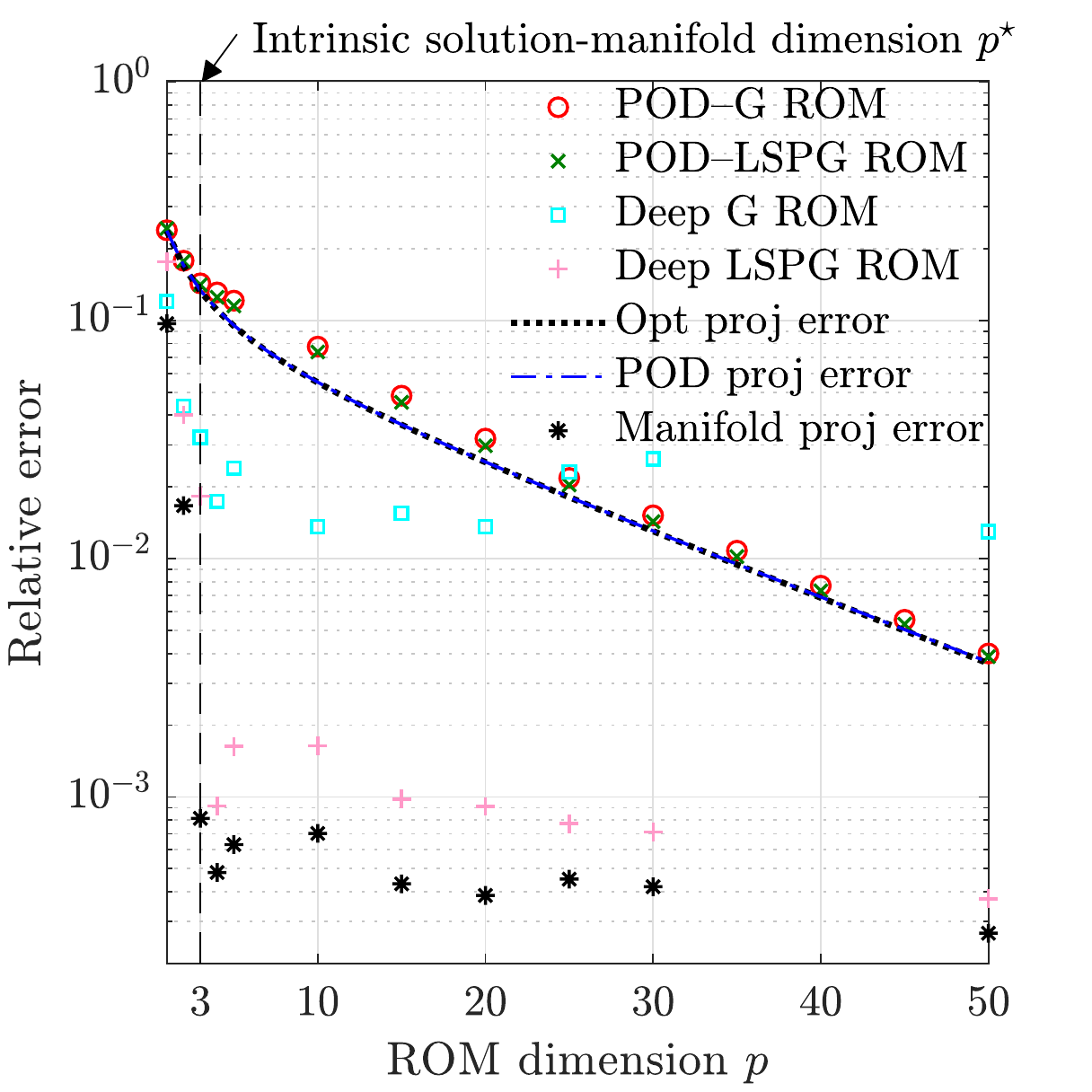}
     	\caption{Online-parameter instance 1, $\paramtest^{1} = \left(4.3, 0.021\right)$}
  	\end{subfigure}
    	\begin{subfigure}[b]{0.46 \linewidth}
	\centering
    	\includegraphics[width=\linewidth]{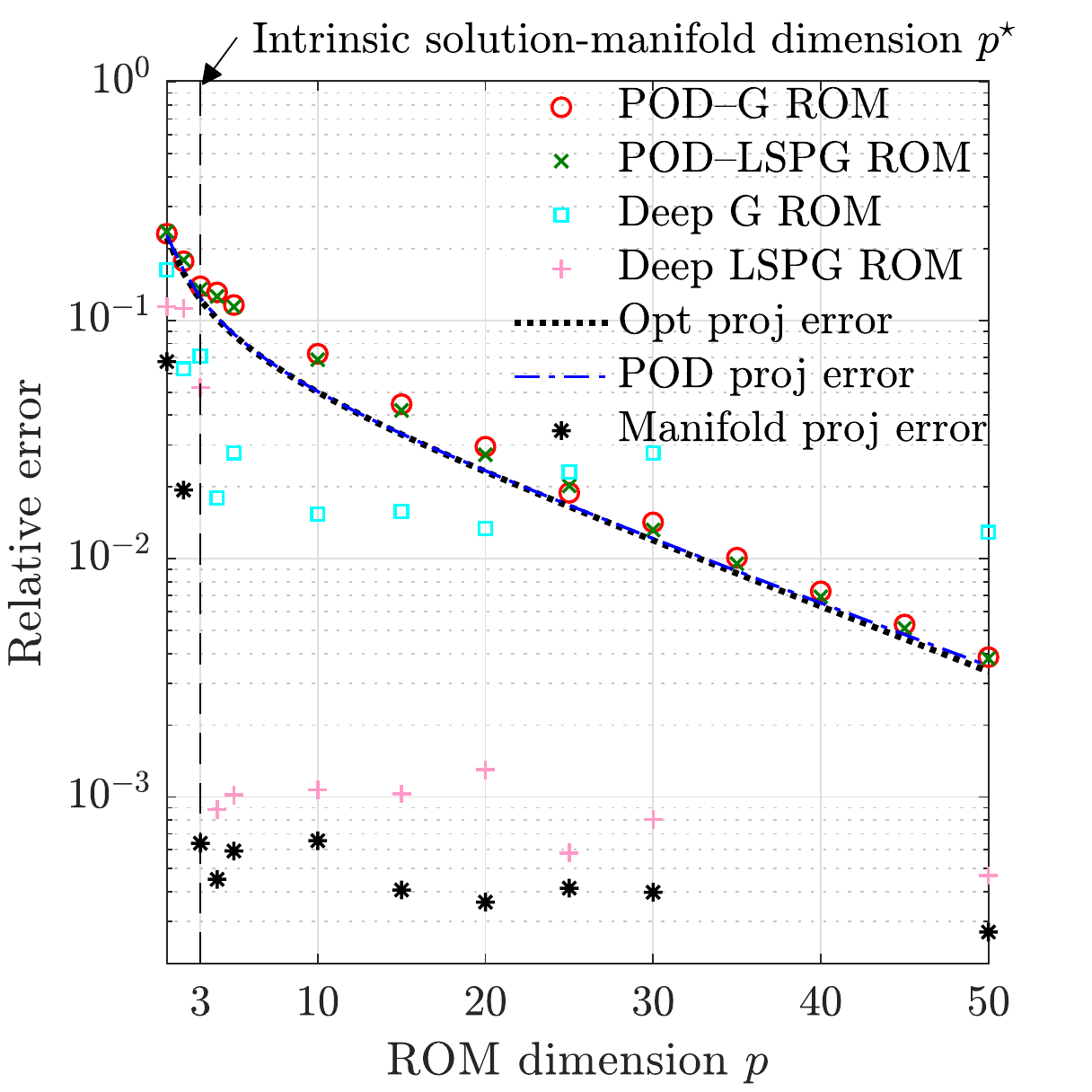}
	\caption{Online-parameter instance 2, $\paramtest^{2} = \left(5.15, 0.0285\right)$}
  	\end{subfigure}
\caption{\textit{1D Burgers' equation.} Relative errors of the ROM solutions
	(Eq.\ \eqref{eq:rel_err})  at online-parameter instances $\paramtest^{1}$
and $\paramtest^{2}$ for varying
	dimensions of the ROMs.
	The figure shows the relative errors of POD--ROMs with
	Galerkin projection (POD--G ROM, red circle) and LSPG projection (POD--LSPG
	ROM, green cross), and the nonlinear-manifold ROMs equipped with the decoder
	with time-continuous optimality (Deep G ROM, cyan square) and time-discrete
	optimality (Deep LSPG ROM, magenta plus sign). The figure also reports the
	projection error (Eq.~\eqref{eq:proj_err}) of the solution 
	onto the  POD basis
	(POD proj error, dashed blue line) and the 
	optimal basis (Opt proj error, \OWN{dotted} black line), \OWN{and reports the manifold projection error of the solution (Eq.~\eqref{eq:disc_opt_problem_optimal}, Manifold proj error, black asterisk)}. The vertical dashed line
	indicates the intrinsic solution-manifold dimension $\intrinsicDim$ (see
	Remark \ref{rem:intrinsic}).}
\label{fig:err_burg}
\end{figure}

For offline training, we set the training-parameter instances to
$\paramspacetrain=\{(4.25 + (1.25/9)i,\ 0.015+(0.015/7)j\OWN{)}\}_{i=0,\ldots,9;\
j=0,\ldots 7}$, resulting in $\ntrain=80$ training-parameter instances.
The restriction operator $\restrictionOp$ and prolongation operator
$\prolongationOp$ correspond to reshaping operators (without interpolation);
the restriction operator reshapes the state vector into a tensor corresponding
to the finite-volume grid such that $\nDim=1$, $\nspaceDim{1} = 256$, and $\nchannel=1$ in definitions \eqref{eq:restrictionOP}
and \eqref{eq:prolongationOp}.
Then, we apply Algorithm
\ref{alg:ae_train} with inputs specified above and the following SGD
hyperparameters: the fraction of snapshots to use for validation
$\nvalidationNNfrac=0.1$; Adam optimizer learning-rate strategy with an initial
uniform learning rate $\learningrate = 10^{-4}$; initial parameters $\nnparam^{(0)}$ computed via \OWN{Xavier initialization \cite{glorot2010understanding}}
for weights and
zero for biases; 
the number of minibatches determined by a fixed
batch size of $\nminibatchArg{\iterToBatch{i}} = 20$, $i=1,\ldots,\nbatch$; 
a maximum number of epochs $\nepoch = 1000$; and early-stopping enforced if
the loss on the validation set fails to decrease over 100 epochs.
 For the online stage, we consider two online-parameter instances $\paramtest^{1} =
 \left(4.3, 0.021\right)$, and  $\paramtest^{2}= \left(5.15, 0.0285\right)$,
 which are not included in $\paramspacetrain$.

Figure \ref{fig:burger_snapshot_fig} reports solutions at four different time
indices computed by using FOM O$\Delta$E and the four considered ROMs.  All
ROMs employ the same reduced dimension of  $\dofrom=10$ in Figure
\ref{fig:burger_snapshot_fig} (left)   and $\dofrom=20$ in Figure
\ref{fig:burger_snapshot_fig} (right).  These results demonstrate that
nonlinear-manifold ROMs \GalerkinNameDeepCap\ and \LSPGNameDeepCap\ produce
extremely accurate solutions, while the linear-subspace ROMs---constructed
using the same training data---exhibit significant errors. This is due to the
fundamental ill-suitedness of linear trial subspaces to advection-dominated
problems.

Figure \ref{fig:err_burg} reports the convergence of the relative error as a
function of reduced dimension $\nstatered$.  
These results illustrate the promise of employing nonlinear trial manifolds.
First, we note that employing a linear trial subspace immediately introduces
significant errors: the projection error onto the optimal basis of dimension
$\dofrom=\intrinsicDim=3$ is over 10\%; the optimal nonlinear trial manifold
(corresponding to the solution manifold) of the same dimension yields zero
error. Even with a reduced dimension $\nstatered=50$, the relative projection
error onto the optimal basis has not yet reached 0.1\%.  Second, we note that 
with a reduced dimension of only $\nstatered=5=\intrinsicDim+2$, the
\LSPGNameDeepCap\ ROM realizes relative errors near 0.1\%, while
linear-subspace ROMs (including the projection error with the optimal basis)
exhibit relative errors near 10\%. Thus, the proposed convolutional
autoencoder is very close to achieving the optimal performance of any
nonlinear trial manifold; the dimension it requires to realize sub-0.1\%
errors is only two larger than the intrinsic solution-manifold dimension
$\intrinsicDim=3$. We also observe that the POD--Galerkin and POD--LSPG ROMs
are also nearly able to achieve optimal performance for linear-subspace ROMs,
as their relative errors are close to the projection error onto the optimal
basis; unfortunately, this error remains quite large relative to the
\GalerkinNameDeep\ and \LSPGNameDeep\ ROMs. This
highlights that the proposed nonlinear-manifold ROMs are able to overcome the
fundamental limitations of linear-subspace ROMs on problems exhibiting a
slowly decaying Kolmogorov $n$-width.  We also observe that the POD basis is
very close to the optimal basis, which implies that the training data are
sufficient to accurately represent the online solution.

Additionally, we note that \LSPGNameDeepCap\ outperforms \GalerkinNameDeepCap\
for this problem, likely due to the fact that the residual-minimization
problem is defined over a finite time step rather \RO{than} time-instantaneously;
similar results have been shown in the case of linear trial subspaces, e.g.,
in Refs.~\cite{carlbergGalDiscOpt}. 

\OWN{
Finally, to illustrate the dependence of the proposed methods on the amount of
training data, we vary the number of training-parameter instances in the set
$\ntrain \in \{5,10,20,40,80,120\}$, and we use the resulting $\nsnap = \ntime
\ntrain$ snapshots to train both the autoencoder for the \GalerkinNameDeep\ and
\LSPGNameDeep\ ROMs and to compute the POD basis for the POD--Galerkin and POD--LSPG
ROMs. We fix the reduced dimension to $\dofrom = 5$. For offline training, we
define the maximum 
number of epochs and the early-stopping strategy in a manner that 
 the Adam optimizer employs the same number of gradient
computations to train each model. Figure
\ref{fig:burger_nparam} reports the relative error of the four ROMs for 
the diferent number of training-parameter instances.
This figure demonstrates that the
proposed \GalerkinNameDeep\ and \LSPGNameDeep\ ROMs yield accurate
results---around 2\% relative error for \GalerkinNameDeep\ and less than 1\%
relative error for \LSPGNameDeep---with only $\ntrain = 5$ parameter
instances. This highlights that---for this example---the proposed methods do
not require an excessive amount of training data relative to standard
POD--based ROMs.  This lack of demand for a large amount of data is likely due
to the use of convolutional layers in the autoencoder, which effectively
employ parameter sharing to reduce significantly the total number of
parameters in the autoencoder, and thus the amount of data needed for
training.
}

\begin{figure}[!h]
\centering
    \begin{subfigure}[b]{0.45 \linewidth}
    \centering
    \includegraphics[width=\linewidth]{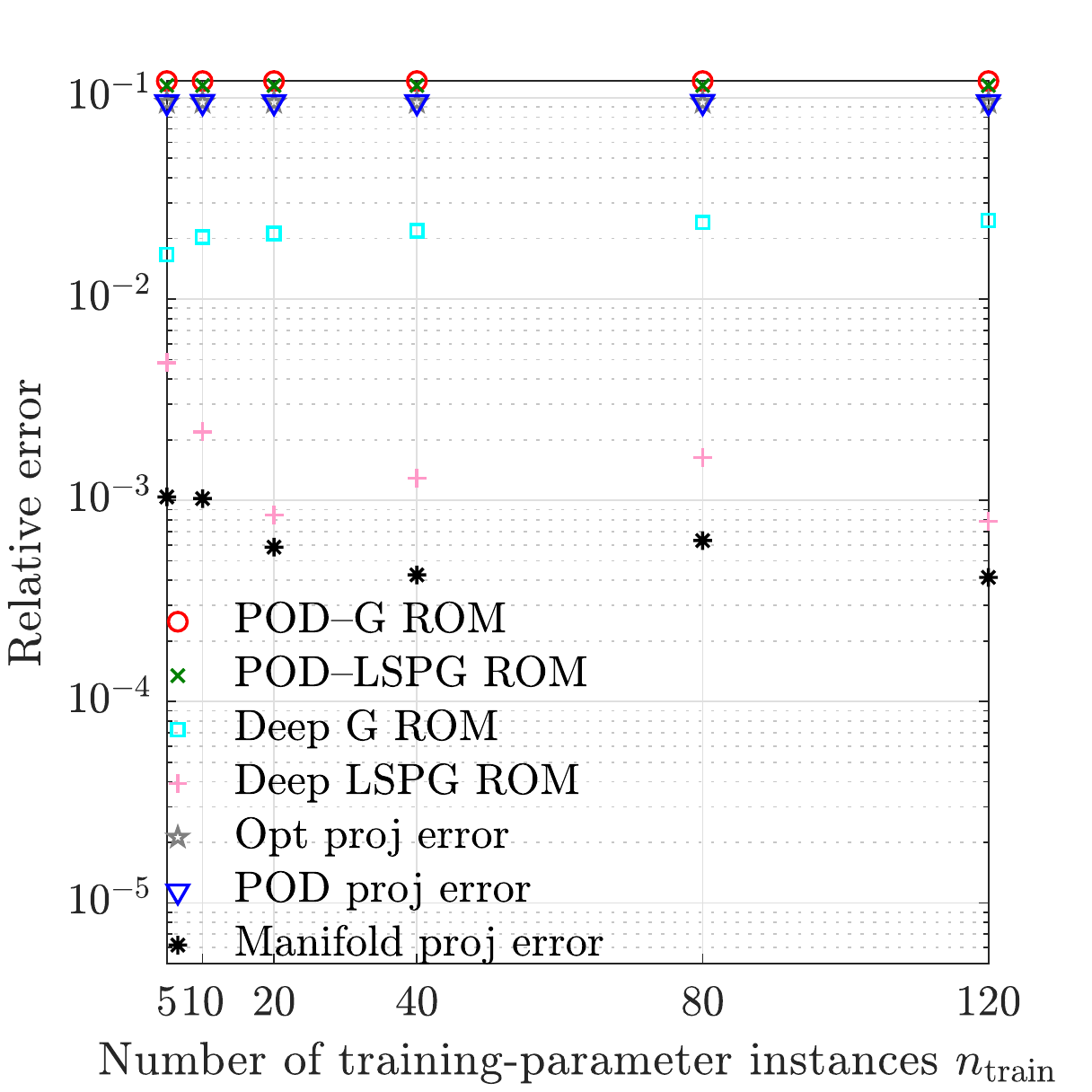}
    \caption{\OWN{Online-parameter instance $\paramtest^{1} = \left(4.3, 0.021\right)$ with $\dofrom=5$}}
    \end{subfigure}
    \begin{subfigure}[b]{0.45 \linewidth}
    \centering
    \includegraphics[width=\linewidth]{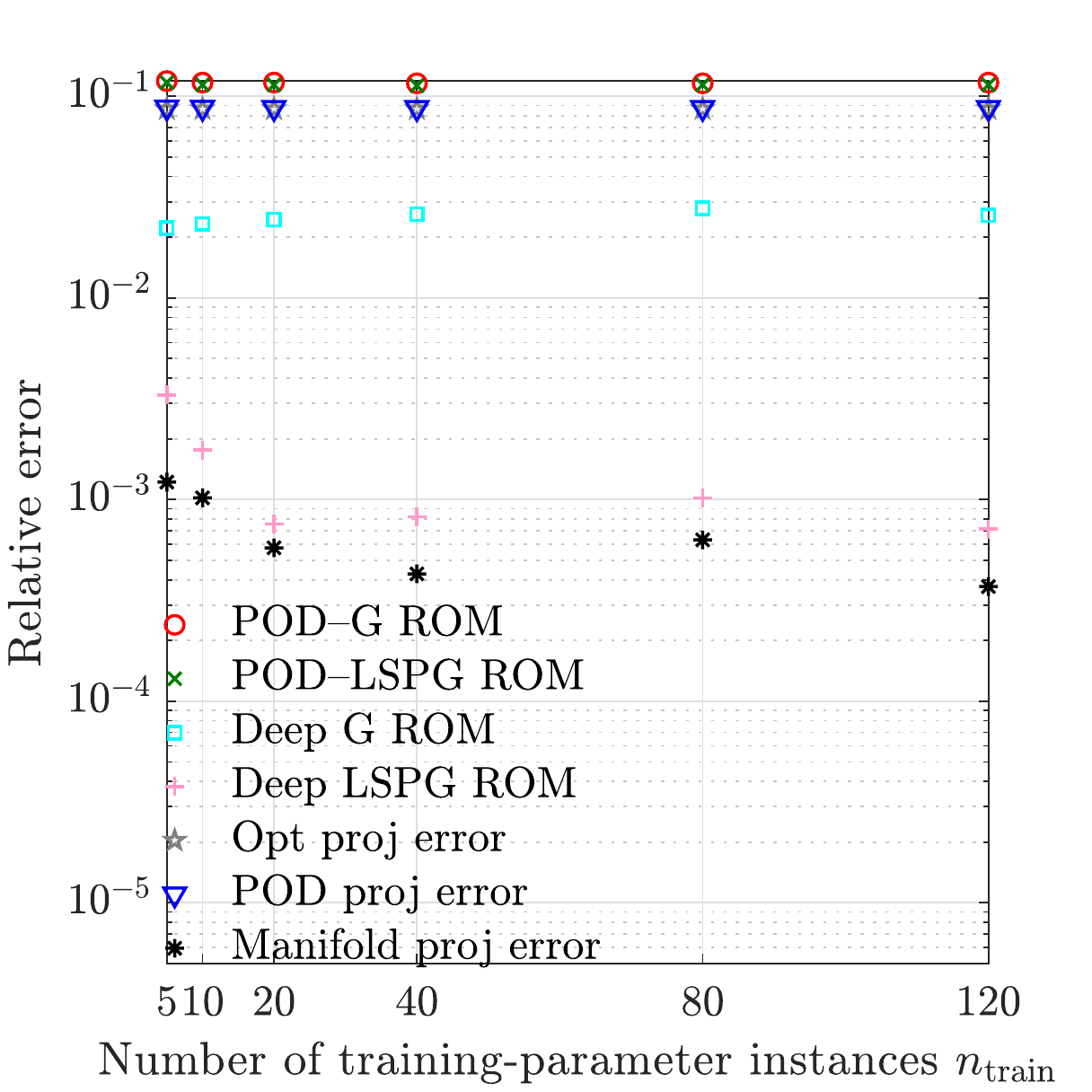}
    \caption{\OWN{Online-parameter instance $\paramtest^{2} = \left(5.15, 0.0285\right)$ with $\dofrom=5$}}
    \end{subfigure}
	\caption{\OWN{\textit{1D Burgers' equation: varying amount of training data.} Relative errors of the ROM
	solutions (Eq.\ \eqref{eq:rel_err})  at online-parameter instances
$\paramtest^{1}$ and
	$\paramtest^{2}$ for a varying the number training-parameter instances
	$\ntrain \in \{5,10,20,40,80,120\}$. The figure shows the relative errors of POD--ROMs with
	Galerkin projection (POD--G ROM, red circle) and LSPG projection (POD--LSPG
	ROM, green cross), and the nonlinear-manifold ROMs equipped with the decoder
	with time-continuous optimality (Deep G ROM, cyan square) and time-discrete
	optimality (Deep LSPG ROM, magenta plus sign). 
	The figure also reports the
	projection error (Eq.~\eqref{eq:proj_err}) of the solution 
	onto the  POD basis
	(POD proj error, blue triangle) and the 
	optimal basis (Opt proj error, gray pentagram), and reports the manifold projection error of the solution (Eq.~\eqref{eq:disc_opt_problem_optimal}, Manifold proj error, black asterisk).}}
\label{fig:burger_nparam}
\end{figure}

\subsection{Chemically reacting flow}
We now consider a model of the reaction of a premixed $\fuel$-air
flame at constant uniform pressure 
\cite{buffoni2010projection}. The evolution of the flame is modeled by the
nonlinear convection--diffusion--reaction equation
\begin{equation}\label{eq:adr}
\frac{\partial \eqvarvec(\xtwod, t; \param)}{\partial t} =  \nabla \cdot (\diffusivity \nabla \eqvarvec(\xtwod, t;\param)) - \velocity \cdot \nabla \eqvarvec(\xtwod, t;\param) + \reactionvec(\eqvarvec(\xtwod, t;\param);\param) \qquad \text{in } \spatialspace \times \paramspace \times [0, T], 
\end{equation}
where $\nabla$ denotes the gradient with respect to physical space,
$\diffusivity$ denotes the molecular diffusivity, $\velocity$ denotes the
velocity field, and 
\begin{equation*}
\eqvarvec(\xtwod, t;\param) \equiv [\eqvar_T(\xtwod, t;\param), \eqvar_{\fuel}(\xtwod, t;\param), \eqvar_{\oxi}(\xtwod, t;\param)), \eqvar_{\product}((\xtwod, t;\param))]\tran \in \mathbb{R}^{4}
\end{equation*}
denotes the thermo-chemical composition vector consisting of the temperature
$\eqvar_{T}(\xtwod, t;\param)$
and the mass fractions of chemical species
$\fuel, \oxi$,
and $\product$, i.e.,
$\eqvar_i(\xtwod, t;\param)$ for $i \in \{\fuel,
\oxi, \product \}$. 
The nonlinear reaction source term $\reactionvec(\eqvarvec(\xtwod,
t;\param);\param) \equiv [\reaction_{T}(\eqvarvec;\param),
\reaction_{\fuel}(\eqvarvec;\param), \reaction_{\oxi}(\eqvarvec;\param),
\reaction_{\product}(\eqvarvec;\param)]\tran$ is of Arrhenius type and
is defined as
\begin{equation*}
\begin{split}
\reaction_{T}(\eqvarvec;\param) &= Q \reaction_{\product}(\eqvarvec;\param)\\
\reaction_{i}(\eqvarvec;\param) &= -\nu_i\left(
	\frac{W_i}{\rho}\right)\left(\frac{\rho \eqvar_{\fuel}}{W_{\fuel}}
	\right)^{\nu_{\fuel}} \left(\frac{\rho \eqvar_{\oxi}}{W_{\oxi}}
	\right)^{\nu_{\oxi}} A \exp\left(- \frac{E}{R \eqvar_T} \right), \quad i
	\in\{
		\fuel, \oxi, \product\},
\end{split}
\end{equation*}
where $(\nu_{\fuel}, \nu_{\oxi}, \nu_{\product}) = (2,1,-2)$ denote
stoichiometric coefficients, $(W_{\fuel}, W_{\oxi}, W_{\product}) = (2.016,
31.9, 18)$ denote molecular weights with units g$\cdot$mol$^{-1}$, $\rho =
1.39 \times 10^{-3}$ g $\cdot$ cm$^{-3}$ denotes the density mixture, $R =
8.314$ J $\cdot$ mol$^{-1} \cdot$ K$^{-1}$ denotes the universal gas constant,
and $Q = 9800$K denotes the heat of the reaction. The $\nparam=2$ parameters
correspond to $\param=(A,E)$, which 
are the
pre-exponential factor $A$ and the activation energy $E$; we set the
corresponding parameter domain to
$\paramspace = [2.3375 \times 10^{12}, 6.5\times 10^{12}] \times
[5.625 \times 10^3,9\times 10^{3}]$.
\RT{Thus, }the intrinsic solution-manifold dimension is
$\intrinsicDim=3$ (see Remark \ref{rem:intrinsic}).
We set the molecular diffusivity to $\kappa
= 2\;$cm$^2\cdot$s$^{-1}$, and the velocity field set to be constant and
divergence-free with $\velocity = [50\; \text{cm}\cdot\text{s}^{-1},\; 0]^T$.
We set the final time to $T=0.06\;\text{s}$.

\begin{figure}[h]
\centering
\includegraphics{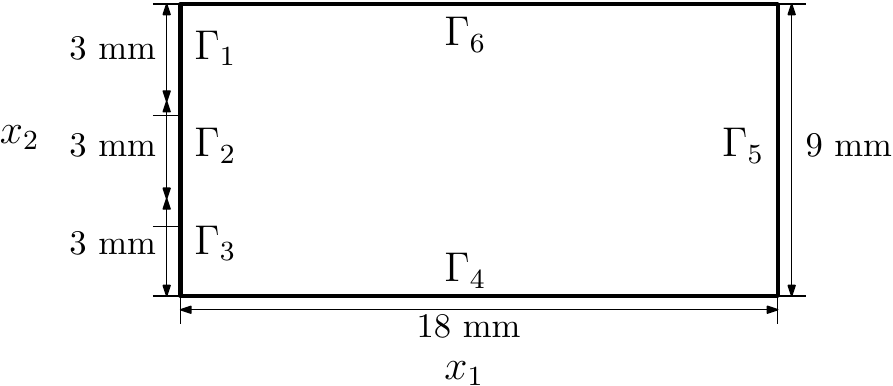}
\caption{\textit{Chemically reacting flow}. The geometry of the spatial domain $\spatialspace$.}
\label{fig:geo_chem}
\end{figure}

Figure \ref{fig:geo_chem} reports the geometry of the spatial domain.  On the
inflow boundary $\Gamma_2$, we impose Dirichlet boundary conditions
$\eqvar_{\fuel}=0.0282$, $\eqvar_{\oxi} = 0.2259$ and $\eqvar_{\product}=0$
for the chemical-species mass fractions and $\eqvar_{T}=950$K for the
temperature. On boundaries $\Gamma_1$ and $\Gamma_3$, we impose homogeneous
Dirichlet boundary conditions for the chemical-species mass fractions, and we
set the temperature $\eqvar_T = 300\;\text{K}$. On $\Gamma_4, \Gamma_5,$ and
$\Gamma_6$, we impose  homogeneous Neumann conditions on the temperature and
mass fractions. We consider a uniform initial condition corresponding to a
domain that is empty of chemical species such that
$\eqvar_{\fuel}=\eqvar_{\oxi}=\eqvar_{\product}=0$ and we set the temperature
to $\eqvar_T = 300\;\text{K}$.

We employ a finite-difference method 
with 65 grid points in the horizontal direction and 32 grid
points in the vertical direction
to spatially discretize Eq.\ \eqref{eq:adr},
which results in a system of parameterized ODEs of the form
\eqref{eq:govern_eq} with $\dofhf = 8192$ degrees of freedom.

For time discretization, we employ the second order backward difference scheme
(BDF2), which corresponds to a linear multistep scheme with 
$k=2$, $\alpha_0=1$, $\alpha_1=-\frac{4}{3}$, $\alpha_2=\frac{1}{3}$,
$\beta_0=\frac{2}{3}$, and $\beta_1 = \beta_2 = 0$ in Eq.\
\eqref{eq:disc_res}.  We consider a uniform time step $\timestep = 10^{-4}$,
resulting in $\nseq = 600$ time instances. 

For offline training, we set the training-parameter instances to
$\paramspacetrain=\{(2.3375 \times 10^{12} + (3.2725\times 10^{12}\OWN{/7)i},\ 5.625
\times 10^3+(3.375\times 10^3/7)j\OWN{)}\}_{i=0,\ldots,7;\
j=0,\ldots 7}$, resulting in $\ntrain=64$ training-parameter instances.  The
restriction operator $\restrictionOp$ and prolongation operator
$\prolongationOp$ correspond to reshaping operators (without interpolation);
the \OWN{restriction} operator reshapes the state vector into a tensor corresponding
to the finite-difference grid such that $\nDim=2$, $\nspaceDim{1} = 64$,
$\nspaceDim{2}=32$, and $\nchannel=4$ in definitions \eqref{eq:restrictionOP}
and \eqref{eq:prolongationOp}. We emphasize that each of the 4 unknown
variables is considered a different channel for the input data.
Then, we apply Algorithm \ref{alg:ae_train} with inputs specified above and
the 
following SGD
hyperparameters: the fraction of snapshots to use for validation
$\nvalidationNNfrac=0.1$; Adam optimizer learning-rate strategy with an initial
uniform learning rate $\learningrate = 10^{-4}$; initial parameters
parameters $\nnparam^{(0)}$ computed via He initialization \cite{he2015delving}
 for weights and
zero for biases; 
number of minibatches determined by a fixed
batch size of $\nminibatchArg{\iterToBatch{i}} = 20$, $i=1,\ldots,\nbatch$; 
a maximum number of epochs $\nepoch = 5000$; and early-stopping enforced if
the loss on the validation set fails to decrease over 500 epochs.
For the online stage, we consider \RT{parameter instances
$\paramtest^{1} =
(2.5\times10^{12},\, 5.85\times10^{3})$ and $\paramtest^{2} =
(3.2\times10^{12},\, 7.25\times10^{3})$ that are} not included in
$\paramspacetrain$.

\begin{figure}[!t]
\centering
\includegraphics[width=0.475\linewidth]{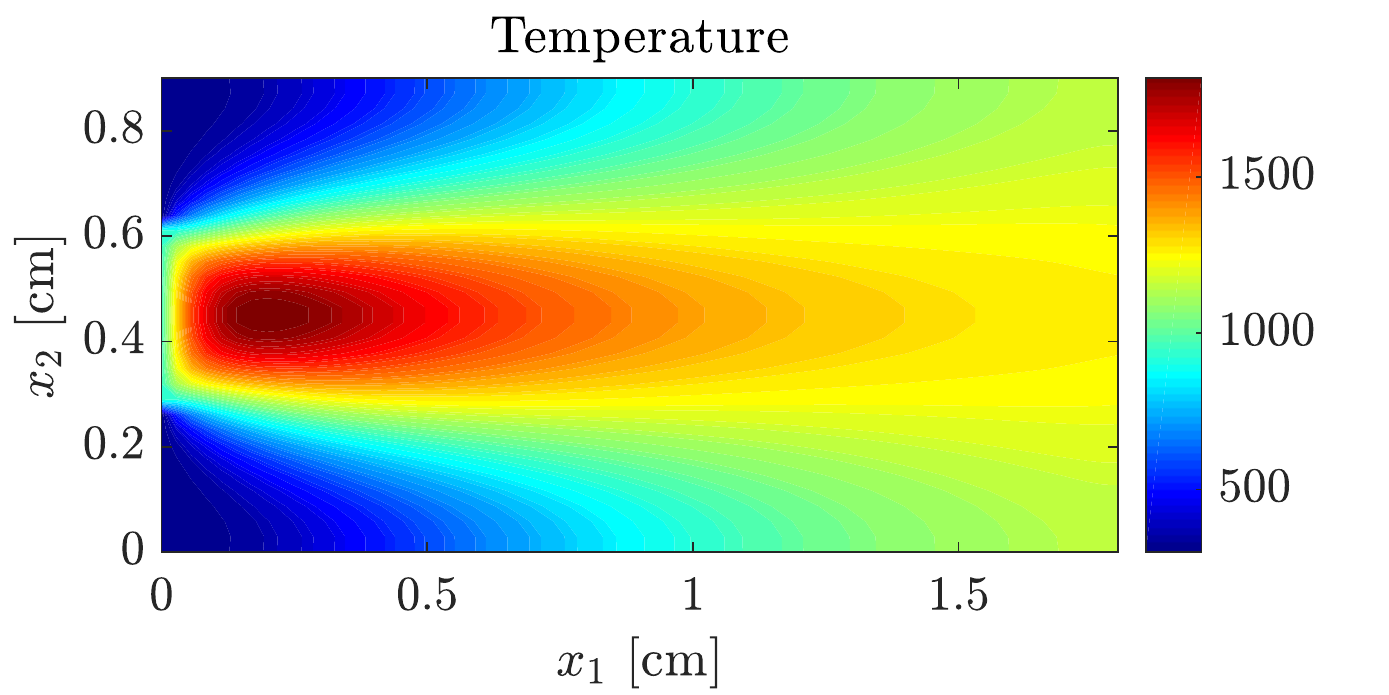}
\includegraphics[width=0.475\linewidth]{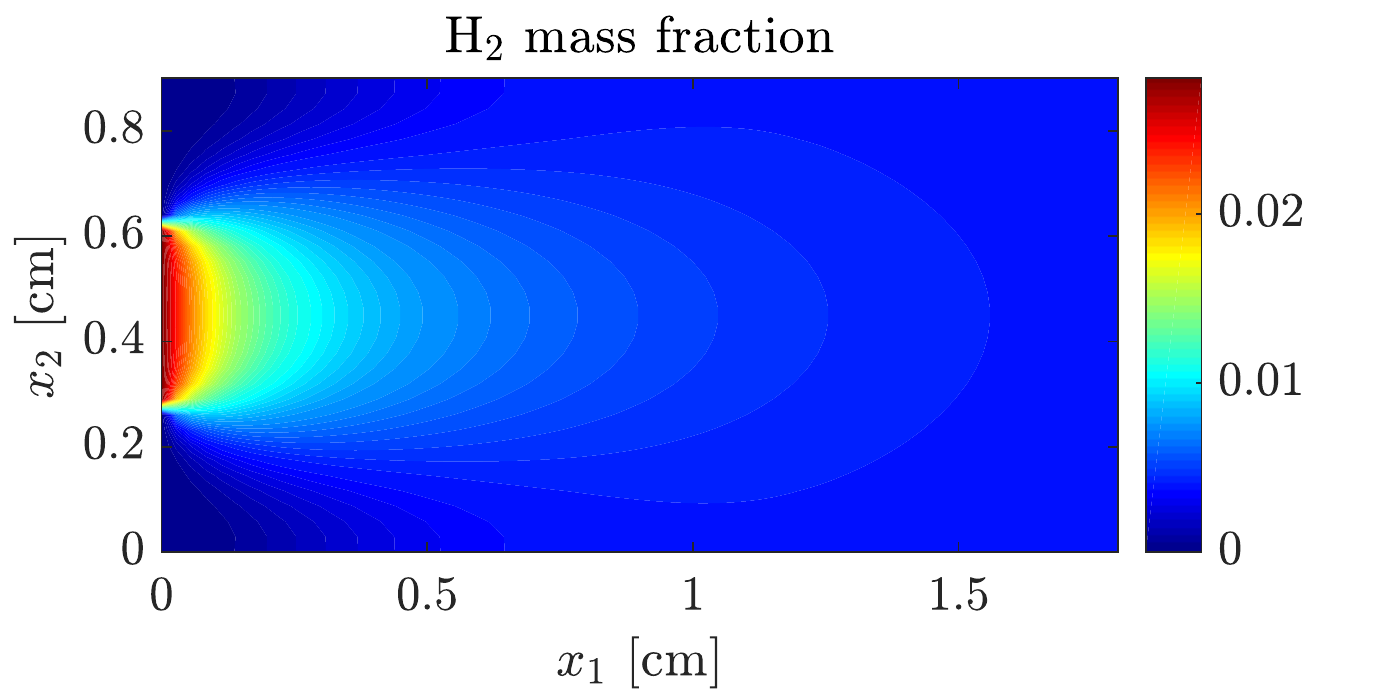}\vspace{3mm}\\
\includegraphics[width=0.475\linewidth]{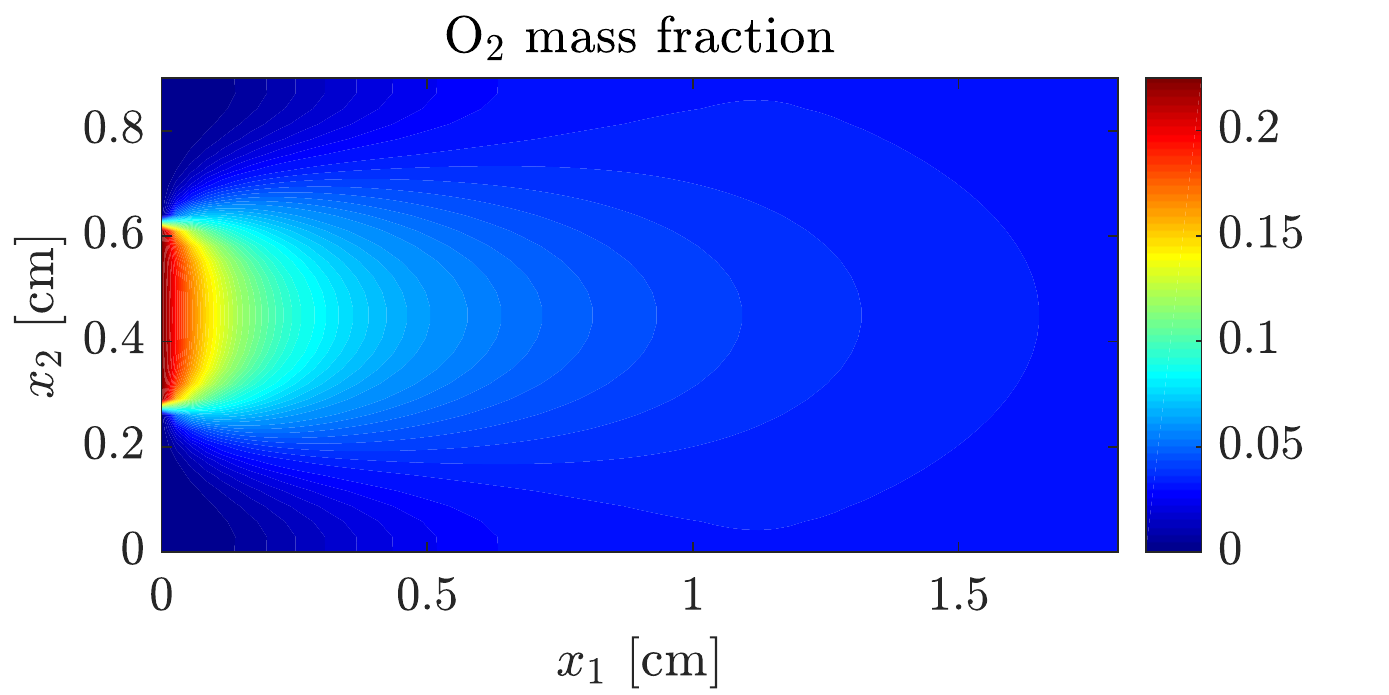}
\includegraphics[width=0.475\linewidth]{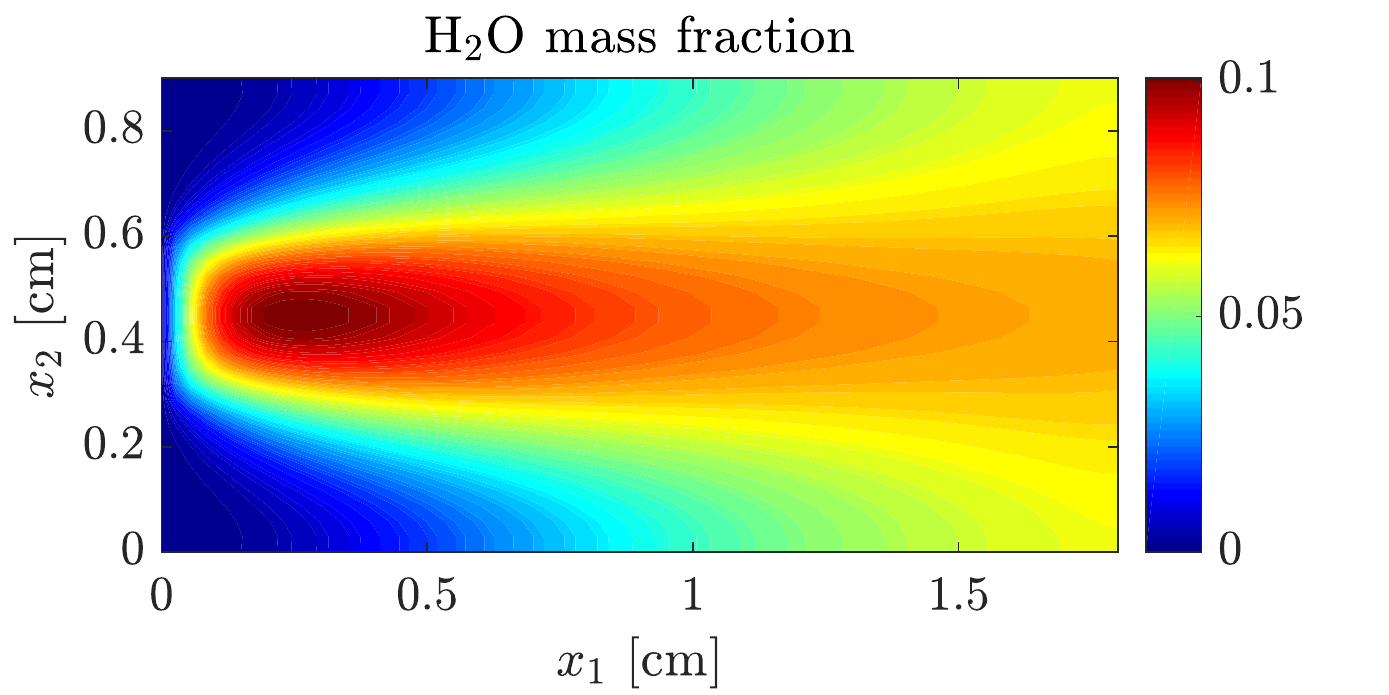}
\caption{\textit{Chemically reacting flow}. FOM solutions of the temperature,
	mass fractions of $\fuel$, $\oxi$, and $\product$ at online-parameter
	instance $\paramtest^{1} =  (2.5\times10^{12},\, 5.85\times10^{3})$ at $t=0.06$.}
\label{fig:fom_param1}
\end{figure}

\begin{figure}[!t]
\centering
\includegraphics[width=0.475\linewidth]{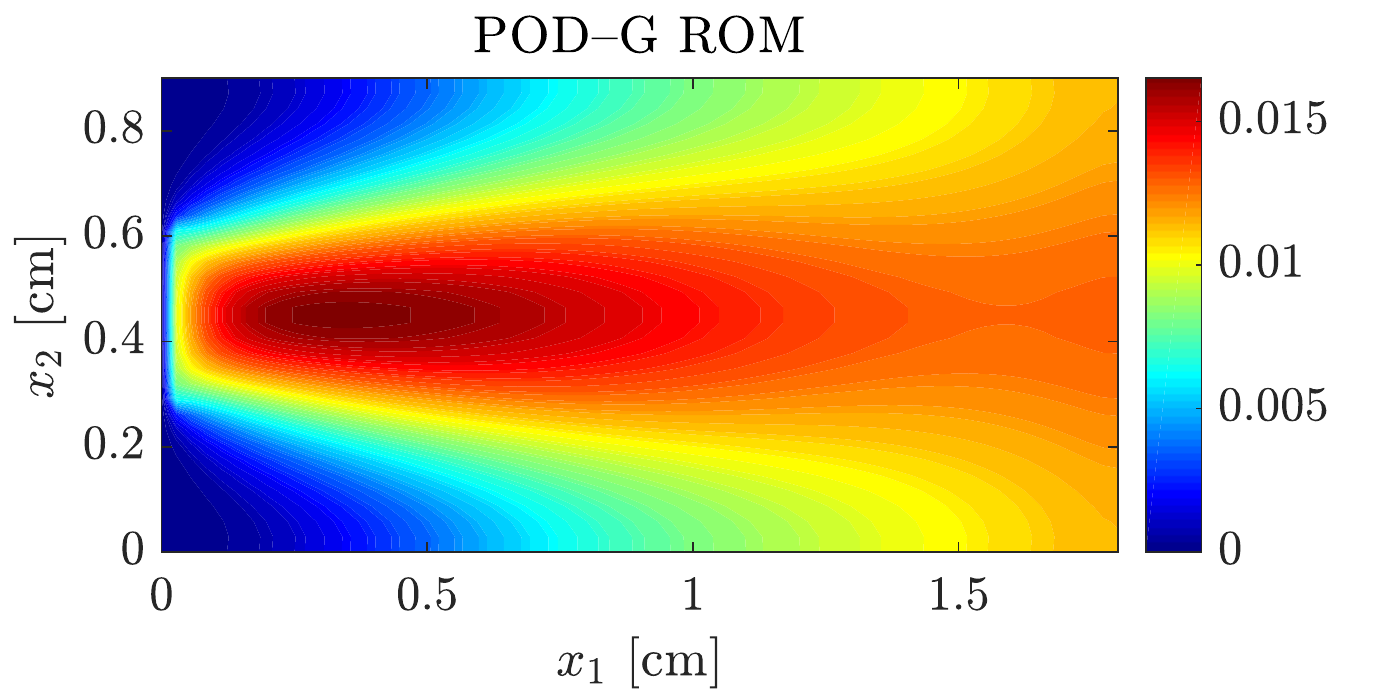}
\includegraphics[width=0.475\linewidth]{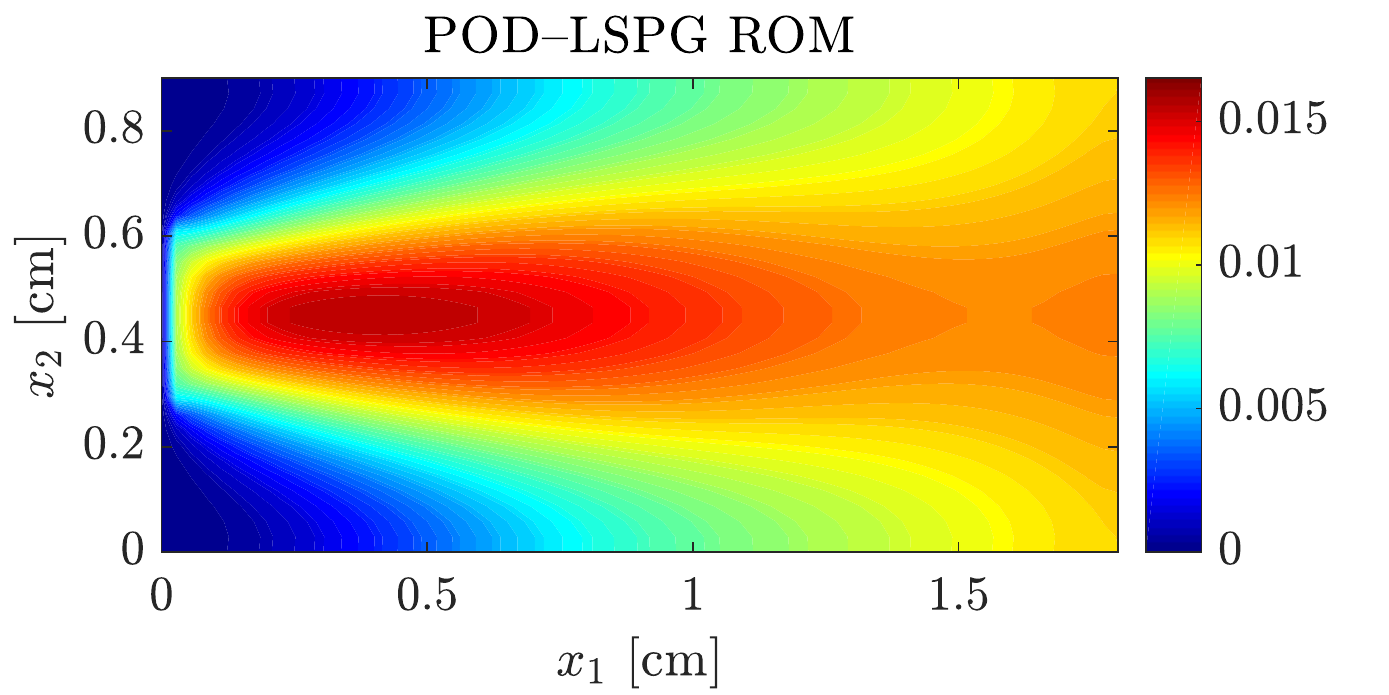}\vspace{3mm}\\
\includegraphics[width=0.475\linewidth]{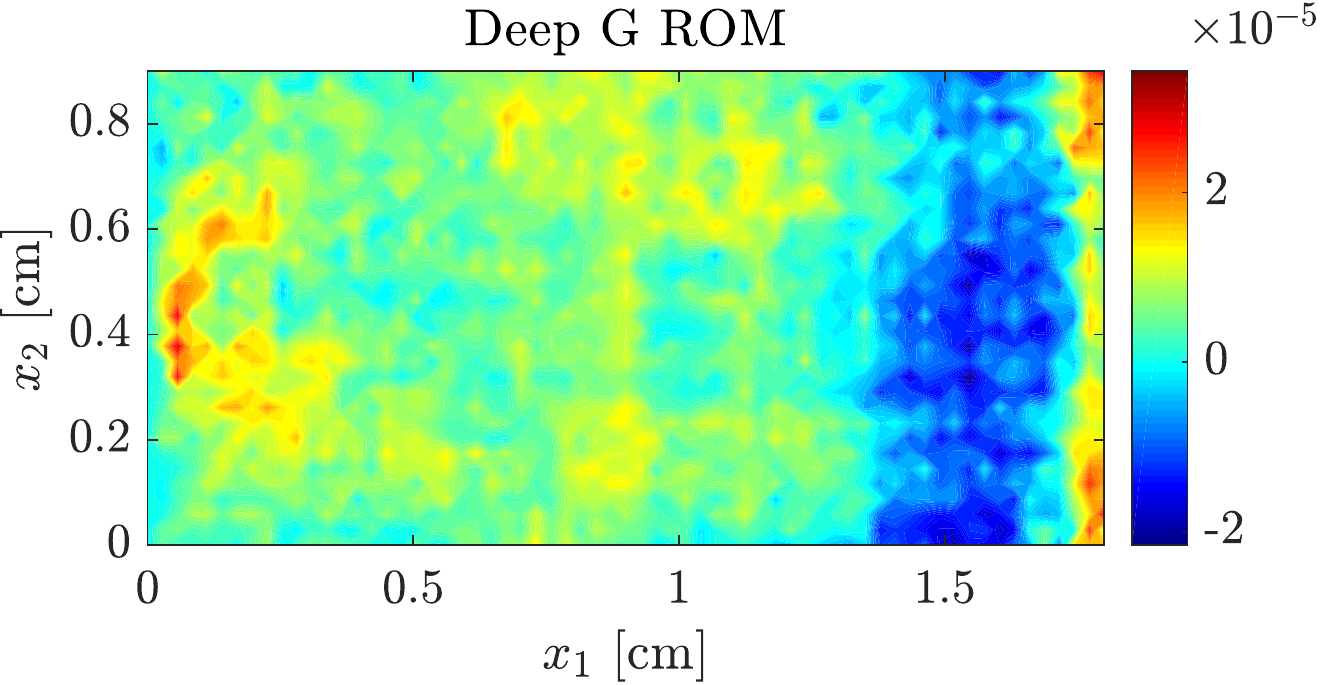}
\includegraphics[width=0.475\linewidth]{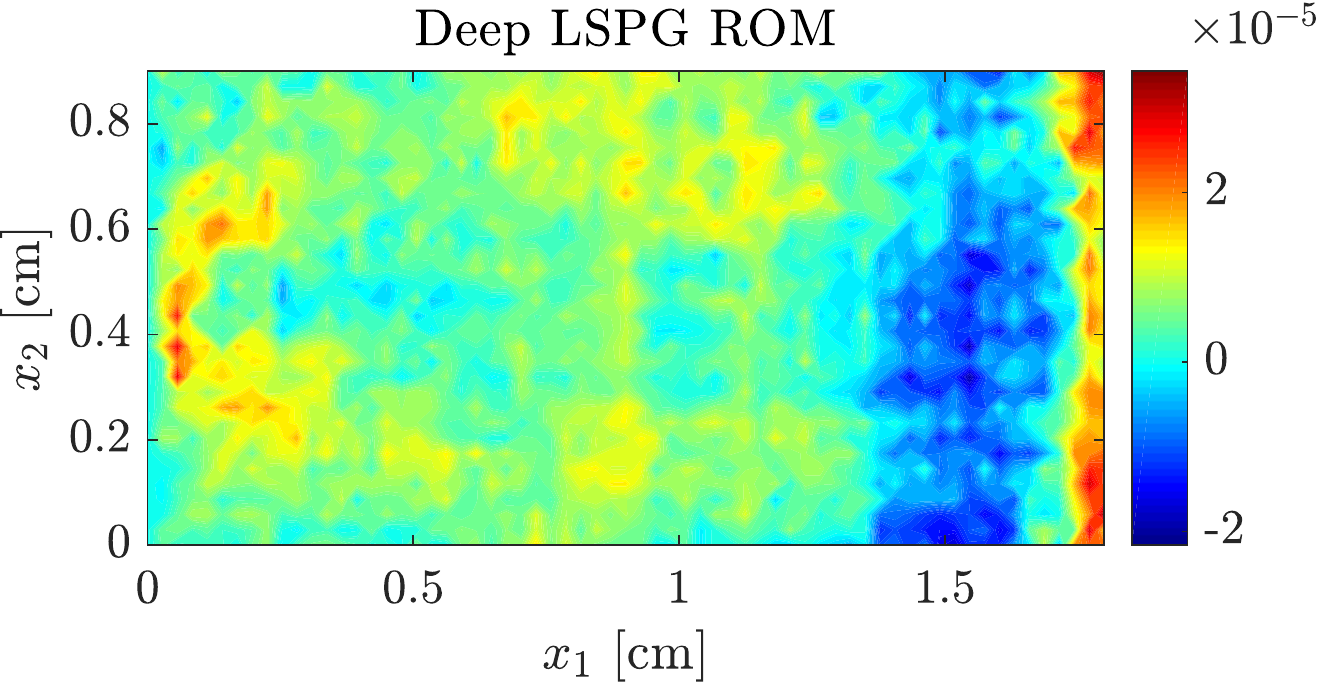}\\
\caption{\textit{Chemically reacting flow}. Relative errors of the temperature
	solution computed using ROMs at online-parameter instance $\paramtest^{1} =
	(2.5\times10^{12},\, 5.85\times10^{3})$ at $t=0.06$. The dimensions of the
	generalized coordinates in both linear-subspace ROMs (top row) and
	nonlinear-manifold ROMs (bottom row) are the same $\nstatered=3$. The
	colormaps on the each row are in the same scale.}
\label{fig:rom_T_param1}
\end{figure}

Figure \ref{fig:fom_param1} reports the FOM solution for online-parameter
instance $\paramtest^{1}$ and final time $t=T=0.06\;\text{s}$. Figures
\ref{fig:rom_T_param1} and \ref{fig:rom_H2_param1} report relative errors of
the temperature solutions and the $\fuel$ mass fraction solutions  at the
final time $t=T$ computed using all considered ROMs with the same reduced
dimension $\nstatered=3$. Again, we observe that the proposed
nonlinear-manifold ROMs produce significantly lower errors as compared with
the linear-subspace ROMs.  Figure \ref{fig:err_adr} reports the convergence of
the relative error as a function of reduced dimension $\nstatered$ along with
the projection errors Eq.~\ref{eq:proj_err} onto the POD basis and the optimal
basis.  This figure again shows that the proposed manifold ROMs
significantly outperform the linear-subspace ROMs, as well as the projection
error onto both the POD and optimal
bases. 
First, we observe that employing a linear trial subspace again introduces
significant errors, as the projection error onto the optimal basis of dimension
$\dofrom=\intrinsicDim=3$ is around 5\%; the optimal nonlinear trial manifold
(corresponding to the solution manifold) of the same dimension yields zero
error. 
Second, we note that for a reduced dimension of
only $\nstatered=\intrinsicDim=3$, both the \GalerkinNameDeepCap\ and \LSPGNameDeepCap\ ROMs
yield relative errors of less than 0.1\%, while the projection errors onto the
POD and optimal bases exceed 5\%, and the linear-subspace ROMs
yield relative errors in excess of 40\%. 
Thus, the proposed convolutional
autoencoder nearly achieves optimal performance, as 
the dimension it requires to realize sub-0.1\%
errors is exactly equal to the intrinsic solution-manifold dimension
$\intrinsicDim=3$.

We also observe that there is a significant gap between the performance of the
linear-subspace ROMs and the POD projection error; this gap is attributable to
the closure problem.  We also observe that---in contrast with the previous
example---there is a non-trivial gap between the POD projection error and the
optimal-basis projection error, which suggests that the online solution is
less well represented by the training data than in the previous case.

Additionally, we observe that for a reduced dimension of $\nstatered=10$, the
projection error onto the optimal basis begins to become smaller than the
relative error of the proposed \GalerkinNameDeepCap\ and \LSPGNameDeepCap\
ROMs; however this occurs for (already small) errors less than 0.1\%. The
error saturation of the \GalerkinNameDeepCap\ and \LSPGNameDeepCap\ ROMs is
likely due to the fact that the online solution cannot be perfectly
represented using the training data, as illustrated by the gap between the
projection errors associated with POD and the optimal basis.

\begin{figure}[!t]
\centering
\includegraphics[width=0.475\linewidth]{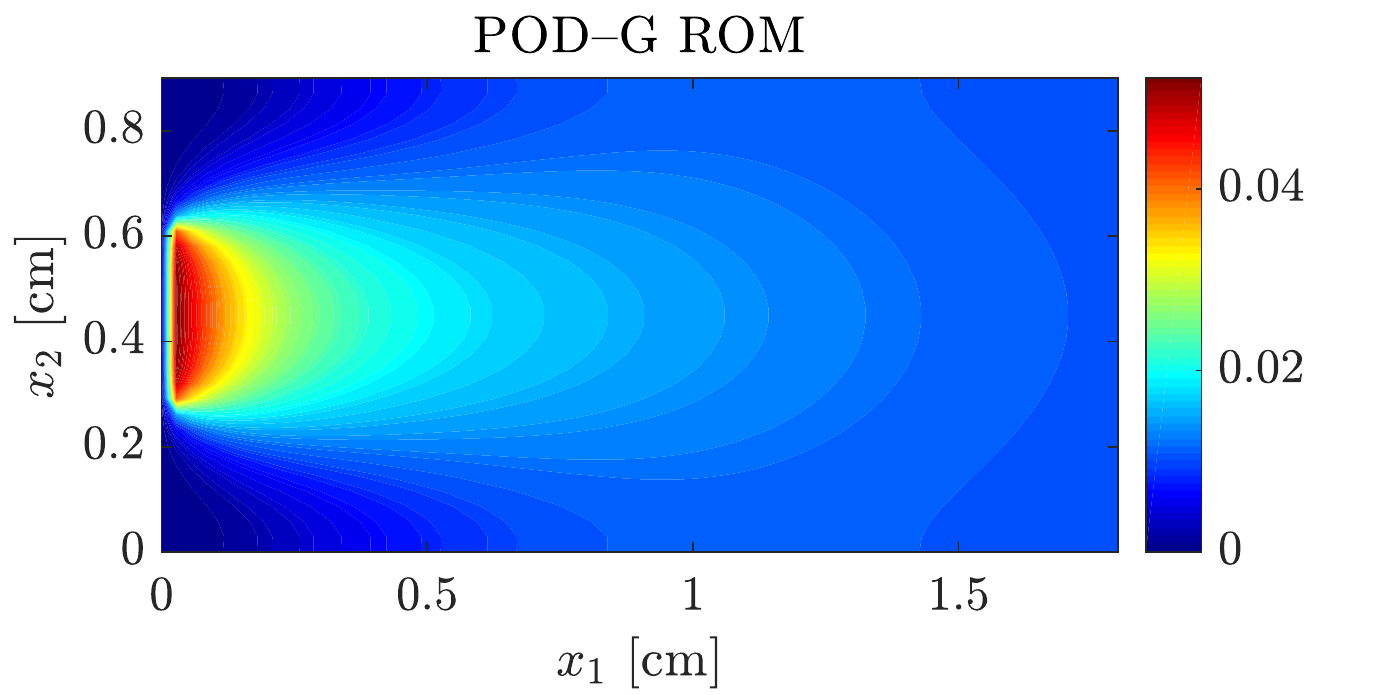}
\includegraphics[width=0.475\linewidth]{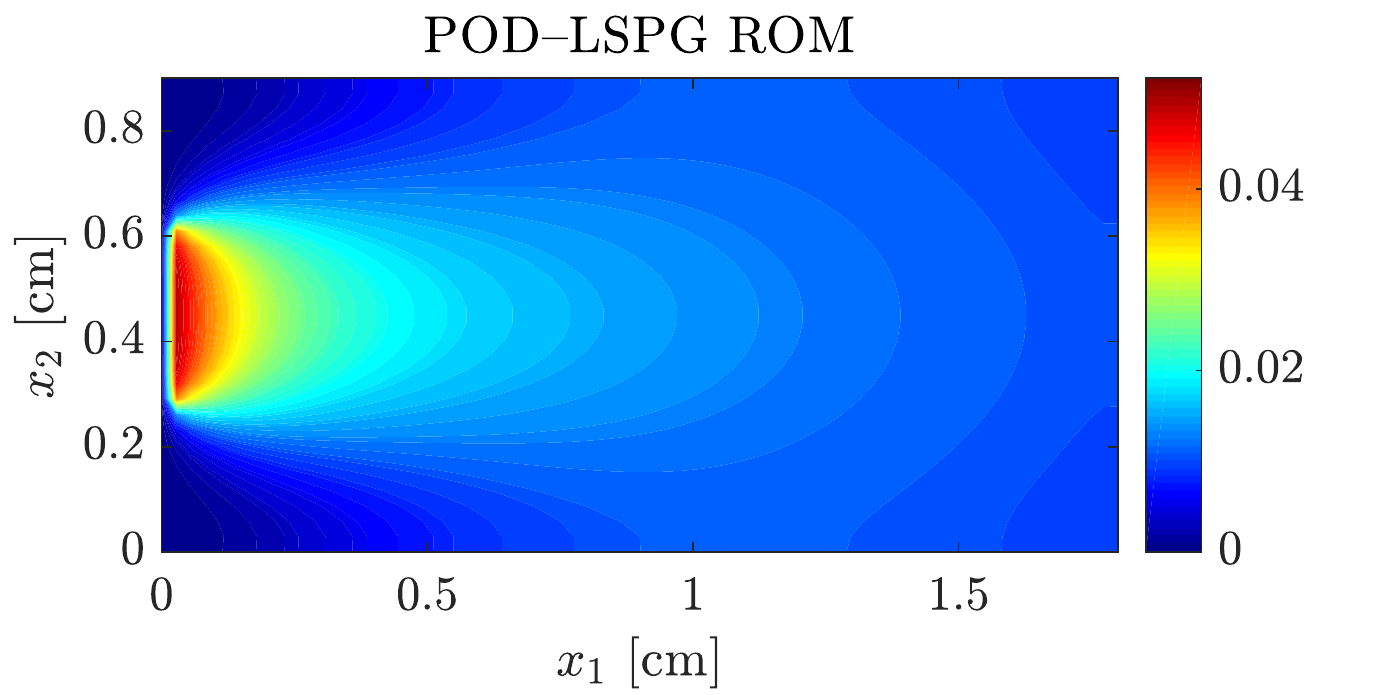}\vspace{3mm}\\
\includegraphics[width=0.475\linewidth]{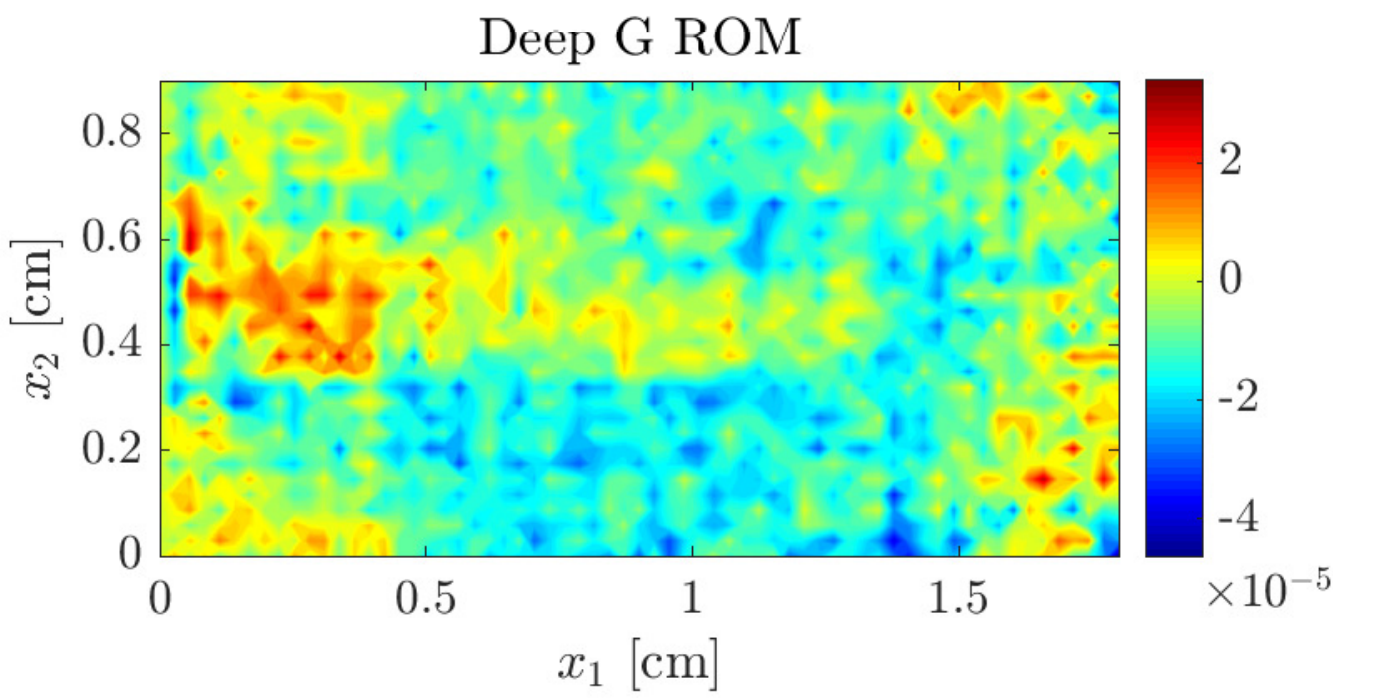}
\includegraphics[width=0.475\linewidth]{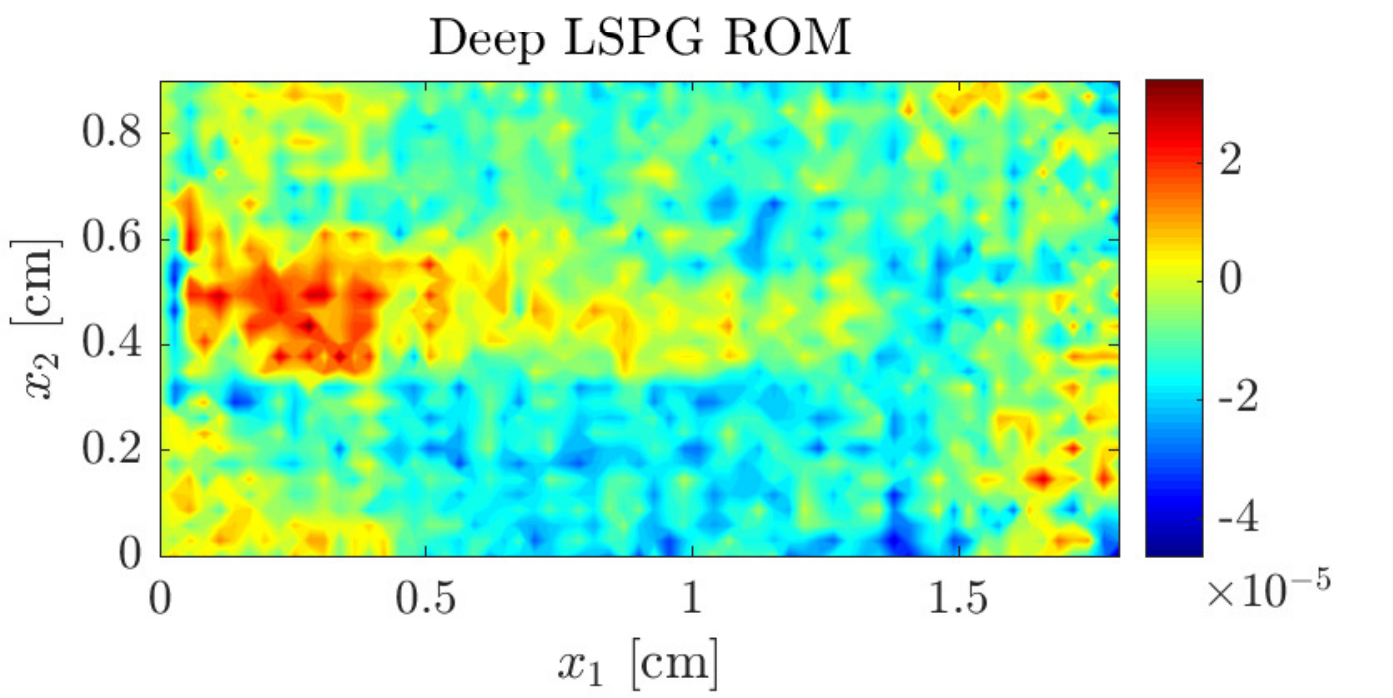}\\
\caption{\textit{Chemically reacting flow}. Relative errors of the mass
	fraction of $\fuel$ computed using ROMs at online-parameter instance
	$\paramtest^{1} =  (2.5\times10^{12},\, 5.85\times10^{3})$ at $t=0.06$. The
	dimensions of the generalized coordinates in both linear-subspace ROMs (top
	row) and nonlinear-manifold ROMs (bottom row) are the same $\nstatered=3$.
	The color maps on the each row are in the same scale.}
\label{fig:rom_H2_param1}
\end{figure}

\begin{figure}[!t]
\centering
    \begin{subfigure}[b]{0.45 \linewidth}
  \centering
    \includegraphics[width=\linewidth]{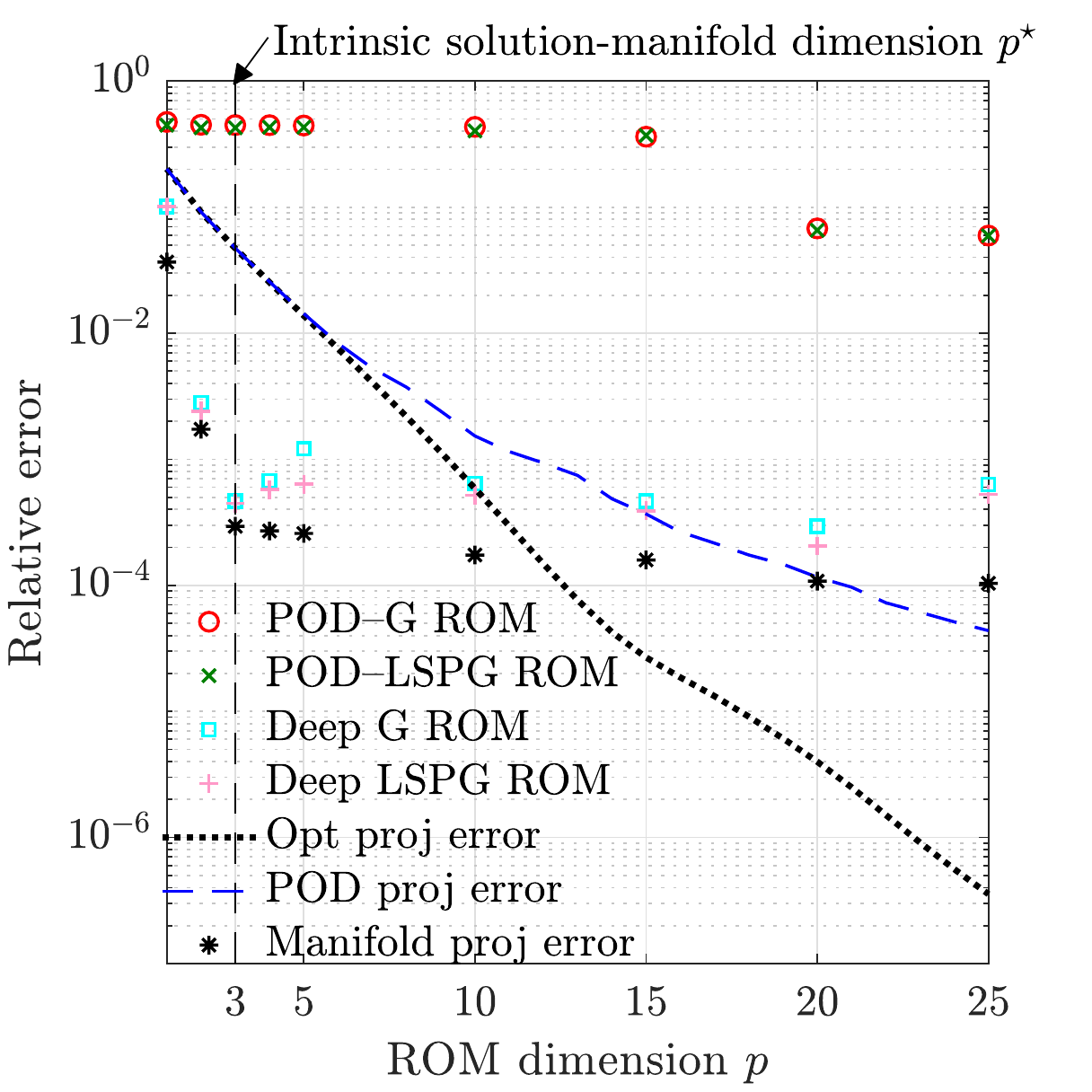}
			\caption{Online-parameter instance 1,\\ $\paramtest^{1} = \left(2.5\times10^{12}, 5.85\times10^3\right)$}
  \end{subfigure}
      \begin{subfigure}[b]{0.45 \linewidth}
  \centering
    \includegraphics[width=\linewidth]{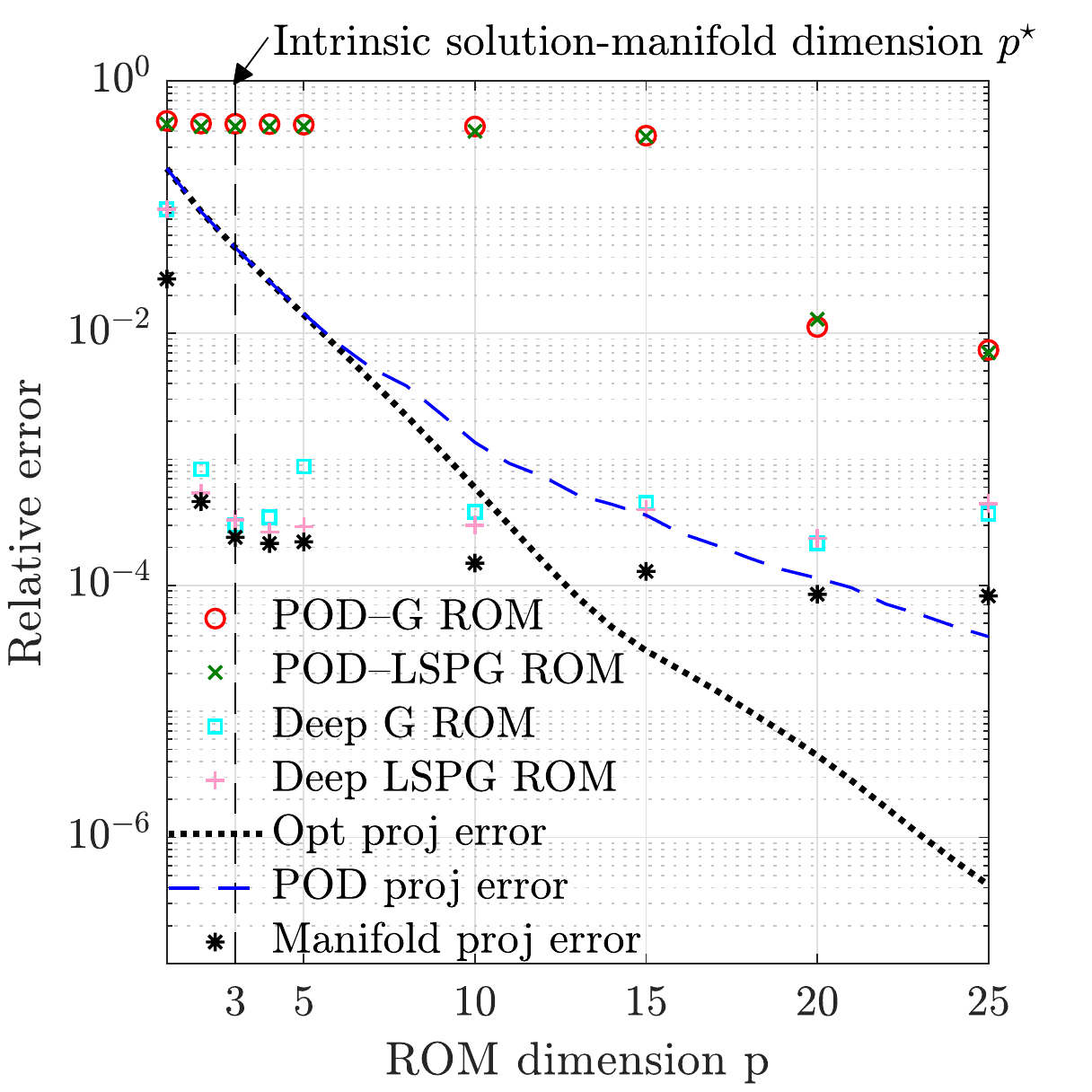}
			\caption{\RT{Online-parameter instance 2,\\ $\paramtest^{2} = \left(3.2\times 10^{12}, 7.25\times 10^3\right)$}}
  \end{subfigure}
\caption{\textit{Chemically reacting flow.} Relative errors of the ROM
	solutions (Eq.\ \eqref{eq:rel_err})  at online-parameter instances
	$\paramtest^{1} =  (2.5\times10^{12},\, 5.85\times10^{3})$ \RT{and $\paramtest^{2} = \left(3.2\times 10^{12}, 7.25\times 10^3\right)$} for varying
	ROM dimension. The figure shows the relative errors of POD--ROMs with
	Galerkin projection (POD--G ROM, red circle) and LSPG projection (POD--LSPG
	ROM, green cross), and the nonlinear-manifold ROMs equipped with the decoder
	with time-continuous optimality (Deep G ROM, cyan square) and time-discrete
	optimality (Deep LSPG ROM, magenta plus sign).  The figure also reports the
	projection error (Eq.~\eqref{eq:proj_err}) of the solution 
	onto the  POD basis
	(POD proj error, dashed blue line) and the 
	optimal basis (Opt proj error, \OWN{dotted} black line), and \OWN{reports the manifold projection error of the solution (Eq.~\eqref{eq:disc_opt_problem_optimal}, Manifold proj error, black asterisk).}
	The vertical dashed line
	indicates the intrinsic solution-manifold dimension $\intrinsicDim$ (see
	Remark \ref{rem:intrinsic}).
	}
\label{fig:err_adr}
\end{figure}

\OWN{We vary the number of training-parameter instances in the set
$\ntrain \in \{4,8,16,32,64\}$, and we use the resulting $\nsnap = \ntime
\ntrain$ snapshots to train both the autoencoder for the \GalerkinNameDeep\ and
\LSPGNameDeep\ ROMs and to compute the POD basis for the POD--Galerkin and POD--LSPG
ROMs. We fix the reduced dimension to $\dofrom = 5$. Figure
\ref{fig:adr_nparam} reports the relative error of the four ROMs for these
different amounts of training data.  This figure demonstrates that the
proposed \GalerkinNameDeep\ and \LSPGNameDeep\ ROMs yield accurate
results---less than 1\%
relative error for both \GalerkinNameDeep\ and \LSPGNameDeep---with only $\ntrain = 4$ parameter
instances. These results again show that the proposed methods do not appear to require
an excessive amount of training data relative to standard POD--based ROMs.}
\begin{figure}[!h]
\centering
    \begin{subfigure}[b]{0.45 \linewidth}
    \centering
    \includegraphics[width=\linewidth]{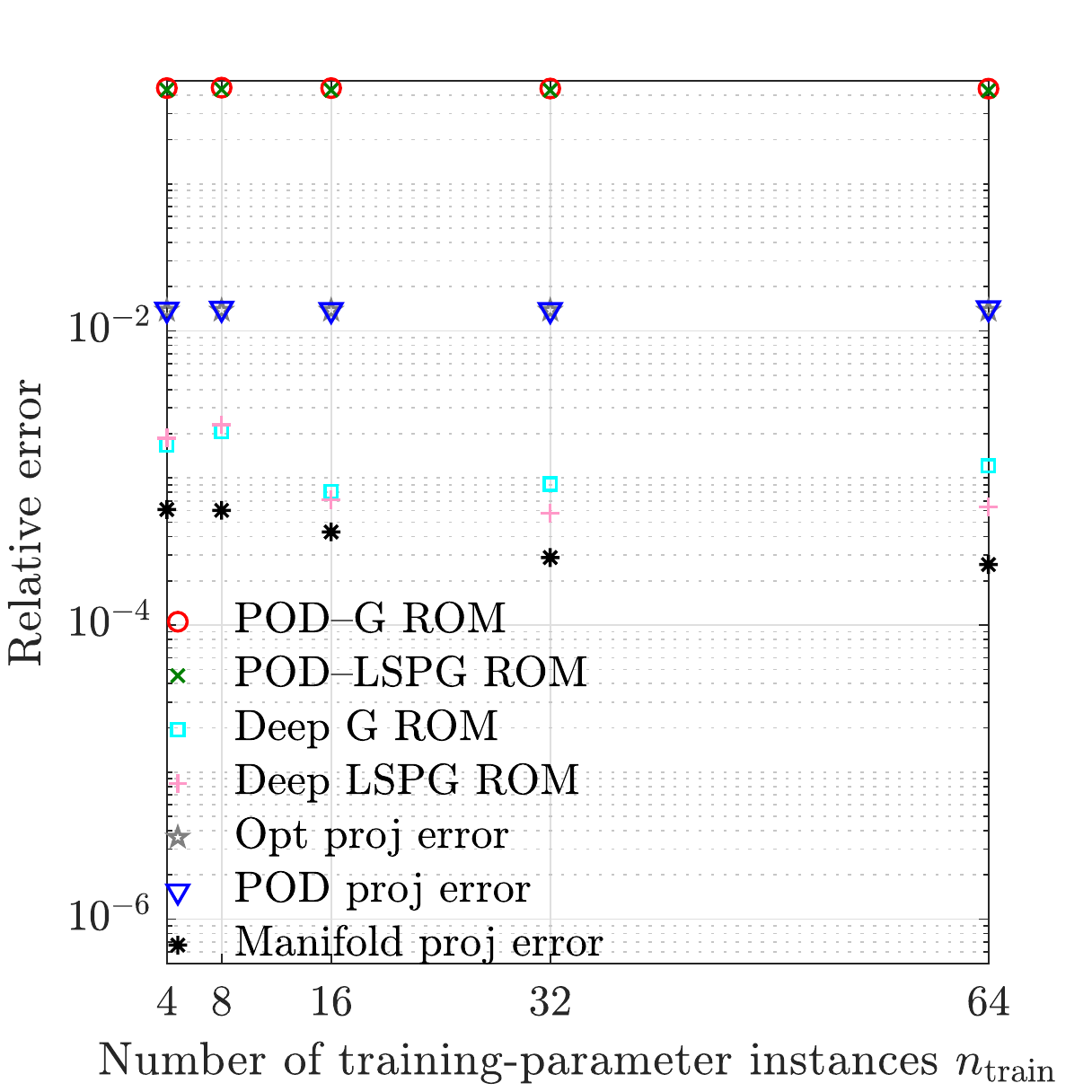}
    \caption{\OWN{Online-parameter instance 1,\\ $\paramtest^{1} = \left(2.5\times10^{12}, 5.85\times10^3\right)$ with $\dofrom=5$}}
    \end{subfigure}
    \begin{subfigure}[b]{0.45 \linewidth}
    \centering
    \includegraphics[width=\linewidth]{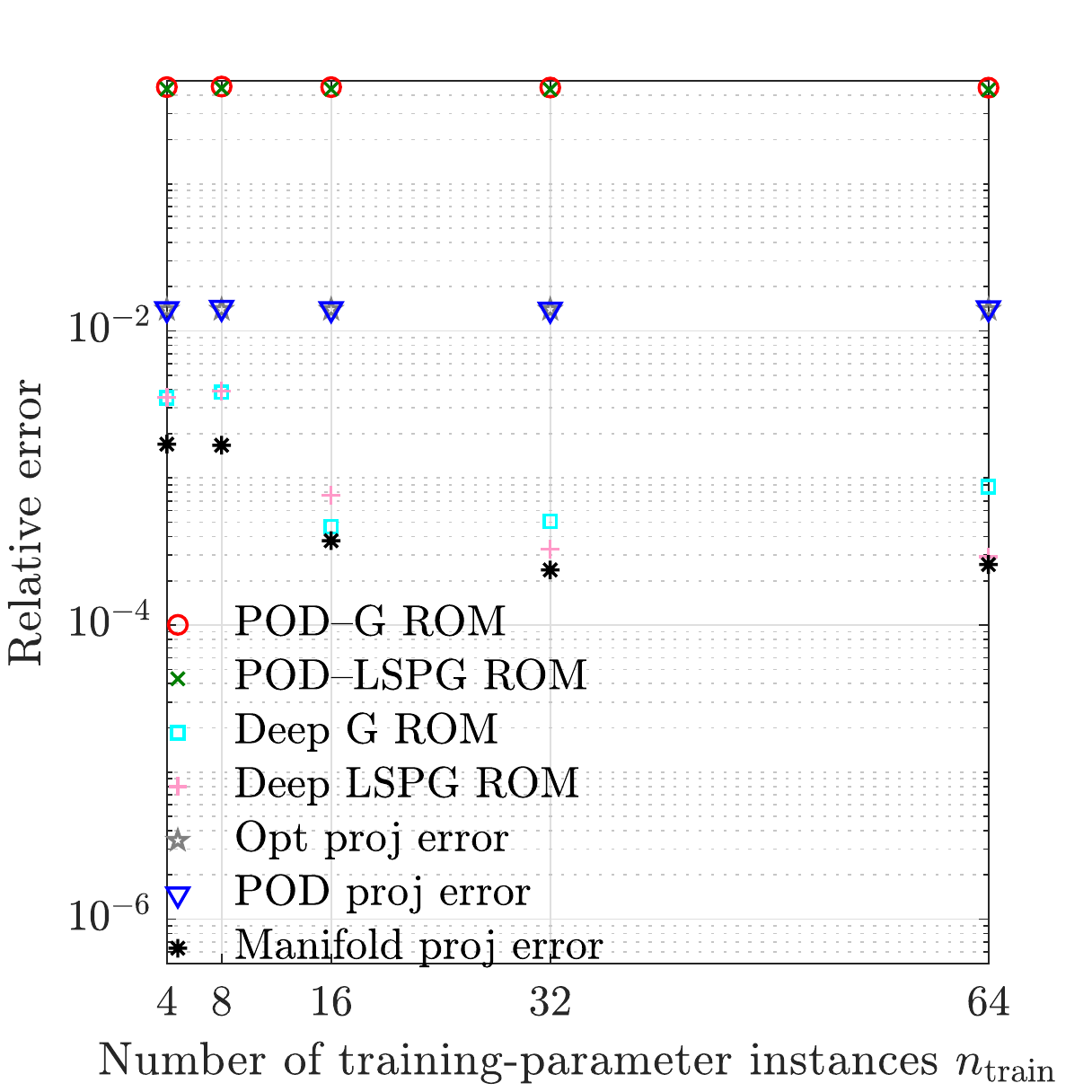}
    \caption{\OWN{Online-parameter instance 2,\\ $\paramtest^{2} = \left(3.2\times 10^{12}, 7.25\times 10^3\right)$ with $\dofrom=5$}}
    \end{subfigure}
	\caption{\OWN{\textit{Chemically reacting flow: varying amount of training data.} Relative errors of the ROM
	solutions (Eq.\ \eqref{eq:rel_err})  at online-parameter instances
	$\paramtest^{1} =  (2.5\times10^{12},\, 5.85\times10^{3})$ and $\paramtest^{2} = \left(3.2\times 10^{12}, 7.25\times 10^3\right)$ for a varying the number training-parameter instances	$\ntrain \in \{4,8,16,32,64\}$. The figure shows the relative errors of POD--ROMs with
	Galerkin projection (POD--G ROM, red circle) and LSPG projection (POD--LSPG
	ROM, green cross), and the nonlinear-manifold ROMs equipped with the decoder
	with time-continuous optimality (Deep G ROM, cyan square) and time-discrete
	optimality (Deep LSPG ROM, magenta plus sign).  The figure also reports the
	projection error (Eq.~\eqref{eq:proj_err}) of the solution 
	onto the  POD basis
	(POD proj error, blue triangle) and the 
	optimal basis (Opt proj error, gray pentagram), and reports the manifold projection error of the solution (Eq.~\eqref{eq:disc_opt_problem_optimal}, Manifold proj error, black asterisk).
	}}
\label{fig:adr_nparam}
\end{figure}

\section{Conclusion}\label{sec:conc}
This work has proposed novel \GalerkinName\  and \LSPGName\ projection
techniques, which project dynamical-system models onto arbitrary
continuously-differentiable nonlinear manifolds. We demonstrated how these methods
can exactly satisfy the initial condition, and provided quasi-Newton solvers
for implicit time integration.

We performed analyses that demonstrated that employing an affine trial
manifold recovers classical linear-subspace Galerkin and LSPG projection. We
also derived sufficient conditions for commutativity of time discretization
and \GalerkinName\ projection, as well as conditions under which
\GalerkinName\ and \LSPGName\ projection are equivalent. \OWN{In addition, we
derived \textit{a posteriori} time-discrete error bounds for the proposed
methods.}

We also proposed a practical strategy for computing a representative
low-dimensional nonlinear trial manifold that employs a specific convolutional
autoencoder tailored to spatially distributed dynamical-system states.  When
the \GalerkinNameCap\ and \LSPGNameCap\ ROMs employ this choice of decoder, we
refer to them as {\GalerkinNameDeepCap}\ and {\LSPGNameDeepCap}\
ROMs, respectively.

Finally, numerical experiments demonstrated the ability of the method to
produce substantially lower errors for low-dimensional reduced states than
even the projection onto the optimal basis. Indeed, the proposed methodology
is nearly able to achieve optimal performance for a nonlinear-manifold method;
in the first case, the reduced dimension required to achieve sub-0.1\% errors
is only two larger than the intrinsic solution-manifold dimension; in the
second case, the reduced dimension required to achieve such errors is exactly
equal to the intrinsic solution-manifold dimension.

\RT{\input{drawbacks}}

Future work involves integrating the proposed \GalerkinName\ and \LSPGName\
projection methods within a full hyper-reduction framework to realize
computational cost savings; this investigation will leverage sparse norms as
discussed in Remarks \ref{rem:normGal} and \ref{rem:normLSPG}, and will also
consider specific instances of the proposed autoencoder architecture that
enable computationally efficient hyper-reduction.  Additional future work
includes appending structure-preserving constraints to the minimum-residual
formulations \cite{carlberg2018conservative}, implementing the proposed
techniques in a production-level simulation code and demonstrating the methods
on truly large-scale dynamical-system models.

\section*{Acknowledgments}
The authors gratefully acknowledge Matthew Zahr for graciously providing the
\texttt{pyMORTestbed} code that was modified to obtain the numerical results,
as well as Jeremy Morton for useful discussions on the application of
convolutional neural networks to simulation data.  The authors also thank
Patrick Blonigan and Eric Parish for providing useful feedback.  This work was
sponsored by Sandia's Advanced Simulation and Computing (ASC) Verification and
Validation (V\&V) Project/Task \#103723/05.30.02.  
This paper describes objective technical results and analysis. Any subjective
views or opinions that might be expressed in the paper do not necessarily
represent the views of the U.S. Department of Energy or the United States
Government.
Sandia National
Laboratories is a multimission laboratory managed and operated by National
Technology \& Engineering Solutions of Sandia, LLC, a wholly owned subsidiary
of Honeywell International Inc., for the U.S. Department of Energy's National
Nuclear Security Administration under contract DE-NA0003525.

\appendix

\section{Stochastic gradient descent for autoencoder training}\label{app:SGD}

This section briefly summarizes stochastic gradient descent (SGD) with
minibatching and early stopping, which is used to  compute the parameters
$\nnparamOpt$ of the autoencoder.

We begin by randomly shuffling the snapshot matrix $\snapshotMat$ (defined in
Eq.~\eqref{eq:snapshotMat}) into neural-network training
snapshots $\snapshotTrain\in\RR{\ndof\times \ntrainNN}$
and validation snapshots $\snapshotVal\in\RR{\ndof\times \nvalNN}$, i.e.,
$$
[\snapshotTrain\ \snapshotVal] = \snapshotMat\randperm
$$
where $\randperm\in\{0,1\}^{\nsnap\times\nsnap}$ denotes a random
permutation matrix and
$\ntrainNN+\nvalNN = \nsnap$, typically with $\nvalNN \approx 0.1\ntrainNN$.

Given the training data $\snapshotTrain$, we compute the optimal parameters
$\nnparamOpt$ by (approximately) solving the optimization problem
\begin{align}\label{eq:cost_func}
	\underset{\nnparam}{\text{minimize}}\ \objfunc({\nnparam}) = \expect_{\gvec
	\sim \empdist} \lossfunc (\gvec, \nnparam) = \frac{1}{\ntrainNN}
	\sum_{i=1}^{\ntrainNN} \lossfunc (\snapshotTrainArg{i}, \nnparam),
\end{align}
which minimizes the empirical risk over the training data.  Here,  $\empdist$
denotes the empirical distribution associated with the training data
$\snapshotTrain$ and the loss function $\mathcal L$ provides a measure of
discrepancy between training snapshot $\snapshotTrainArg{i}$ and its
reconstruction $\autoenc(\snapshotTrainArg{i};\nnparam)$; for example, the
$\ell^2$-loss function is often employed such that
\begin{equation}\label{eq:mse}
\lossfunc:(\gvec,\nnparam)\mapsto
\|\gvec-\autoenc(\gvec;\nnparam)\|_2^2,
\end{equation}
in which case the loss function is equivalent to that employed for POD
(compare with Problem \eqref{eq:PODopt}).
Note that autoencoder training is categorized as an
\textit{unsupervised learning} (or semi-supervised learning) problem, as there
is no target response variable other than recovery of the original input data.

We apply SGD to (approximately) solve
optimization problem \eqref{eq:cost_func}, which leads to parameter updates at
the $i$th iteration of the form 
\begin{align}\label{eq:optUpdate}
	\nnparam^{(i+1)} = \nnparam^{(i)} - \learningrate 
	\gradientApproxArg{i}({\nnparam^{(i)}}).
\end{align} 
Here, $\learningrate\in\RR{}$ denotes the step length or \textit{learning rate} and
$\gradientApproxArg{i}(\approx\gradient)$ denotes a gradient approximation at
optimization iteration $i$. The gradient approximation corresponds to the sample mean of
the gradient over a \textit{minibatch}
$\minibatchArg{\iterToBatch{i}}\in\RR{\ndof\times\nminibatchArg{\iterToBatch{i}}}$, where
			$$[\minibatchArg{1}\ \cdots\ \minibatchArg{\nbatch}] =
			\snapshotTrain\randpermTrain,$$
			where			$\iterToBatchNo:i\mapsto((i-1)\mod\nbatch)+1$ provides the mapping from
			optimization iteration $i$ to batch index, 
			and
			$\randpermTrain\in\{0,1\}^{\ndof\times \ntrainNN}$ denotes a random
			permutation matrix. That is, the gradient approximation at optimization
			iteration $i$ is 
\begin{equation} \label{eq:SGD}
	\gradientApproxArg{i}({\nnparam^{(i)}}) =
	\frac{1}{\nminibatchArg{\iterToBatch{i}}}
	\sum_{j=1}^{\nminibatchArg{\iterToBatch{i}}}\nabla_{\nnparam}\lossfunc(
	\minibatchVecArgs{\iterToBatch{i}}{j},\nnparam).
\end{equation} 
For small batch sizes, this gradient approximation not only significantly
reduces the per-iteration cost of the optimization algorithm, it can also
improve generalization performance, and---for early iterations and for
minibatch size 1---leads to the same sublinear rate of convergence of the
\textit{expected risk} as the empirical risk \cite{bottou2018optimization}. 
In practice, each gradient contribution $\nabla_{\nnparam}\lossfunc(
\minibatchVecArgs{\iterToBatch{i}}{j},\nnparam)$ is computed from the chain
rule via automatic differentiation, which is referred to in deep learning as
\textit{backpropagation} \cite{rumelhart1986learning}.

Some optimization methods employ adaptive learning rates that are tailored for
each parameter, which comprises a modification of the parameter update
\eqref{eq:optUpdate} to
\begin{align}\label{eq:optUpdateMod}
	\nnparam^{(i+1)} = \nnparam^{(i)} - \learningrateVecArg{i}\odot
	\gradientApproxArg{i}({\nnparam^{(i)}})
\end{align} 
where $\learningrateVecArg{i}$ denotes a vector of learning rates and $\odot$ denotes the (element-wise) Hadamard product. Examples include
AdaGrad \cite{duchi2011adaptive}, RMSProp \cite{Tieleman2012}, and Adam. 

Although we have presented the specific case of non-adaptive minibatches
(which we employ in the numerical experiments), it is also possible to employ
adaptive batch sizes; often, the batch size increases with iteration count to
produce a lower-variance gradient estimate as a local minimum is approached
\cite{bottou2018optimization}\RO{.}

Rather than terminating iterations when a local minimum of the objective
function $\objfunc$ is reached, we instead employ \textit{early stopping},
which is a widely used form of regularization that has been shown to
improve generalization performance in many applications
\cite{goodfellow2016deep,bottou2018optimization}. Here, we
terminate iterations when the loss function on the validation snapshots
$$
\frac{1}{\nvalNN}
	\sum_{i=1}^{\nvalNN} \lossfunc (\snapshotValArg{i}, \nnparam)
$$
does not decrease for a certain number of \textit{epochs}, where an epoch is
equivalent to $\nbatch$ iterations, i.e., a single pass through the training
data.  Early stopping effectively treats the number of optimization iterations
as a hyperparameter, as allowing a large number of iterations can be
associated with a higher-capacity model.

Algorithm \ref{alg:ae_sgd} describes autoencoder training using SGD with
minibatching and early stopping.  Note that the adaptive learning rate
strategy, initial parameters $\nnparam^{(0)}$, number of minibatches
$\nbatch$, maximum number of epochs $\nepoch$, and early-stopping criterion
are all SGD hyperparameters that comprise inputs to the algorithm.  Line
\ref{ln:lr_update} of Algorithm \ref{alg:ae_sgd} applies the adaptive learning
rate strategy. For example, AdaGrad scales the learning rate to be inversely
proportional to the square root of the sum of all the historical squared
values of the gradient \cite{duchi2011adaptive}; Adam updates the learning
rate based on estimates of the first moment (the mean) and the second moment
(the uncentered variance) of the gradients \cite{kingma2014adam}.
Alternatively, the learning rate can be kept to a single constant for all
parameters over the entire training procedure, which simplifies the parameter
update procedure in Line \eqref{ln:param_update} of Algorithm \ref{alg:ae_sgd}
to the typical SGD update \eqref{eq:optUpdate}.  See
Ref.~\cite{bottou2018optimization} for a review of training deep neural
networks.

\begin{algorithm}[H]
    \caption{Stochastic gradient descent with minibatching and early stopping}
    \label{alg:ae_sgd}
    \algorithmicrequire{
			Snapshot matrix $\snapshotMat$;
			autoencoder 
		$\autoenc(\gvec; \nnparam) = \decoder(\cdot;\decparam)\circ
		\encoder(\gvec; \encparam)$; SGD hyperparameters (fraction of snapshots to employ for
			validation $\nvalidationNNfrac\in [0,1]$;
		adaptive learning rate strategy; initial parameters $\nnparam^{(0)}$; 
		 number of minibatches $\nbatch$; maximum number of epochs $\nepoch$;
		 early-stopping criterion)}\\
    \algorithmicensure{Encoder
		$\encoder(\cdot) = \encoder(\cdot;\encparam^\star)$; decoder $\rdbasisnl(\cdot) =
		\decoder(\cdot;\decparam^\star)$; }
    \begin{algorithmic}[1] 
        \State Randomly shuffle the snapshot matrix into neural-network
				training and validation snapshots 
				$
[\snapshotTrain\ \snapshotVal] = \snapshotMat\randperm
				$ with $\snapshotVal\in\RR{\ndof\times\nvalNN}$ with $\nvalNN =
				\nvalidationNNfrac \nsnap$.
        \State Randomly shuffle the training snapshots into minibatches
				$
[\minibatchArg{1}\ \cdots\ \minibatchArg{\nbatch}] =
			\snapshotTrain\randpermTrain
				$.
				\State Initialize optimization iterations $i\leftarrow 0$
        \For{$j = 1, \ldots, \nepoch$}
          \For{$k = 1,\ldots,\nbatch$}
              \State Update learning rate $\learningrateVecArg{i}$ based on adaptive learning rate strategy \label{ln:lr_update}
              \State Perform parameter update $\nnparam^{(i+1)} = \nnparam^{(i)} - \learningrateVecArg{i}\odot\gradientApproxArg{i}({\nnparam^{(i)}})$ \label{ln:param_update}
              \State $i\leftarrow i+1$
          \EndFor
          \State Terminate if early-stopping criterion is satisfied on the loss 
					$\frac{1}{\nvalNN} \sum_{i=1}^{\nvalNN} \lossfunc (\snapshotValArg{i}, \nnparam^{(i)})
					$
        \EndFor
				\State Set $\nnparamOpt\leftarrow\nnparam^{(i)}$
    \end{algorithmic}
\end{algorithm}


\newcommand*{\htensor}{{\mathcal H}}
\newcommand*{\wtensor}{{\mathcal W}}
\newcommand*{\btensor}{{\mathcal B}}
\newcommand*{\nfilter}{n_\text{filter}}
\newcommand*{\kerlen}{k}
\newcommand*{\stride}{s}

\section{Basics of convolutional layers}\label{app:conv}
This section provides notations and basic operations performed in convolutional layers. See \cite{dumoulin2016guide} for more details of the convolution arithmetic for deep learning.

Units in convolutional layers are organized as \textit{feature maps} $\htensor$, where each unit is connected to the local patches of the feature maps of the previous layer through a discrete convolution defined by a set of \textit{kernels}, which is denoted by \textit{filter bank} $\wtensor$, followed by a nonlinear activation and a pooling opertion. The feature map at layer $l$ can be considered as a 3-dimensional tensor $\htensor^{l} \in \RR{\nchannel^l\times\nspaceDim{1}^l\times\nspaceDim{2}^l}$ with element $\htensor^{l}_{i,j,k}$ representing a unit within channel $i$ at row $j$ and column $k$, and the filter banks at layer $l$ can be considered as a 4-dimensional tensor $\wtensor^{l} \in \RR{\nfilter^l \times\nchannel^{l-1} \times \kerlen_1^l \times \kerlen_2^l}$ with element $\wtensor^{l}_{i,j,m,n}$ connecting between a unit in channel $i$ of the output and a unit in channel $j$ of the input, with an offset of $m$ rows and $n$ columns between the output unit and the input unit. The number of filters in the filter bank is denoted by $\nfilter$, and the kernel length is characterized by $\kerlen_1$ and $\kerlen_2$. Convolving a feature map $\htensor^{l-1}$ with a filter bank $\wtensor^{l}$ can be written as
\begin{align*}
\htensor^{l}_{i,j,k} = \activation_{l} \left(\sum_{r,m,n} \htensor^{l-1}_{r, (j-1)\times s+m,(k-1) \times s+n   } \wtensor^{l}_{i,r,m,n} + \btensor^{l}_{i,j,k}\right),
\end{align*}
for all valid $i$, where $\mathcal B^{l}$ is a tensor indicating bias. Here, $s$ denote the \textit{stride}, which determines downsampling rate of each convolution; only every $s$ elements is sampled in each direction in the output. By having $s > 1$, the dimension of the next feature map can be reduced by factor of $s$ in each direction. The filter banks $\wtensor$ and the biases $\btensor$ are learnable parameters, whereas the kernel length [$\kerlen_1$, $\kerlen_2$] and the number of filters $\nfilter$, and the stride  $s$ are the hyperparameters. After the nonlinearity, a pooling function is typically applied to the output, extracting a statistical summary of the neighboring units of the output at certain locations. 

\RT{
\section{Additional numerical experiments: extrapolation and interpolation in
time}
We now perform additional investigations that assess the ability of the
proposed \GalerkinNameDeep\ and \LSPGNameDeep\ ROMs to both extrapolate
(\ref{app:timeExtrap}) and interpolate (\ref{sec:timeInterp})
in time. We also assess the effect of early stopping on the performance of these methods
(\ref{app:earlyStop}). All numerical experiments in this section are performed on
the 1D Burgers' equation described in Section \ref{sec:burg} with the same setup
except when otherwise specified.
\subsection{Time extrapolation}\label{app:timeExtrap}
To assess time extrapolation, we collect snapshots associated with the first
$\nsubsnap(\leq\ntime=500)$ time instances for each training-parameter instance and use the
resulting $\nsnap = \nsubsnap \ntrain$ snapshots to train the autoencoder for
the \GalerkinNameDeep\ and \LSPGNameDeep\ ROMs and to compute the POD basis for the
POD--Galerkin and POD--LSPG ROMs. Figure \ref{fig:burger_extrapolation}
reports the relative error for each of these ROMs of dimension $\dofrom =
5$  
at the online-parameter
instances $\paramtest^1$ and $\paramtest^2$ for the number of collected
snapshots varying in the set $\nsubsnap \in \{200,300,400,500\}$.
These results show that the proposed \GalerkinNameDeep\ and \LSPGNameDeep\
ROMs yield more accurate results than the POD--Galerkin and POD--LSPG ROMs for
$\nsubsnap\geq 300$, with \LSPGNameDeep\ yielding relative errors around 0.1\%
for $\nsubsnap\geq 400$. Thus, we conclude that---for this example---the
proposed \GalerkinNameDeep\ and \LSPGNameDeep\ generally yield superior
time-extrapolation results than the POD--Galerkin and POD--LSPG ROMs.
\begin{figure}[!h]
\centering
    \begin{subfigure}[b]{0.45 \linewidth}
    \centering
    \includegraphics[width=\linewidth]{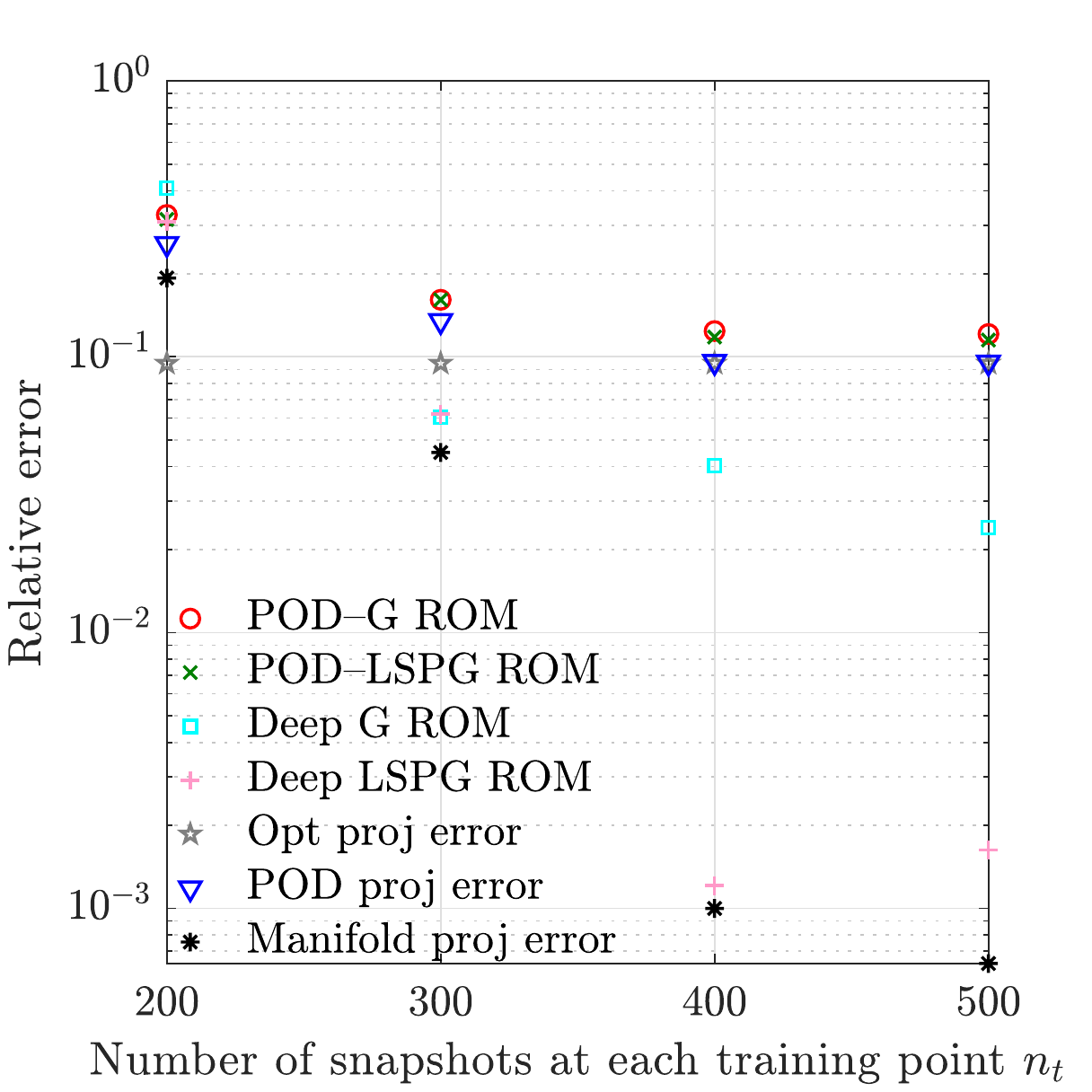}
    \caption{\RT{Online-parameter instance $\paramtest^{1} = \left(4.3, 0.021\right)$ with $\dofrom=5$}}
    \end{subfigure}
    \begin{subfigure}[b]{0.45 \linewidth}
    \centering
    \includegraphics[width=\linewidth]{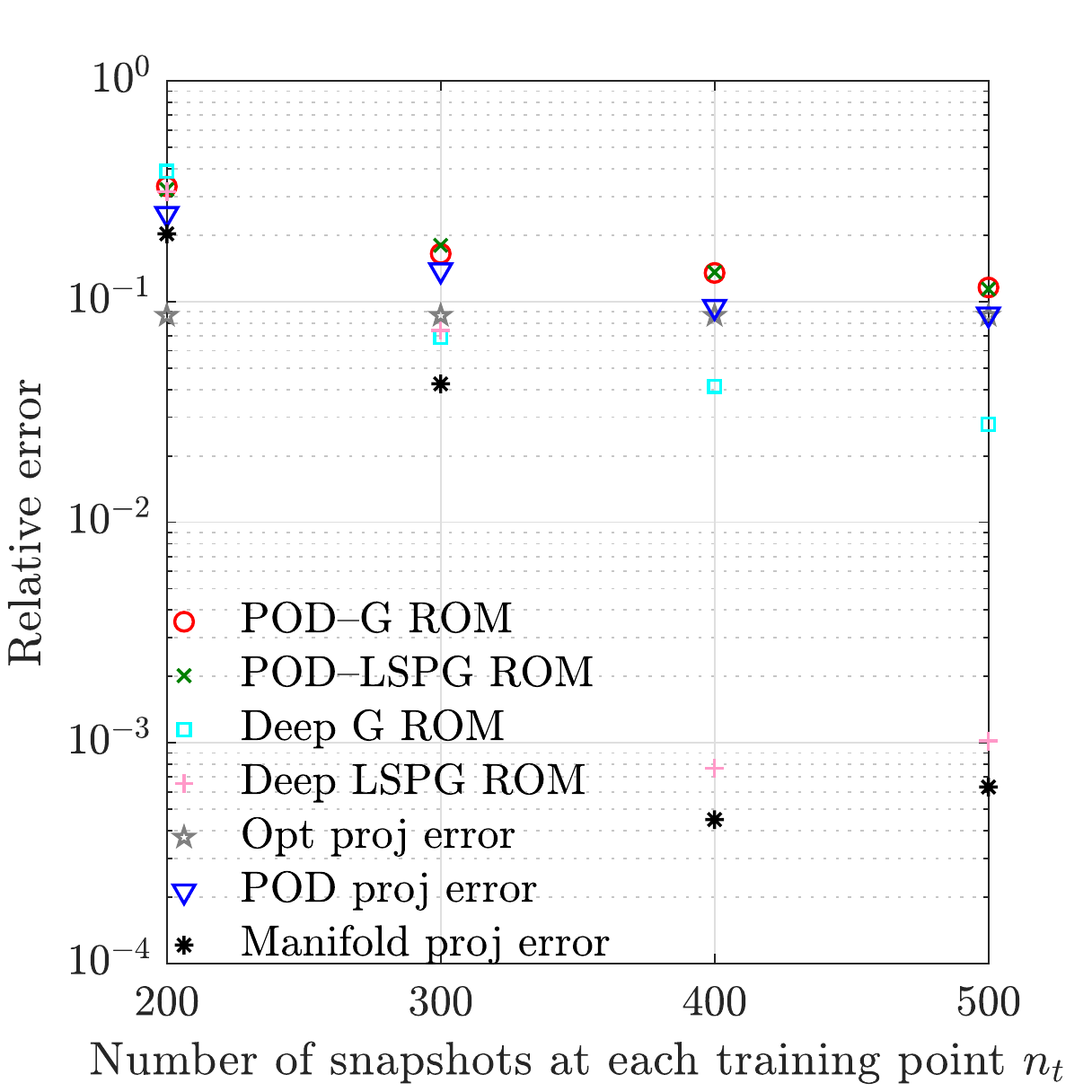}
    \caption{\RT{Online-parameter instance $\paramtest^{2} = \left(5.15, 0.0285\right)$ with $\dofrom=5$}}
    \end{subfigure}
\caption{\RT{\textit{1D Burgers' equation: time extrapolation.} Relative errors of
	the ROM solutions (Eq.\ \eqref{eq:rel_err}) at online-parameter instances 
	$\paramtest^{1}$ and $\paramtest^{2}$ for a varying number of snapshots
	$\nsubsnap \in \{200,300,400,500\}$ collected at each
	training-parameter instance. In this
	numerical experiment, these snapshots correspond to the state collected at
	the first $\nsubsnap$ time instances at each training-parameters instance.
	The figure shows the relative errors of POD--ROMs with Galerkin projection
	(POD--G ROM, red circle) and LSPG projection (POD--LSPG ROM, green cross),
	and the nonlinear-manifold ROMs equipped with the decoder with
	time-continuous optimality (Deep G ROM, cyan square) and time-discrete
	optimality (Deep LSPG ROM, magenta plus sign). The figure also reports the
	projection error (Eq.~\eqref{eq:proj_err}) of the solution 
	onto the  POD basis
	(POD proj error, blue triangle) and the 
	optimal basis (Opt proj error, gray pentagram), and reports the manifold projection error of the solution (Eq.~\eqref{eq:disc_opt_problem_optimal}, Manifold proj error, black asterisk).}}
\label{fig:burger_extrapolation}
\end{figure}
\subsection{Time interpolation}\label{sec:timeInterp}
To assess time interpolation, we now consider collecting $\nsubsnap$ snapshots
that are equispaced in the time interval at each training-parameter instance,
with the final snapshot collected at the final time instance $t^\ntime$.
As in the case of time extrapolation, we then use the resulting $\nsnap =
\nsubsnap \ntrain$ snapshots to train both the autoencoder for the
\GalerkinNameDeep\ and \LSPGNameDeep\ ROMs and to compute the POD basis for the
POD--Galerkin and POD--LSPG ROMs. 
Figure \ref{fig:burger_interpolation}
reports the relative error for each of these ROMs 
of dimension $\dofrom =
5$  
at the online-parameter
instances $\paramtest^1$ and $\paramtest^2$ for the number of collected
snapshots varying in the set $\nsubsnap \in \{10,20,50,100,140,250,500\}$.
First, these results show that 
the
POD--Galerkin and POD--LSPG ROMs are almost entirely insensitive to the number
of snapshots collected within the time interval for $\nsubsnap\geq 10$. Similarly, the
proposed \GalerkinNameDeep\ and \LSPGNameDeep\ ROMs are insensitive to the
number of snapshots for $\nsubsnap\geq 100$. For $\nsubsnap\in\{10,20,50\}$,
the \GalerkinNameDeep\ and \LSPGNameDeep\ ROMs yield larger errors than for
$\nsubsnap\geq 100$; however,
these errors are still smaller than the errors produced by  the POD--Galerkin
and POD--LSPG ROMs. Thus, we conclude that---for this example---the proposed
methods do not exhibit accuracy degradation when only 20\% of the snapshots
(collected equally spaced in time) are used to train the autoencoder.
\begin{figure}[!h]
\centering
    \begin{subfigure}[b]{0.45 \linewidth}
    \centering
    \includegraphics[width=\linewidth]{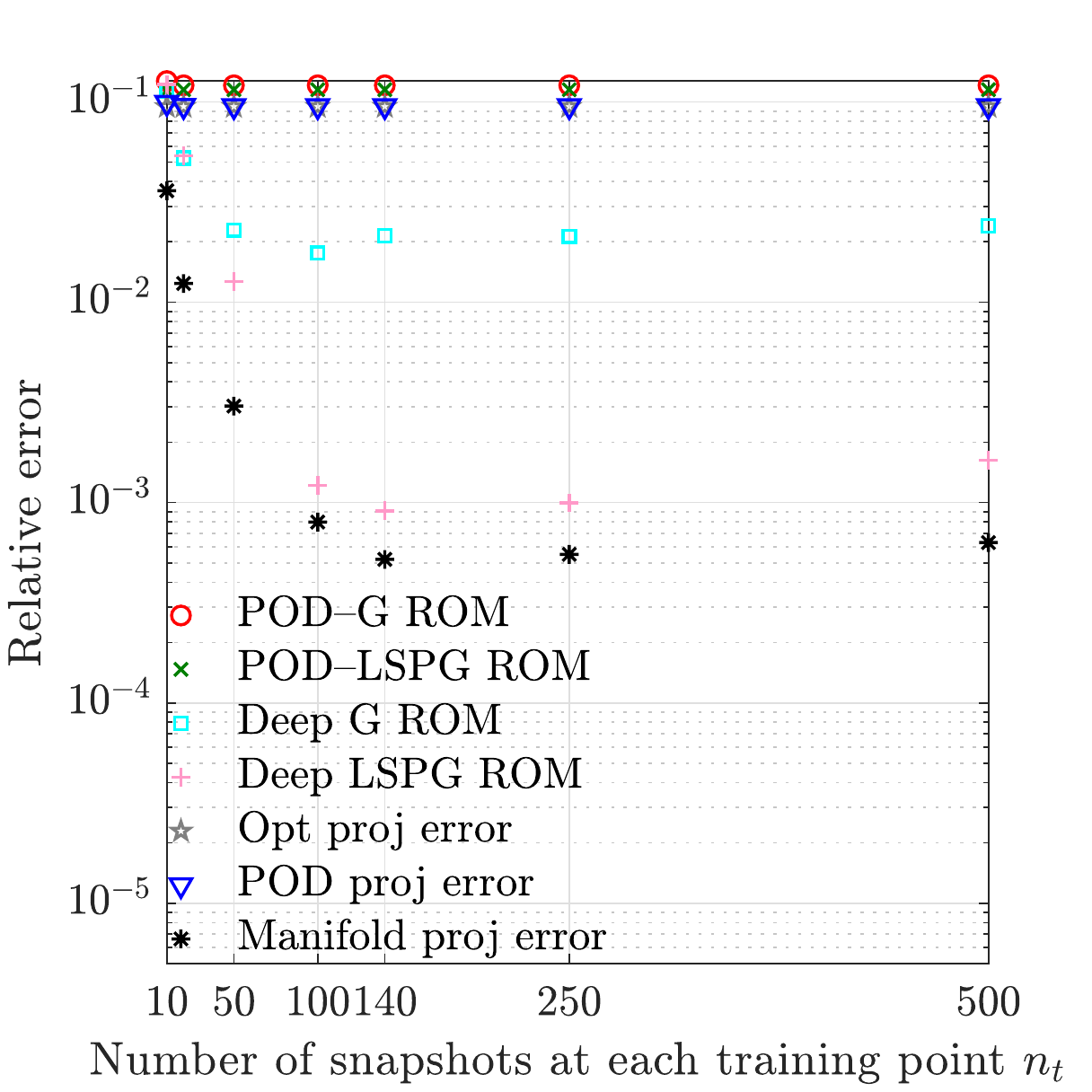}
    \caption{\RT{Online-parameter instance $\paramtest^{1} = \left(4.3, 0.021\right)$ with $\dofrom=5$}}
    \end{subfigure}
    \begin{subfigure}[b]{0.45 \linewidth}
    \centering
    \includegraphics[width=\linewidth]{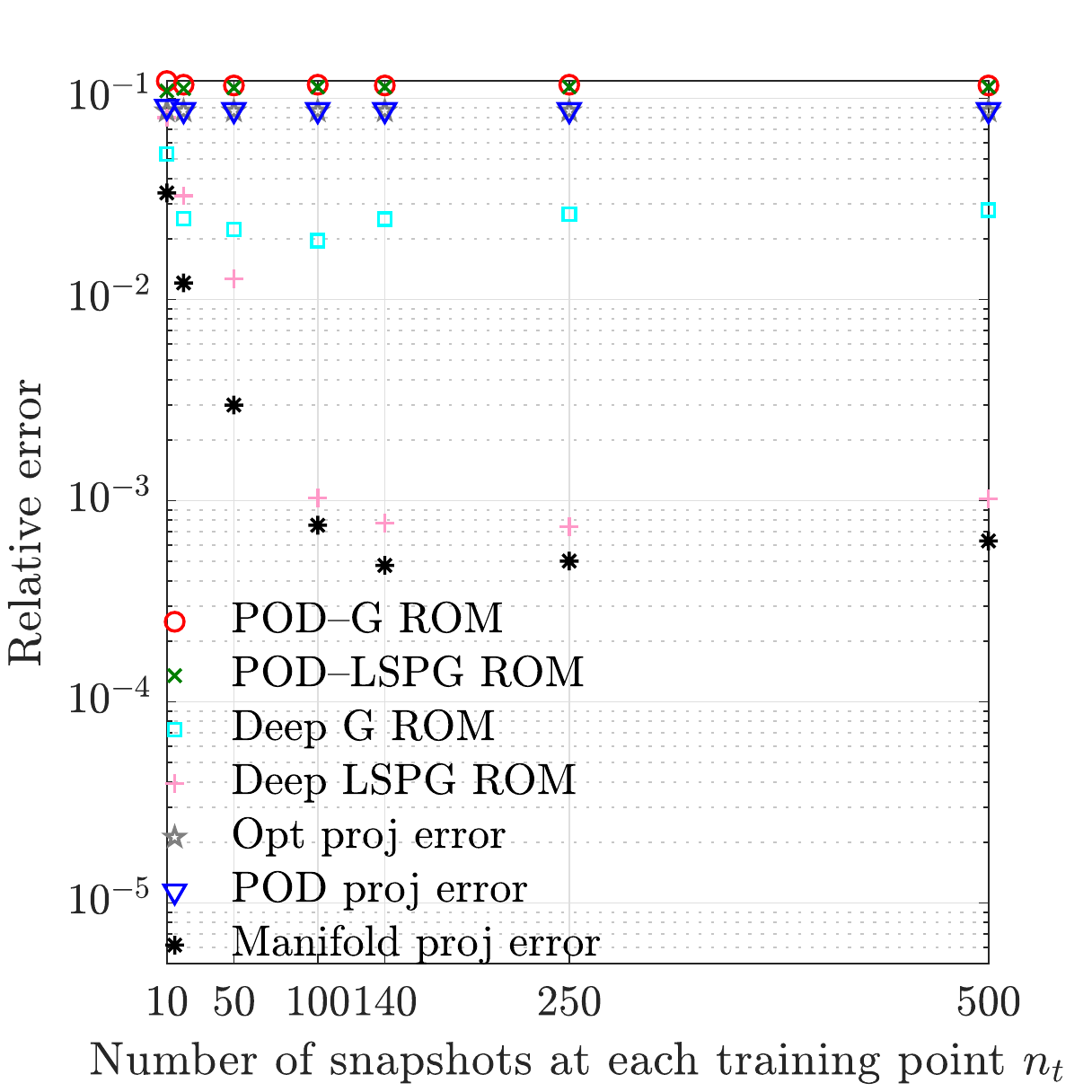}
    \caption{\RT{Online-parameter instance $\paramtest^{2} = \left(5.15, 0.0285\right)$ with $\dofrom=5$}}
    \end{subfigure}
\caption{\RT{\textit{1D Burgers' equation: time interpolation.} 
	Relative errors of
	the ROM solutions (Eq.\ \eqref{eq:rel_err}) at online-parameter instances 
	$\paramtest^{1}$ and $\paramtest^{2}$ for a varying number of snapshots
	$\nsubsnap \in \{10,20,50,100,140,250,500\}$ collected at each
	training-parameter instance. In this
	numerical experiment, these snapshots correspond to the state collected at
	$\nsubsnap$ equally-spaced time instances at each training-parameters
	instance, with the final snapshot being collected at the final time
	$t^\ntime$.
	The figure shows the relative errors of POD--ROMs with Galerkin projection
	(POD--G ROM, red circle) and LSPG projection (POD--LSPG ROM, green cross),
	and the nonlinear-manifold ROMs equipped with the decoder with
	time-continuous optimality (Deep G ROM, cyan square) and time-discrete
	optimality (Deep LSPG ROM, magenta plus sign). The figure also reports the
	projection error (Eq.~\eqref{eq:proj_err}) of the solution 
	onto the  POD basis
	(POD proj error, blue triangle) and the 
	optimal basis (Opt proj error, gray pentagram), and reports the manifold projection error of the solution (Eq.~\eqref{eq:disc_opt_problem_optimal}, Manifold proj error, black asterisk).}}
\label{fig:burger_interpolation}
\end{figure}
}
\RO{
\subsection{Early-stopping study}\label{app:earlyStop}
We now assess the effect of early stopping---a common approach in training
deep neural networks for improving generalization performance---on the
performance of the proposed
\GalerkinNameDeep\ and \LSPGNameDeep\ methods. To do so, 
we set the reduced dimension to $\dofrom=30$ and 
train the autoencoder both with and without the early stopping
strategy. To train without early stopping, we set the maximum number of epochs
to
$\nepoch=1000$, save the network parameters at each epoch, and select the
parameters (out of all 
$\nepoch=1000$ candidates) that yield the smallest value of  the loss function
$\lossfunc$ (defined in Eq.\ \ref{eq:mse}) on the validation set.
Table \ref{tab:wo_earlystop} reports the relative errors obtained by the
proposed ROMs---as well as the optimal projection error onto the manifold---at
the online-parameter instances. In this case, the
early-stopping strategy terminated iterations after 506 epochs, while the
strategy outlined above selected the network parameters arising at the 944th
epoch. These results
show that adopting the `no early stopping' strategy can slightly improve
performance over the early-stopping stratgy. We hypothesize this to be the
case because---for this example---the training data are quite informative of the
prediction task; thus, attempting to prevent overfitting essentially amounts
to underfitting.
\begin{table}[!h]
	\caption{\RO{\textit{1D Burgers' equation.} Relative errors of the ROM solutions
	(Eq.\ \eqref{eq:rel_err}) at online-parameter instances 
	$\paramtest^{1}$ and
	$\paramtest^{2}$.
	In this experiment, the offline training is performed both with and without the early
	stopping strategy.}}
	\label{tab:wo_earlystop}
	\begin{center}
	\begin{tabular}{|c |c |c| c | c| }
		\hline
	 	& \multicolumn{2}{c|}{$\paramtest^{1} = \left(4.3, 0.021\right)$} &  \multicolumn{2}{c|}{$\paramtest^{2} = \left(5.15, 0.0285\right)$}\\
		\hline
		Early stopping &  with & without & with & without\\
		\hline
		\LSPGNameCap		& $7.12\times 10^{-4}$ & $6.78\times 10^{-4}$ &
		$8.05\times 10^{-4}$	& $4.26\times 10^{-4}$\\
		\hline
		\GalerkinNameCap   		& $2.62\times 10^{-2}$ & $1.15\times 10^{-2}$ &
		$2.77\times 10^{-2}$	& $1.20\times 10^{-2}$\\
		\hline
		Optimal projection error 	&$4.20\times 10^{-4}$ & $2.87\times 10^{-4}$ &
		$3.98\times 10^{-4}$	& $2.81\times 10^{-4}$\\
		\hline
	\end{tabular}
	\end{center}
\end{table}
}

\newcommand*{\htensorsup}[1]{{\mathcal H}^{#1}}
\RO{
\section{Computational costs}
Let us consider the $i$th layer of the autoencoder, 
\begin{equation*}
\enclayer_{i}:(\gvec; \encparaml{i}) \mapsto \encactivation_i ( \encaffinetrans_i (\encparaml{i}, \gvec)).
\end{equation*}
\underline{Fully-connected layers}:
If the $i$th layer is a fully-connected layer, then $\encparaml{i} \in
\mathbb{R}^{\dofenclayer{i} \times (\dofenclayer{i-1}+1)}$ is a real-valued
matrix and the operation $\enclayer_{i}$ comprises a matrix--vector product
$\encaffinetrans_i (\encparaml{i}, \gvec) =\encparaml{i}[1, \gvec\tran]\tran$,
followed by an element-wise nonlinear activation $\encactivation_i$. Thus,
evaluating 
$\encactivation_i ( \encaffinetrans_i (\encparaml{i}, \gvec))$ incurs
$2\dofenclayer{i}(\dofenclayer{i-1}+1) + \activationcost\dofenclayer{i}$
floating-point operations (flops), where $\activationcost$ denotes the number of
flops incurred by applying the nonlinear activation to a
scalar. \\\\
\noindent \underline{Convolutional layers}:
If the $i$th layer is a (transposed-)convolutional layer, the operation $\enclayer_{i}$
comprises discrete convolutions, followed by an element-wise nonlinear
activation $\encactivation_i$. The input and the output of the $i$th
convolutional layer are 3-way tensors $\htensorsup{i-1} \in
\RR{\nchannel^{i-1} \times \nspaceDim{1}^{i-1} \times \nspaceDim{2}^{i-1}}$,
and $\htensorsup{i} \in \RR{\nchannel^{i} \times \nspaceDim{1}^{i} \times
\nspaceDim{2}^{i}}$ (see \ref{app:conv}). Each element of the output tensor is
computed by applying convolutional filters to a certain local region at every
channel of the input tensor, where a single application of the discrete
convolution with filter size $\kerlen_1^{i}\times\kerlen_2^{i}$ incurs
$2\kerlen_1^{i}\kerlen_2^{i}\nchannel^{i-1}$ flops. That is, computing
$\dofenclayer{i} (= \nchannel^{i}  \nspaceDim{1}^{i}  \nspaceDim{2}^{i})$
elements with convolutional filters of size $\kerlen_1^{i}\times\kerlen_2^{i}$
requires $2 \dofenclayer{i} \kerlen_1^{i}\kerlen_2^{i}\nchannel^{i-1}$ flops.
Thus,
evaluating $\encactivation_i ( \encaffinetrans_i (\encparaml{i}, \gvec))$ incurs $2 \dofenclayer{i} \kerlen_1^{i}\kerlen_2^{i}\nchannel^{i-1} + \activationcost\dofenclayer{i}$ flops and, again, $\activationcost$ denotes the number of
flops incurred by applying the nonlinear activation to a scalar.\\\\
\noindent \underline{Restrictor, prolongator, scaling and inverse scaling operator}: In this study, the restrictor and the prolongator are restricted to only change the shapes of the snapshot matrix into/from a tensor representing spatially distributed data and, thus, they do not incur floating-point operations. The scaling and the inverse scaling operators are applied to each FOM solution snapshot and incur $2N$ flops. \\\\
Figure \ref{fig:cost} reports the estimated computational costs of the decoder
with the latent code of dimension $\nstatered=5$ for varying choices of the
hyperparameters. On the left, the figure shows the computational costs
required at each layer of the decoder for varying convolutional filter sizes $\kerlen =
\{25,16,9,5,3\}$ with the number of the input channels and the number of convolutional filters specified as $\{64,32,16,8,1\}$, where the first element is the number of the input channels and the last four elements are the number of convolutional filters, and with strides $\{4,4,4,2\}$. On the right, the figure shows the total computational costs
for varying convolutional filter sizes $\kerlen = \{25,16,9,5,3\}$ and 
different numbers of input channels and convolutional filters $\{\{64,32,16,8,1\},
\{32,16,8,4,1\}, \{16,8,4,2,1\}, \{8,4,2,1,1\}, \{4,2,1,1,1\}, \{2,1,1,1,1\}\}$ and with strides $\{4,4,4,2\}$. The gray plane is the
computational cost for the linear decoder (e.g., POD) with the reduced dimension
$\nstatered=85$, which produces approximated solutions with the accuracy
comparable to the accuracy of the manifold ROM solutions with the autoencoder
configured as described in Table \ref{tab:autoenc_arch} with the reduced dimension $\nstatered=5$. The figure also reports the computational 
cost for the linear decoder (e.g., POD) with the latent dimension $\nstatered=5$ (the blue plane).
}
\begin{figure}[!h]
\centering
    \begin{subfigure}[b]{0.45 \linewidth}
    \centering
    \includegraphics[width=\linewidth]{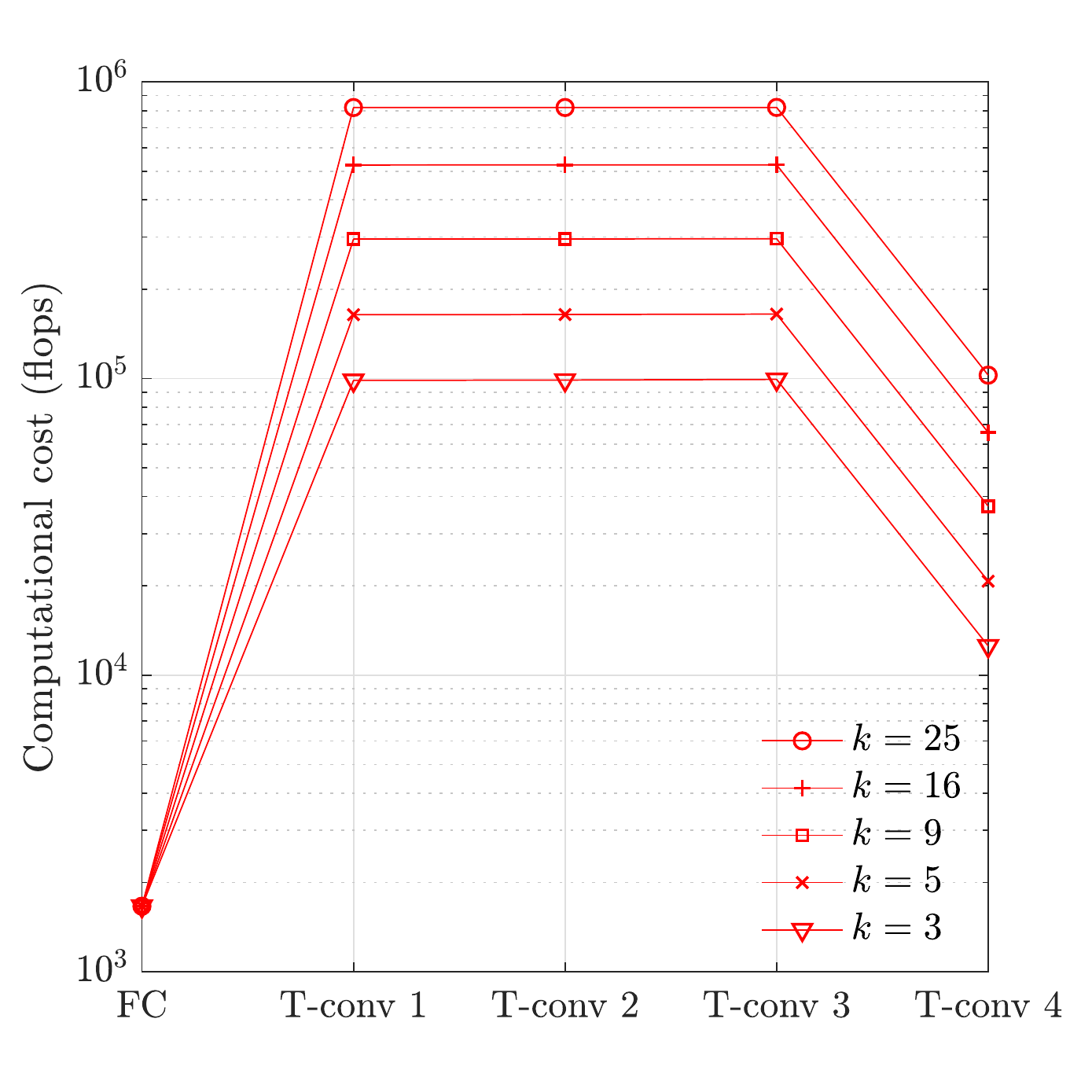}
    \caption{\RO{Estimated computational costs of the decoder at each layer}}
    \end{subfigure}
    \begin{subfigure}[b]{0.45 \linewidth}
    \centering
    \includegraphics[width=\linewidth]{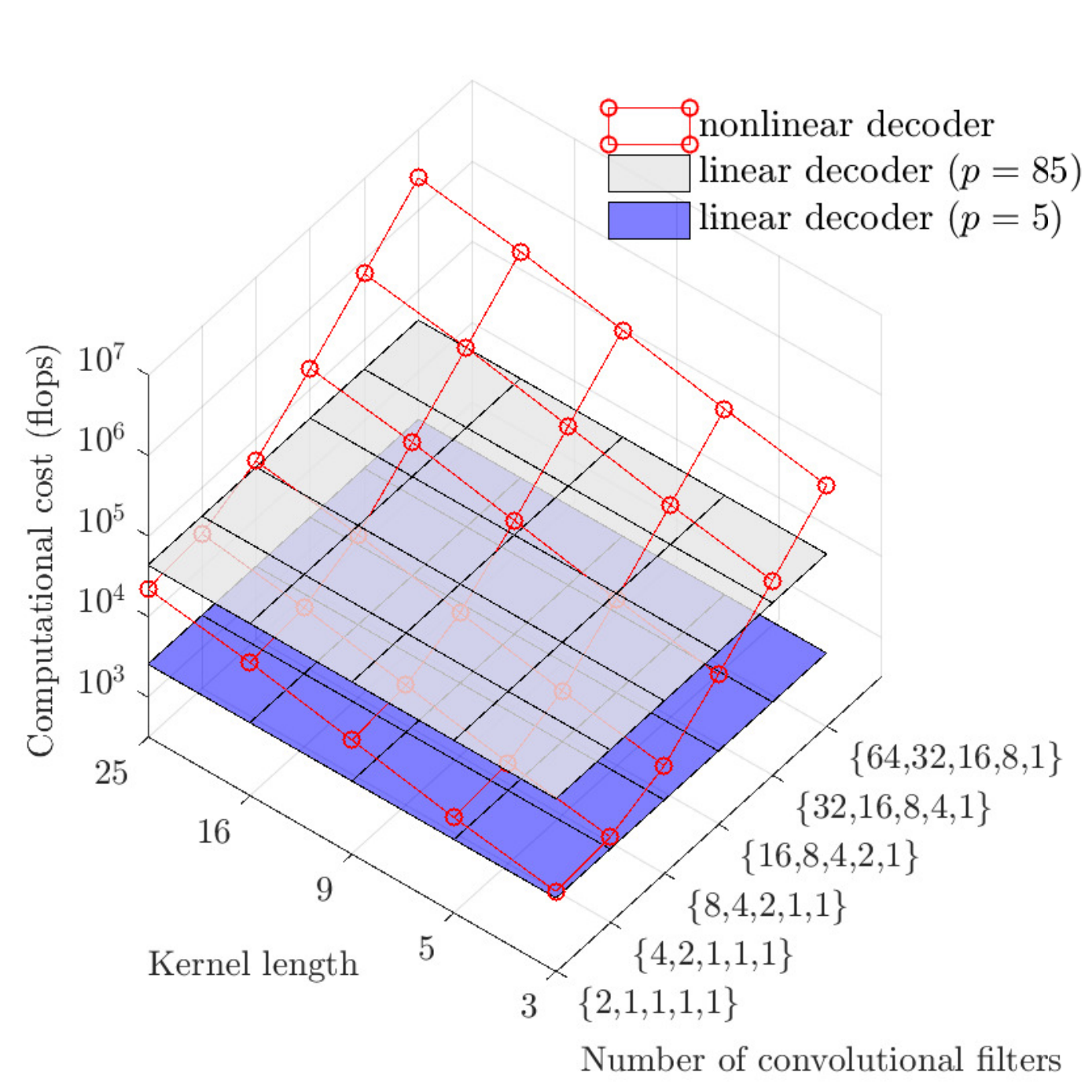}
    \caption{\RO{Estimated total computational costs of linear and nonlinear decoder}}
    \end{subfigure}
\caption{\RO{\textit{1D Burgers' equation: estimated computational costs of the
	decoder.} The computational costs
required at a fully-connected layer (FC) and transposed-convolutional layers (T-conv) of the decoder with the latent code of dimension $\nstatered=5$ for varying convolution filter sizes $\kerlen =
\{25,16,9,5,3\}$ with the number of convolutional filters $\{64,32,16,8,1\}$  (left) and the total computational costs
for varying convolutional filter sizes $\kerlen = \{25,16,9,5,3\}$ and
varying numbers of convolutional filters $\{\{64,32,16,8,1\},
\{32,16,8,4,1\}, \{16,8,4,2,1\}, \{8,4,2,1,1\}, \{4,2,1,1,1\}, \{2,1,1,1,1\}\}$ (right). The gray and blue planes are the
computational costs for the linear decoder (e.g., POD) with the reduced dimension
$\nstatered=85$ and $\nstatered=5$, respectively.}}
\label{fig:cost}
\end{figure}



\bibliography{ref}
\bibliographystyle{siam}

\end{document}

%% file: analysis.tex
\section{Analysis}\label{sec:analysis}
We now perform analysis of the proposed \GalerkinName\ and \LSPGName\
projection methods. For notational simplicity, this section omits the
dependence of operators on parameters $\param$.
\begin{proposition}[Affine trial manifold recovers classical Galerkin and
	LSPG]\label{thm:affine}
	If the trial manifold is affine,
	then \LSPGName\ projection is equivalent to classical linear-subspace LSPG projection. 
	If additionally the decoder mapping associates with an orthgonal matrix,
	then \GalerkinName\ projection is equivalent to classical linear-subspace Galerkin
	projection.
	\begin{proof}
	If the trial manifold is affine, then the decoder
		can be expressed as the linear mapping
		\begin{equation} \label{eq:linearDecoder}
			\rdbasisnl:\drdstate\mapsto\rdbasis
		\drdstate 
\end{equation} 
		with $\rdbasis\in\RR{\nstate\times\nstatered}$.\\
		\noindent \underline{\GalerkinNameCap}.
		Substituting the linear decoder \eqref{eq:linearDecoder} into the
		\GalerkinName\ ODE \eqref{eq:cont_rom} yields
\begin{align} \label{eq:cont_romGalOne}
	\difftime{\rdstate}(t\paramSemiProof) =
	\rdbasis\pseudoinv
	\velo\left(\stateRef\paramParProof + \rdbasis\rdstate(t\paramSemiProof),
	t\paramSemiProof
	\right),\quad \rdstate(0\paramSemiProof)=\zero.
\end{align}
		Similarly, substituting the linear decoder \eqref{eq:linearDecoder}
		into
the \GalerkinName\ O$\Delta$E Eq.~\eqref{eq:GalerkinODeltaE} yields the same
		system, but with the residual defined as
		\begin{equation} \label{eq:galModResidual}
			\rdres^n:\grdstate\gparamSemiProof \mapsto\rdbasis\pseudoinv \res^{n}
			(\stateRef\gparamParProof +
	 \rdbasis \grdstate^n\gparamSemiProof).
	  \end{equation} 
		If additionally the trial-basis matrix is orthogonal, i.e.,
		then 
		$\rdbasis\pseudoinv = \rdbasis\tran$ and the ODE \eqref{eq:cont_romGalOne}
		is equivalent to the
		classical Galerkin ODE
\begin{align} \label{eq:cont_romGal}
	\difftime{\rdstate}(t\paramSemiProof) =
	\rdbasis\tran
	\velo\left(\stateRef\paramParProof + \rdbasis\rdstate(t\paramSemiProof),
	t\paramSemiProof
	\right),\quad \rdstate(0\paramSemiProof)=\zero,
\end{align}
		while the O$\Delta$E residual \eqref{eq:galModResidual} is equivalent to the
		classical Galerkin O$\Delta$E residual
 \begin{equation} 
	 \rdres^n:\grdstate\gparamSemiProof \mapsto\rdbasis\tran \res^{n}
	 (\stateRef\gparamParProof +
	 \rdbasis \grdstate^n\gparamSemiProof).
	  \end{equation} 
		\noindent \underline{\LSPGNameCap}.
Substituting the linear decoder \eqref{eq:linearDecoder} into the \LSPGName\
		O$\Delta$E \eqref{eq:lspg_proj} yields the classical LSPG O$\Delta$E
\begin{align}\label{eq:lspg_projLinear}
	\rdtestbasis^{n}(\rdstate^{n}\paramSemiProof)\tran
	\res^{n}\left(\stateRef\paramParProof +
	\rdbasis\rdstate^n\paramSemiProof\right) = \zero
\end{align}
with classical LSPG test basis 
\begin{align*}
\rdtestbasis^{n}:\grdstate\gparamSemiProof &\mapsto 
	\left( \alpha_0
	\identity - \Delta t \beta_0 \frac{\partial \velo}{\partial \dstate}
	\left( \stateRef\gparamParProof + \rdbasis\grdstate, t^n\gparamParProof \right) \right)
	\rdbasis.
\end{align*}

	\end{proof}
\end{proposition}

The \GalerkinName\ O$\Delta$E
\eqref{eq:GalerkinODeltaE}--\eqref{eq:time_cont_res} was derived using a
project-then-discretize approach, as we formulated the residual-minimization
problem \eqref{eq:cont_opt_problem} at the time-continuous level, and
subsequently discretized the equivalent ODE \eqref{eq:cont_rom} in time. We
now derive conditions under which an equivalent model could have been derived
using a discretize-then-project approach, i.e., by substituting the
approximated state $\states \leftarrow \aprxstate$ defined in
Eq.~\eqref{eq:state_approx} into the FOM O$\Delta$E \eqref{eq:fomODeltaE} and
subsequently premultiplying the overdetermined system by a test basis.

\begin{theorem}[Commutativity of time discretization and \GalerkinName\
	projection]\label{prop:comm}
For \GalerkinName\ projection, time discretization and projection are commutative if 
	either (1) the 
	trial manifold corresponds to an affine subspace, or (2) 
	the nonlinear trial manifold is
	twice continuously differentiable;
	$\|\rdstate^{n-j} - \rdstate^{n}\| = \bigO(\Delta t)$ for all
	$n\innatno{\ntime}$, $j\innatno{k}$; and
	the limit
	$\Delta t\rightarrow 0$ is taken.
\end{theorem}
\begin{proof}
	 The discrete residual
characterizing	
	the \GalerkinName\ O$\Delta$E \eqref{eq:GalerkinODeltaE}, which
	was derived using a project-then-discretize approach,
	 is defined as $\rdres^n$ in Eq.~\eqref{eq:time_cont_res}.\\
\noindent Substituting the approximated state $\states \leftarrow \aprxstate$ defined in
	Eq.~\eqref{eq:state_approx} into the FOM O$\Delta$E \eqref{eq:fomODeltaE}
	yields
$\res^{n}(\stateRef+\rdbasisnl(\rdstate^n)\paramSemiProof) =
	\zero$.
	This is an overdetermined system of equations and thus may not have a
	solution. As such, we perform projection by enforcing orthgonality of this
	residual to a test basis $\testBasisGal:\RR{\nstatered}\rightarrow\RR{\nstate\times\nstatered}$ such
	that the discretize-then-project \GalerkinName\ O$\Delta$E corresponds to 
	$\rdresDisc^n(\rdstate^n)=\zero$ with 
 \begin{equation} 
	 \rdresDisc^n:\grdstate\mapsto\testBasisGal(\grdstate)^T\res^{n}(\stateRef+\rdbasisnl(\grdstate)\paramSemiProof).
	  \end{equation} 
	The discretize-then-project \GalerkinName\ ROM is equivalent to the proposed
	\GalerkinName\ ROM if and only if there exists a full-rank matrix
	$\constantSymbGal:\RR{\nstatered}\timesparamspaceProof\rightarrow\RRstar{\nstatered\times\nstatered}$
	such that
	\begin{equation}\label{eq:matchingGal}
\rdres^n(\grdstate\gparamSemiProof)=\constantAnalysisGal\rdresDisc^n(\grdstate\gparamSemiProof).
	\end{equation}
	\noindent \textit{Case 1.} 	We now consider the general (non-asymptotic)
	case.  
	Eq.~\eqref{eq:matchingGal} holds for any sequence of approximate
	solutions $\rdstate^n$, $n=1,\ldots,\ntime$ if and only if
	each term in this expansion matches, i.e.,
		\begin{align}\label{eq:requirementOneGal}
		\begin{split}
			\grdstate=&
			\constantAnalysisGal\testBasisTransposeAnalysisGal\rdbasisnl(\grdstate)\\
 \jacrdbasisnl(\grdstate)\pseudoinv \velo(
			\stateRef+\rdbasisnl(\grdstate),t^n\gparamSemiProof)=&
			\constantAnalysisGal\testBasisTransposeAnalysisGal\velo(
	 \stateRef+\rdbasisnl(\grdstate),t^n\gparamSemiProof)\\
			\rdstate^{n-j}=&\constantAnalysisGal\testBasisTransposeAnalysisGal\rdbasisnl(\rdstate^{n-j})\\
\jacrdbasisnl(\rdstate^{n-j})\pseudoinv \velo(
			\stateRef+\rdbasisnl(\rdstate^{n-j}),t^{n-j}\gparamSemiProof)=&
			\constantAnalysisGal\testBasisTransposeAnalysisGal \velo (
	 \stateRef+\rdbasisnl(\rdstate^{n-j}), t^{n-j}\gparamSemiProof),
			\end{split}
		\end{align}
where we have used $\sum_{j=0}^k\alpha_j=0$.
	We first consider the first and third conditions of
	\eqref{eq:requirementOneGal}.
	Because the inverse of a nonlinear operator---if it
		exists---must also be nonlinear, a necessary condition for these two
		requirements is that the trial manifold corresponds is affine
		such that the decoder satisfies \eqref{eq:linearDecoder}.
		Substituting Eq.~\eqref{eq:linearDecoder} into the first and third of
		the conditions of \eqref{eq:requirementOneGal} yields
		\begin{align}\label{eq:requirementTwoGal}
		\begin{split}
			\grdstate=&
			\constantAnalysisGal\testBasisTransposeAnalysisGal\rdbasis \grdstate \\
			\rdstate^{n-j}=&\constantAnalysisGal\testBasisTransposeAnalysisGal\rdbasis \rdstate^{n-j}
			\end{split}
		\end{align}
		for $j=1,\ldots,k$. A necessary and sufficient conditions for these to
		hold is
\begin{equation} \label{eq:constantDefGal}
	\constantAnalysisGal= \constantDefAnalysisGal.
\end{equation} 
We now consider the second and fourth conditions of
\eqref{eq:requirementOneGal}.
Substituting Eqs.~\eqref{eq:linearDecoder} and \eqref{eq:constantDefGal} into
these conditions yields
		\begin{align}\label{eq:requirementOneLinearGal}
		\begin{split}
 \rdbasis\pseudoinv \velo(\stateRef\gparamParProof +
			\rdbasis\grdstate,t^n\gparamSemiProof)=&
			\constantDefAnalysisGal\testBasisTransposeAnalysisGal\velo(\stateRef\gparamParProof +
	 \rdbasis\grdstate,t^n\gparamSemiProof)\\
\rdbasis\pseudoinv \velo(\stateRef\gparamParProof +
			\rdbasis\rdstate^{n-j},t^{n-j}\gparamSemiProof)=&
			\constantDefAnalysisGal\testBasisTransposeAnalysisGal \velo (\stateRef\gparamParProof +
	 \rdbasis\rdstate^{n-j}, t^{n-j}\gparamSemiProof)
			\end{split}
		\end{align}
		for $j=1,\ldots,k$. Noting that $\rdbasis\pseudoinv =
		[\rdbasis\tran\rdbasis]^{-1}\rdbasis\tran$, a necessary and sufficient
		condition for \eqref{eq:requirementOneLinearGal} to hold is 
		$
\rdbasis\pseudoinv = 
\constantDefAnalysisGal\testBasisTransposeAnalysisGal.
$
Because $\rdbasis$ has no dependence on the generalized state, the test basis
must also be independent of the generalized state, which yields
\begin{equation}\label{eq:finalEquivCondition}
	\rdbasis\pseudoinv =
	[\testBasisAnalysisGalNo^T\rdbasis]^{-1}\testBasisAnalysisGalNo^T.
\end{equation}
Thus, without using any asymptotic arguments, the proposed \GalerkinName\ ROM
could have been derived using a discretize-then-project approach if and only
if the test basis $\testBasisAnalysisGalNo$ applied to the overdetermined FOM
O$\Delta$E acting on the trial manifold satisfies
Eq.~\eqref{eq:finalEquivCondition}.
Examples of test bases that satisfy this condition include
$\testBasisAnalysisGalNo = [\rdbasis\pseudoinv]\tran$ and
$\testBasisAnalysisGalNo = \scalarConstant\rdbasis$ for some nonzero scalar
$\scalarConstant\in\RR{}$.
\\
\textit{Case 2.} We now consider asymptotic arguments.
	Assuming the nonlinear trial manifold is
	twice continuously differentiable, and
	$\|\rdstate^{n-j} - \rdstate^{n}\| = \bigO(\Delta t)$, then for $\grdstate$
	in a neighborhood of $\rdstate^n$ such that 
	$\|\grdstate - \rdstate^{n}\| = \bigO(\Delta t)$, we have
\begin{align} \label{eq:limiting}
	\begin{split}
		\jacrdbasisnl(\rdstate^{n-j})\pseudoinv &=
		\jacrdbasisnl(\rdstate^{n})\pseudoinv + \bigO(\Delta t)\\
	\jacrdbasisnl(\grdstate)\pseudoinv &=
		\jacrdbasisnl(\rdstate^{n})\pseudoinv + \bigO(\Delta t)\\
	\rdbasisnl(\rdstate^{n-j}) &= \rdbasisnl(\rdstate^{n}) + \jacrdbasisnl(\rdstate^{n})[
\rdstate^{n-j}-\rdstate^{n}
	]
	+ \bigO(\Delta t)\\
	\rdbasisnl(\grdstate) &= \rdbasisnl(\rdstate^{n}) + \jacrdbasisnl(\rdstate^{n})[
\grdstate-\rdstate^{n}
	]
	+ \bigO(\Delta t).
	\end{split}
\end{align} 
Substituting these expressions into the definition of the
discretize-then-project \GalerkinName\ residual $\rdresDisc^n$ and enforcing
equivalence of each term in  the matching conditions
\eqref{eq:matchingGal} and taking the limit $\Delta t\rightarrow 0$ yields
		\begin{align}\label{eq:requirementOneGalAsymp}
		\begin{split}
			\grdstate=&
			\constantAnalysisGal\testBasisTransposeAnalysisGal
			\jacrdbasisnl(\rdstate^{n})
			\grdstate\\
 \jacrdbasisnl(\rdstate^n)\pseudoinv \velo(
			\stateRef+\rdbasisnl(\grdstate),t^n\gparamSemiProof)=&
			\constantAnalysisGal\testBasisTransposeAnalysisGal\velo(
	 \stateRef+\rdbasisnl(\grdstate),t^n\gparamSemiProof)\\
			\rdstate^{n-j}=&\constantAnalysisGal\testBasisTransposeAnalysisGal
			\jacrdbasisnl(\rdstate^{n})
			\rdstate^{n-j}\\
\jacrdbasisnl(\rdstate^{n})\pseudoinv \velo(
			\stateRef+\rdbasisnl(\rdstate^{n-j}),t^{n-j}\gparamSemiProof)=&
			\constantAnalysisGal\testBasisTransposeAnalysisGal \velo (
	 \stateRef+\rdbasisnl(\rdstate^{n-j}), t^{n-j}\gparamSemiProof),
			\end{split}
		\end{align}
		where we have used $\sum_{j=0}^k\alpha_j=0$.
A necessary and sufficient condition for the first and
third conditions to hold is
\begin{equation} 
\constantAnalysisGal = 
	[\testBasisTransposeAnalysisGal
			\jacrdbasisnl(\rdstate^{n})]^{-1}.
	\end{equation} 
	Then, a necessary and sufficient condition for the second and fourth
	conditions to hold is
	 \begin{equation} \label{eq:finalEquivConditionAsy}
\jacrdbasisnl(\rdstate^n)\pseudoinv = 
[\testBasisTransposeAnalysisGal
			\jacrdbasisnl(\rdstate^{n})]^{-1}
\testBasisTransposeAnalysisGal.
		  \end{equation} 
Examples of test bases that satisfy Eq.~\eqref{eq:finalEquivConditionAsy}
include 
$\testBasisAnalysisGal = [\jacrdbasisnl(\rdstate^{n})\pseudoinv]\tran$ and 
$\testBasisAnalysisGalNo = \scalarConstant\jacrdbasisnl(\rdstate^{n})$ for
some nonzero scalar
		$\scalarConstant\in\RR{}$.
\end{proof}
Theorem \ref{prop:comm} shows that \GalerkinName\ can be derived using a
discretize-then-project approach for nonlinear trial manifolds. This shows
that previous analysis \cite[Theorem 3.4]{carlbergGalDiscOpt} related to
commutativity of (residual-minimizing) Galerkin projection and time
discretization extends to the case of nonlinear trial manifolds.

We now show that the limiting equivalence result reported in \cite[Section
5]{carlbergGalDiscOpt} between Galerkin and LSPG projection for linear
subspaces also extends to nonlinear trial manifolds.

\begin{theorem}[Limiting equivalence]\label{thm:equivalence}
	\GalerkinNameCap\ and \LSPGName\ are equivalent if 
	either (1) the trial
	manifold corresponds to an affine subspace
	and either an explicit scheme is
	employed or the limit $\Delta t\rightarrow 0$ is taken, or (2) the
	nonlinear manifold is twice continuously differentiable; $\|\rdstate^{n-j} -
	\rdstate^{n}\| = \bigO(\Delta t)$ for all $n\innatno{\ntime}$,
	$j\innatno{k}$; and the limit $\Delta t\rightarrow 0$ is taken.
\end{theorem}
\begin{proof}
	The proof follows similar steps to those applied in the proof of Theorem
	\ref{prop:comm}.
	The discrete residual $\rdresLSPGn$ defined in \eqref{eq:lspg_projsimple},
	which characterizes the \LSPGName\ O$\Delta$E \eqref{eq:lspg_proj}, can be
	written as
	\begin{align} \label{eq:LSPGODeltaEResFull}
 \begin{split} 
	 \rdresLSPGn:\grdstate\gparamSemiProof\mapsto&
	\alpha_0\rdtestbasis^{n}(\grdstate\gparamSemiProof)\tran\rdbasisnl(\grdstate) - \Delta t \beta_0
	\rdtestbasis^{n}(\grdstate\gparamSemiProof)\tran\velo(
	 \stateRef+\rdbasisnl(\grdstate),t^n\gparamSemiProof) + \\
	 & \sum_{j=1}^{k} \alpha_j 
	 \rdtestbasis^{n}(\grdstate\gparamSemiProof)\tran\rdbasisnl(\rdstate^{n-j})
	 - \Delta t \sum_{j=1}^{k}\beta_j \rdtestbasis^{n}(\grdstate\gparamSemiProof)\tran \velo (
	 \stateRef+\rdbasisnl(\rdstate^{n-j}), t^{n-j}\gparamSemiProof)
 \end{split} 
	  \end{align} 
		where we have used $\sum_{i=0}^k\alpha_i=0$. \LSPGNameCap\ projection is
		equivalent to \GalerkinName\ projection if and only if the O$\Delta$Es
		characterizing the two methods are equivalent, which is equivalent
		to requiring the existence of a full-rank matrix
		$\constantSymb:\RR{\nstatered}\timesparamspaceProof\rightarrow\RRstar{\nstatered\times\nstatered}$ such that
		\begin{equation}\label{eq:conditionsMatchL}
		\rdresn(\grdstate\gparamSemiProof)=\constantAnalysis\rdresLSPGn(\grdstate\gparamSemiProof).
		\end{equation}
			
		\noindent\textit{Case 1.} We now consider the general (non-asymptotic)
		case. Eq.~\eqref{eq:conditionsMatchL} holds for any sequence of
		approximate solutions 
		$\rdstate^n$, $n=1,\ldots,\ntime$
if and only if each term in this expansion matches, i.e.,
		\begin{align}\label{eq:requirementOne}
		\begin{split}
			\grdstate=&
			\constantAnalysis\testBasisTransposeAnalysis\rdbasisnl(\grdstate)\\
 \jacrdbasisnl(\grdstate)\pseudoinv \velo(
			\stateRef+\rdbasisnl(\grdstate),t^n\gparamSemiProof)=& \constantAnalysis\testBasisTransposeAnalysis\velo(
	 \stateRef+\rdbasisnl(\grdstate),t^n\gparamSemiProof)\\
			\rdstate^{n-j}=&\constantAnalysis\testBasisTransposeAnalysis\rdbasisnl(\rdstate^{n-j})\\
\jacrdbasisnl(\rdstate^{n-j})\pseudoinv \velo(
			\stateRef+\rdbasisnl(\rdstate^{n-j}),t^{n-j}\gparamSemiProof)=&
			\constantAnalysis\testBasisTransposeAnalysis \velo (
	 \stateRef+\rdbasisnl(\rdstate^{n-j}), t^{n-j}\gparamSemiProof)
			\end{split}
		\end{align}
		for $j=1,\ldots,k$.
		Comparing Conditions \eqref{eq:requirementOne} with
		\eqref{eq:requirementOneGal}, we see that the requirements are the same as
		in Theorem \ref{prop:comm}, but with 
		$\testBasis$ replacing $\testBasisGal$ and  $\constantSymb$ replacing
		$\constantSymbGal$. Thus, we arrive at the same result: necessary and
		sufficient conditions for equivalence are that
the trial manifold corresponds to an affine
		subspace such that the decoder satisfies \eqref{eq:linearDecoder}
		and  the test basis $\testBasisAnalysis$ satisfies
\begin{equation} \label{eq:lspgGenCond}
\rdbasis\pseudoinv = \constantDefAnalysis\testBasisTransposeAnalysis.
\end{equation} 
Note that necessary conditions for Eq.~\eqref{eq:lspgGenCond} to hold are
$\range(\testBasisAnalysis) = \range(\rdbasis)$.
Examples of test basis that satisfy this condition include
$\testBasisAnalysis = [\rdbasis\pseudoinv]\tran$ and 
$\testBasisAnalysis = \scalarConstant\rdbasis$ for some nonzero scalar
		$\scalarConstant\in\RR{}$.
However, unlike in Theorem \ref{prop:comm}, we are not free to choose the
test basis; rather, it is defined as 
 \begin{equation} \label{eq:lspgTestOverall}
\rdtestbasis^{n}:\grdstate\mapsto\testBasisAnalysisLongLinear.
	  \end{equation}
		For the necessary condition
$\range(\testBasisAnalysis) = \range(\rdbasis)$ to hold, the second term in 
Eq.~\eqref{eq:lspgTestOverall} must be zero. This occurs if and only if either
(1) an explicit scheme is
	employed such that $\beta_0=0$, or (2) the limit $\Delta t\rightarrow 0$ is
	taken. In both cases, 
	$
	\testBasisAnalysis=\scalarConstant\rdbasis$
	holds with
		$\scalarConstant=\alpha_0$.\\
		\noindent
\noindent\textit{Case 2.} We now consider asymptotic arguments.
Assuming the nonlinear trial manifold is
	twice continuously differentiable, and
	$\|\rdstate^{n-j} - \rdstate^{n}\| = \bigO(\Delta t)$, then for $\grdstate$
	in a neighborhood of $\rdstate^n$ such that 
	$\|\grdstate - \rdstate^{n}\| = \bigO(\Delta t)$, we obtain the expressions
	\eqref{eq:limiting}.
Substituting these expressions into the definition of the $\rdresLSPGn$ in
Eq.~\eqref{eq:LSPGODeltaEResFull},
enforcing
equivalence of each term in  the matching conditions
\eqref{eq:conditionsMatchL}, and taking the limit $\Delta t\rightarrow 0$ yields
		\begin{align}\label{eq:requirementOneAsymp}
		\begin{split}
			\grdstate=&
			\constantAnalysis\testBasisTransposeAnalysis
			\jacrdbasisnl(\rdstate^{n})
			\grdstate\\
 \jacrdbasisnl(\rdstate^n)\pseudoinv \velo(
			\stateRef+\rdbasisnl(\grdstate),t^n\gparamSemiProof)=&
			\constantAnalysis\testBasisTransposeAnalysis\velo(
	 \stateRef+\rdbasisnl(\grdstate),t^n\gparamSemiProof)\\
			\rdstate^{n-j}=&\constantAnalysis\testBasisTransposeAnalysis
			\jacrdbasisnl(\rdstate^{n})
			\rdstate^{n-j}\\
\jacrdbasisnl(\rdstate^{n})\pseudoinv \velo(
			\stateRef+\rdbasisnl(\rdstate^{n-j}),t^{n-j}\gparamSemiProof)=&
			\constantAnalysis\testBasisTransposeAnalysis \velo (
	 \stateRef+\rdbasisnl(\rdstate^{n-j}), t^{n-j}\gparamSemiProof),
			\end{split}
		\end{align}
		where we have used $\sum_{j=0}^k\alpha_j=0$.  
	A necessary and sufficient condition for the first and
	third conditions to hold is 	
 \begin{equation} 
\constantAnalysis = 
	 [\testBasisTransposeAnalysis\jacrdbasisnl(\rdstate^{n})]^{-1}.
	  \end{equation} 
		Then, 
	a necessary and sufficient condition for the second and fourth
	conditions to hold is
	 \begin{equation} \label{eq:finalEquivConditionAsyL}
\jacrdbasisnl(\rdstate^n)\pseudoinv = 
[\testBasisTransposeAnalysis
			\jacrdbasisnl(\rdstate^{n})]^{-1}
\testBasisTransposeAnalysis.
		  \end{equation} 
		As in Case 1 above, a necessary condition for \eqref{eq:finalEquivConditionAsyL}
		to hold is that $\range(\testBasisAnalysis) =
		\range(\jacrdbasisnl(\rdstate^n))$. Due to the definition of the test
		basis, this requires the second term on the right-hand side of Eq.~\eqref{eq:testBasisMatDef}
		to be zero. This is already satisfied by the stated assumption that the
		limit $\Delta t\rightarrow 0$ is taken, in which case $
		\testBasisAnalysis=\scalarConstant\jacrdbasisnl (\grdstate)$ holds with
		$\scalarConstant=\alpha_0$. It can be easily verified that this test basis
		satisfies the condition \eqref{eq:finalEquivConditionAsyL}.
\end{proof}

%% file: error_analysis.tex
We now derive \textit{a posteriori} local discrete-time error bounds for both
the manifold Galerkin and manifold LSPG projection methods in the context of
linear multistep methods, and demonstrate that manifold LSPG projection
sequentially minimizes the error bound in time. For notational simplicity, we
suppress the second argument of the velocity $\velo$, as the time superscript
appearing in the second argument always matches that of the first argument.

\begin{theorem}[Error bound]\label{thm:error_bound}
	If the velocity $\velo$ is Lipschitz continuous, i.e.,
	there exists a constant $\lipschitz>0$ such that
	$\norm{\velo(\bm{x})-\velo(\bm{y})}\leq\lipschitz\norm{\bm{x}-\bm{y}}$
	for all $\bm{x},\,\bm{y}\in\RR{\nstate}$, and
	the time step is sufficiently small such that
	$ \Delta t < {\abs{\alpha_0}}/{\abs{\beta_0}\lipschitz}$,
	then 
\begin{align} 
	\label{eq:boundGal}&\norm{\stateFOMArg{n} \negstateROMGalArg{n}}\leq\frac{1}{\hconstant}
	\norm{\resnROMArg{\stateROMGalArg{n}}} +
\sum_{j=1}^{k}\gammaconstantj\norm{
\stateFOMArg{n-j}\negstateROMGalArg{n-j}
		}\\
	\label{eq:boundLSPG}&\norm{\stateFOMArg{n} \negstateROMLSPGArg{n}}\leq\frac{1}{\hconstant}
	\min_{\dstate\in\stateRef+ \manifold}\norm{\resnROMArg{\dstate}} +
\sum_{j=1}^{k}\gammaconstantj\norm{
\stateFOMArg{n-j}\negstateROMLSPGArg{n-j}
		},
\end{align} 
	where $\rdstateGal^n$ denotes the manifold Galerkin solution satisfying 
	\eqref{eq:GalerkinODeltaE} and $\rdstateLSPG^n$
	denotes the manifold LSPG solution satisfying \eqref{eq:disc_opt_problem}.
	Here, $\hconstant\defeq\abs{\alpha_0} - \abs{\beta_0}\lipschitz\Delta t$
and $\gammaconstantj\defeq\left(\abs{\alpha_j} + 
	\abs{\beta_j}\lipschitz
	\Delta
	t\right)/\hconstant$.
\end{theorem}
\begin{proof}
	We begin by defining linear multistep residual operators associated with the
	FOM sequence of solutions $\{\stateFOMArg{j}\}_{j=1}^n$ and an (arbitrary)
	sequence of approximate solutions $\{\stateROMArg{j}\}_{j=1}^n$, i.e.,
	\begin{align}
		\resnFOM:\dstate&\mapsto\alpha_0 \dstate - \Delta t \beta_0
		\velo(\dstate) + \sum_{j=1}^{k} \alpha_j \stateFOMArg{n-j}
	- \Delta t \sum_{j=1}^{k} \beta_j \velo (\stateFOMArg{n-j})\\
		\resnROM:\dstate&\mapsto
		\alpha_0 \dstate - \Delta t \beta_0
		\velo(\dstate) + \sum_{j=1}^{k} \alpha_j \stateROMArg{n-j}
	- \Delta t \sum_{j=1}^{k} \beta_j \velo (\stateROMArg{n-j}).
	\end{align}
We note that $\stateROMArg{j} = \stateROMGalArg{j}$, $j=1,\ldots,n$ in the case of manifold
Galerkin projection, and 
$\stateROMArg{j} = \stateROMLSPGArg{j}$, $j=1,\ldots,n$ in the case of manifold
LSPG projection.
Subtracting  $\resnROMArg{\stateROMArg{n}}$ from $\resnFOMArg{\stateFOMArg{n}}$ and
noting from 
Eq.~\eqref{eq:fomODeltaE} that $\resnFOMArg{\stateFOMArg{n}} =
\zero$ yields
\begin{align*} 
\begin{split} 
	-\resnROMArg{\stateROMArg{n}} &= \alpha_0\left(\stateFOMArg{n} -
	\stateROMArg{n}\right) - \Delta
	t\beta_0\left(\velo(\stateFOMArg{n})-\velo(\stateROMArg{n})\right)
	 + \sum_{j=1}^{k} \alpha_j
	 \left(\stateFOMArg{n-j}-\stateROMArg{n-j}\right)
	- \Delta t \sum_{j=1}^{k} \beta_j \left(\velo (\stateFOMArg{n-j})
	-\velo (\stateROMArg{n-j})\right).
\end{split} 
\end{align*} 
Rearranging this expression gives
\begin{align} \label{eq:IandII}
\begin{split} 
	&\underbrace{\stateFOMArg{n} - \stateROMArg{n} - 
	\frac{\Delta
	t\beta_0}{\alpha_0}\left(\velo(\stateFOMArg{n})-\velo(\stateROMArg{n})\right)}_\text{(I)}
	 =\\
	&\qquad\underbrace{-\frac{1}{\alpha_0}\resnROMArg{\stateROMArg{n}} -  \frac{1}{\alpha_0}\sum_{j=1}^{k} \alpha_j
	 \left(\stateFOMArg{n-j}-\stateROMArg{n-j}\right)
	+ \frac{\Delta t}{\alpha_0} \sum_{j=1}^{k} \beta_j \left(\velo (\stateFOMArg{n-j})
	-\velo (\stateROMArg{n-j})\right)}_\text{(II)}.
\end{split} 
\end{align} 
We proceed by bounding $\norm{\text{(I)}}$ from below, and $\norm{\text{(II)}}$ from above.
To bound $\norm{\text{(I)}}$ from below, we apply the reverse triangle to obtain
 \begin{equation} \label{eq:Iboundone}
	 \norm{\text{(I)}}\geq
	 \abs{
	 \norm{\stateFOMArg{n} - \stateROMArg{n}} - 
	 \norm{\underbrace{
\frac{\Delta
	 t\beta_0}{\alpha_0}\left(\velo(\stateFOMArg{n})-\velo(\stateROMArg{n})\right)
		 }_\text{(I.a)}}
	 }.
 \end{equation}
 We now use Lipschitz continuity of the velocity $\velo$ to bound
 $\norm{\text{(I.a)}}$ from above as
\begin{equation} \label{eq:Iboundtwo}
\norm{
	\text{(I.a)}
}\leq \frac{\Delta t
		 \abs{\beta_0}\lipschitz}{\abs{\alpha_0}}\norm{\stateFOMArg{n} -
		 \stateROMArg{n}}.
\end{equation} 
If time-step-restriction condition $ \Delta t < {\abs{\alpha_0}}/{\abs{\beta_0}\lipschitz}$ holds,
then we can combine inequalities \eqref{eq:Iboundone} and \eqref{eq:Iboundtwo}
as
\begin{equation} \label{eq:Ilowerbound}
\norm{\text{(I)}}\geq
	 \frac{\hconstant}{\abs{\alpha_0}}\norm{\stateFOMArg{n} - \stateROMArg{n}}.
\end{equation} 

To bound $\norm{\text{(II)}}$ in Eq.~\eqref{eq:IandII} from above, we apply the triangle
inequality and employ Lipschitz continuity of the velocity $\velo$, which
gives
\begin{equation} \label{eq:IIupperbound}
	\norm{
		\text{(II)}
		}\leq
	\frac{1}{\abs{\alpha_0}}\norm{\resnROMArg{\stateROMArg{n}}}
	+ \frac{1}{\abs{\alpha_0}}\sum_{j=1}^{k}\left(\abs{\alpha_j} + 
	\abs{\beta_j}\lipschitz
	\Delta
	t\right)\norm{
\stateFOMArg{n-j}-\stateROMArg{n-j}
		}.
\end{equation} 
Combining Eq.~\eqref{eq:IandII} with inequalities \eqref{eq:Ilowerbound} and
\eqref{eq:IIupperbound} gives
\begin{equation} \label{eq:finalGenResult}
	\norm{\stateFOMArg{n} - \stateROMArg{n}}\leq\frac{1}{\hconstant}
	\norm{\resnROMArg{\stateROMArg{n}}} +
\sum_{j=1}^{k}\gammaconstantj\norm{
\stateFOMArg{n-j}-\stateROMArg{n-j}
		}.
\end{equation} 
The manifold-Galerkin error bound \eqref{eq:boundGal} results from substituting
$\stateROMArg{j} = \stateROMGalArg{j}$, $j=1,\ldots,n$ in \eqref{eq:finalGenResult}, while
the manifold-LSPG error bound \eqref{eq:boundLSPG} results from substituting
$\stateROMArg{j} = \stateROMLSPGArg{j}$, $j=1,\ldots,n$ in \eqref{eq:finalGenResult}
and noting that the manifold LSPG solution $\stateROMLSPGArg{j}$ satisifies the
time-discrete residual-minimization property \eqref{eq:disc_opt_problem}.
\end{proof}

This result illustrates that the time-discrete residual minimization property
of manifold LSPG projection allows its approximation to sequentially minimize
the error bound, as the first term on the right-hand side of bound \eqref{eq:boundLSPG}
corresponds to the new contribution to the error bound at time instance $t^n$,
while the second term on the right-hand side comprises the recursive term in
the bound.

%% file: drawbacks.tex
We note that one drawback of the method in its current form (relative to
classical linear-subspace methods) is that it incurs a costlier training
process, as training a deep convolutional autoencoder is significantly more
computationally expensive than simply computing the left singular vectors of a
snapshot matrix. In addition, the proposed method is characterized by
significantly more hyperparameters---which relate to the autoencoder
architecture---than classical methods.